\documentclass[10pt,oneside,shortlabels]{amsart}
\usepackage{geometry}
\usepackage{amsbsy}
\usepackage{amstext}
\usepackage{amssymb}
\usepackage{esint}

\makeatletter
\usepackage{amsthm}
\usepackage{enumitem}		
 \newcommand\thmsname{Theorem}
 \newcommand\nm@thmtype{thm}
 \theoremstyle{plain}
 
 \newenvironment{namedthm}[1]{
   \renewcommand\thmsname{#1}\renewcommand\nm@thmtype{namedtheorem}
   \begin{\nm@thmtype}
}
   {\end{\nm@thmtype}
}

\usepackage{amsmath}
\usepackage{amsthm}
\usepackage{accents}

\theoremstyle{plain}

\numberwithin{equation}{section}

\newtheorem{thm}{Theorem}[section]
\newtheorem*{thm*}{Theorem}
\newtheorem{cor}[thm]{Corollary}
\newtheorem*{cor*}{Corollary}
\newtheorem{lem}[thm]{Lemma}
\newtheorem*{lem*}{Lemma}
\newtheorem{prop}[thm]{Proposition}
\newtheorem*{prop*}{Proposition}

\newtheorem*{conjecture*}{Conjecture}

\newtheorem*{fact*}{Fact}

\newtheorem*{criterion*}{Criterion}

\newtheorem*{algorithm*}{Algorithm}

\newtheorem*{ax*}{Axiom}

\newtheorem*{assumption*}{Assumption}

\newtheorem*{question*}{Question}

\theoremstyle{remark}

\newtheorem{rem}[thm]{Remark}
\newtheorem*{rem*}{Remark}

\newtheorem*{rems*}{Remarks}

\newtheorem*{claim*}{Claim}

\newtheorem*{exercise*}{Exercise}

\newtheorem*{note*}{Note}
\newtheorem{notation}[thm]{Notation}
\newtheorem*{notation*}{Notation}

\newtheorem*{summary*}{Summary}

\newtheorem*{acknowledgement*}{Acknowledgement}

\newtheorem*{conclusion*}{Conclusion}

\theoremstyle{definition}

\newtheorem{defn}[thm]{Definition}
\newtheorem*{defn*}{Definition}
\newtheorem{example}[thm]{Example}
\newtheorem*{example*}{Example}

\newtheorem*{examples*}{Examples}

\newtheorem*{problem*}{Problem}

\newtheorem*{xca*}{Exercise}

\newtheorem*{xcas*}{Exercises}

\newtheorem*{condition*}{Condition}

\def\sideremark#1{\ifvmode\leavevmode\fi\vadjust{\vbox to0pt{\vss 
      \hbox to 0pt{\hskip\hsize\hskip1em           
 \vbox{\hsize2cm\tiny\raggedright\pretolerance10000
 \noindent #1\hfill}\hss}\vbox to8pt{\vfil}\vss}}}%


\makeatother

\begin{document}

\title[normal forms and embeddings for power-log transseries]{NORMAL FORMS AND EMBEDDINGS \\
FOR POWER-LOG TRANSSERIES}

\author{P. Marde\v si\'c$^{1}$, M. Resman$^{2}$, J.-P. Rolin$^{3}$,
V. \v Zupanovi\'c$^{4}$}
\begin{abstract}
\emph{Dulac series} are asymptotic expansions of first return maps
in a neighborhood of a hyperbolic polycycle. In this article, we consider
two algebras of power-log transseries (generalized series) which extend
the algebra of Dulac series. We give a formal normal form and prove
a formal embedding theorem for transseries in these algebras. \end{abstract}
\maketitle
\begin{acknowledgement*}
This research was supported by: Croatian Science Foundation (HRZZ)
project no. 2285, French ANR project STAAVF, French-Croatian bilateral
Cogito project \emph{Classification de points fixes et de singularit\'es
}\`a\emph{ l'aide d'epsilon-voisinages d'orbites et de courbes},
and University of Zagreb research support for 2015 and 2016
\end{acknowledgement*}
`\tableofcontents{}

\section{Introduction and main results}

\subsection{Description of the problem and motivation\label{sub:description-problem}}

In the study of discrete dynamical systems, two problems are particularly
important: the search of a \emph{normal form} and the \emph{embedding}
problem. The search of a normal form means a definite process of choosing
a representative of the class of conjugacy of the original system
under topological, smooth or holomorphic conjugacies. This representative
should be simpler than the original one. The embedding problem consists
in finding a vector field such that the original system is the time-one
map of its flow. These two problems are obviously connected: it should
be easier to embed a normal form in a flow than the original system
itself.

The study of holomorphic or real analytic systems at the origin of
the ambient space naturally leads to the problems of \emph{formal
normal form} and \emph{formal embedding}. These questions are discussed
in detail, e.g., in \cite[Chapter I, Sections 3 and 4]{ilya}. The
case of dimension $1$ is well understood. Consider, for example,
a \emph{parabolic} formal series 
\[
f\left(z\right)=z+a_{1}z^{p+1}+z^{p+1}\varepsilon\left(z\right),
\]
where $a_{1}\in\mathbb{C}^{*}$, $\varepsilon\left(z\right)\in\mathbb{C}\left[\left[z\right]\right]$
and $\varepsilon\left(0\right)=0$. It is well known that the formal
conjugacy class of $f$ has a polynomial representative $f_{0}\left(z\right)=z+z^{p+1}+az^{2p+1}$,
where the \emph{residual index} $a\in\mathbb{C}$ is a formal invariant
(see \cite[Prop. 1.3.1]{loray}, \cite[Proposition 3.10]{abate} or
\cite{milnor}). The polynomial $f_{0}$ is classically called the
\emph{formal normal form} of $f$. Moreover, $f_{0}$ embeds formally
in the flow of the vector field $X_{p,a}=\Big(z^{p+1}+\big(a-\frac{p+1}{2}\big)z^{2p+1}\Big)\frac{d}{dz}$,
or in the formally equivalent flow of the vector field $X_{p,a}=\frac{z^{p+1}}{1+\big(\frac{p+1}{2}-a\big)z^{p}}\frac{d}{dz}$.
\\

It turns out that there exist interesting maps in one real variable
which do not admit asymptotic expansion in integer powers of the variable
at the origin. Consider for example a hyperbolic monodromic polycycle
$\Gamma$ of an analytic planar vector field (see \cite[Chapter $0$]{ilyalim}
for precise definitions). It has been proved by Dulac in \cite{Dulac}
that a semitransversal to $\Gamma$ can be chosen in such a way that
the corresponding \emph{first return map} (or \emph{Poincar\'e}\emph{
map}) admits at the origin an asymptotic expansion, called a \emph{Dulac
series} (see \emph{e.g.} \cite[Chapter $0$]{ilyalim} or \cite[Chapter 3.3]{roussarie}).
It is a formal series of the form: 
\begin{equation}
D\left(x\right)=c_{0}x^{\lambda_{0}}+\sum_{i=1}^{\infty}x^{\lambda_{i}}P_{i}(\log x),\quad c_{0}>0,\label{eq:dulac}
\end{equation}
where $\left(\lambda_{i}\right)_{i\in\mathbb{N}_{0}}$ is an increasing
sequence of strictly positive real numbers tending to infinity and
each $P_{i}$, $i\in\mathbb{N},$ is a polynomial with real coefficients.
Here, it is understood that $D\left(x\right)$ is an asymptotic expansion
of $P\left(x\right)$ at $0$ if, for any $N\in\mathbb{N}$, there
exists $k_{N}\in\mathbb{N}$ such that $P\left(x\right)-\sum_{i=1}^{k_{N}}x^{\lambda_{i}}P_{i}\left(\log x\right)=o\left(x^{N}\right)$.
The Dulac germs, as well as their asymptotic expansions, form a group
for the composition (see \cite[Chapter $0$]{ilyalim}). In particular,
notice that the condition $c_{0}>0$ in \eqref{eq:dulac} guarantees
that the composition of two Dulac series is well defined and is also
a Dulac series (no iterated logarithms are generated in the composition).
It can be proved that the exponents $\lambda_{i},\ i\in\mathbb{N}$,
in \eqref{eq:dulac} for the first return maps of hyperbolic monodromic
polycycles belong to a finitely generated additive semigroup of $\mathbb{R}$.
We will say, following the terminology of \cite[Section 7]{dries},
that such series is of \emph{finite type}. By \cite[Section 7]{dries},
the \emph{collection of Dulac series of finite type} form a subgroup
of all Dulac series for the composition, and is denoted by $\mathfrak{D}$
in the present work. \\

These examples lead us to consider formal normal form and formal embedding
problems for series with real coefficients in the monomials $x^{\alpha}\left(\log x\right)^{k}$,
$\alpha>0$, $k\in\mathbb{Z}$, considered as germs of functions at
$0^{+}$. More precisely, we are looking for classes of formal series
which extend the collection $\mathfrak{D}$ and inside which both
questions can be solved. It turns out that the set $\mathfrak{D}$
itself does not fit this purpose, mainly because Dulac series contain
only \emph{polynomials} in $\log x$, see Example~\ref{insuffic}
in Section~\ref{examp} for explanation. Hence, we introduce two
classes $\mathcal{L}_{\mathfrak{D}}$ and $\mathcal{L}$ of \emph{generalized
series} (or \emph{transseries} if one follows the terminology introduced
by \'Ecalle in \cite[Chapter 4]{ecalleb}), with $\mathfrak{D}\subset\mathcal{L}_{\mathfrak{D}}\subset\mathcal{L}$,
and we prove that both of them have the required properties. Compared
to Dulac series \eqref{eq:dulac}, the elements of $\mathcal{L}_{\mathfrak{D}}$
and $\mathcal{L}$ are of \emph{transfinite nature}: they involve
not only polynomials in $\log x$, but also \emph{infinite series}
in $\log x$. We do not address here the question of \emph{summability}
of transseries on some domain, nor the meaning of \emph{transseries
asymptotic expansions} of germs in general. The question of the existence
of an analytic function on some domain with an asymptotic expansion
in the form of a given transseries is left for future research (possibly
related to \'Ecalle's \emph{accelero--summability of transseries}
\cite{ecalle-acc-operators}).\\

We denote the set of positive real numbers by $\left(0,\infty\right)$
or $\mathbb{R}_{>0}$.\\

The elements of the class $\mathcal{L}_{\mathfrak{D}}$ are the transseries
\begin{equation}
f\left(x\right)=\sum_{\alpha\in S}\sum_{k=N_{\alpha}}^{\infty}a_{\alpha,k}x^{\alpha}\left(-\frac{1}{\log x}\right)^{k},\, a_{\alpha,k}\in\mathbb{R},\, N_{\alpha}\in\mathbb{Z},\label{eq:class-finite}
\end{equation}
where $S\subset\mathbb{R}_{>0}$, and the pairs $\left(\alpha,k\right)$
are contained in a sub-semigroup of the additive semigroup $\mathbb{R}_{>0}\times\mathbb{Z}$
generated by $\{(0,1)\}$ and by finitely many elements of $\mathbb{R}_{>0}\times\mathbb{Z}$.
Obviously, $\mathfrak{D}\subseteq\mathcal{L}_{\mathfrak{D}}$. For
short, we will say that the \emph{support} of $f$ (namely, the set
$\mathcal{S}\left(f\right)=\left\{ \left(\alpha,k\right)\in\mathbb{R}_{>0}\times\mathbb{Z}\colon a_{\alpha,k}\neq0\right\} $)
is of \emph{finite type}, or equivalently, that $f$ itself is of
finite type. Notice that, while the germ $x$ at the origin $0^{+}$
of $\mathbb{R}_{>0}$ is a positive infinitesimal, the germ $\log x$
at $0^{+}$ is negative and infinitely large. This is why we prefer
to work with the germ $-1/\log x$ instead and to introduce the symbol
\[
\boldsymbol{\ell}=-\dfrac{1}{\log x}.
\]
We denote the elements of $\mathcal{L}_{\mathfrak{D}}$ as in \eqref{eq:class-finite}
indifferently by $f$ or by $f\left(x\right)$. We call a germ $x^{\alpha}\boldsymbol{\ell}^{k}$,
$\alpha>0$, $k\in\mathbb{Z}$, a \emph{power-log monomial}.\emph{
}Finally, the series $f\left(x\right)=x$ will often be denoted simply
by $f=\text{id}$.\\

In order to define the class $\mathcal{L}$, let us recall that an
ordered set $X$ is called \emph{well-ordered} if every non empty
subset of $X$ has a smallest element. This property implies in particular
that $X$ is totally ordered. The elements of $\mathcal{L}$ are the
formal transseries of the following form: 
\begin{align}
f(x)= & \sum_{\alpha\in S}\sum_{k\in\mathbb{Z}}a_{\alpha,k}\, x^{\alpha}(-\frac{1}{\log x})^{k}=\sum_{\alpha\in S}\sum_{k\in\mathbb{Z}}a_{\alpha,k}x^{\alpha}\boldsymbol{\ell}^{k},\ a_{\alpha,k}\in\mathbb{R},\label{class}
\end{align}
where $S\subset\mathbb{R}_{>0}$, and the support $\mathcal{S}\left(f\right)=\left\{ \left(\alpha,k\right)\in\mathbb{R}_{>0}\times\mathbb{Z}\colon a_{\alpha,k}\neq0\right\} $
is a well-ordered subset of $\mathbb{R}_{>0}\times\mathbb{Z}$, equipped
with the \emph{lexicographic order} $\preceq$. It is equivalent to
assume: 
\begin{itemize}
\item (1) $\mathcal{S}\left(f\right)$ is a well-ordered subset of $\mathbb{R}_{>0}\times\mathbb{Z}$; 
\item (2) $S$ is a well-ordered subset of $\mathbb{R}_{>0}$ and, for every
$\alpha\in S$, there exists $N_{\alpha}\in\mathbb{Z}$, such that
a pair $\left(\alpha,k\right)$ belongs to $\mathcal{S}\left(f\right)$
only if $k\ge N_{\alpha}$. 
\end{itemize}
In particular, one can easily check that $\mathcal{L}_{\mathfrak{D}}\subset\mathcal{L}$.
Recall that a well-ordered subset of $\mathbb{R}$ is countable; hence,
the supports of the elements of $\mathcal{L}$ are countable. The
lexicographic order $\preceq$ on pairs of exponents corresponds to
the usual order $\leq$ on germs of functions at the origin, in the
following way: given two pairs $\left(\alpha,k\right)$ and $\left(\alpha',k'\right)$
in $\mathbb{R}_{>0}\times\mathbb{Z}$, we have 
\[
\left(\alpha,k\right)\preceq\left(\alpha',k'\right)\Longleftrightarrow\lim_{x\to0^{+}}\frac{x^{\alpha'}\boldsymbol{\ell}^{k'}}{x^{\alpha}\boldsymbol{\ell}^{k}}<\infty\Longleftrightarrow x^{\alpha'}\boldsymbol{\ell}^{k'}\leq x^{\alpha}\boldsymbol{\ell}^{k}.
\]
Similarly $\left(\alpha,k\right)\prec\left(\alpha',k'\right)$ means
that $\left(\alpha,k\right)\preceq\left(\alpha',k'\right)$ and $\left(\alpha,k\right)\ne\left(\alpha',k'\right)$.
We call the pair $(\alpha,k)$ the \emph{order} of the monomial $x^{\alpha}\boldsymbol{\ell}^{k}$.
The \emph{order of a transseries} $f\in\mathcal{L}$, denoted by $\text{ord}\left(f\right)$,
is the smallest element of $\mathcal{S}\left(f\right)$. If $\left(\alpha,k\right)$
is the order of $f$ then $x^{\alpha}\boldsymbol{\ell}^{k}$ is called
the \emph{leading monomial} of $f$ and is denoted by $\mathrm{Lm}(f)$,
$a_{\alpha,k}$ is called the \emph{leading coefficient} of $f$ and
is denoted by $\mathrm{Lc}(f)$, and $a_{\alpha,k}x^{\alpha}\boldsymbol{\ell}^{k}$
is called the \emph{leading term} of $f$ and is denoted by $\mathrm{Lt}(f)$. 
\begin{notation*}
We will sometimes denote by $\left[f\right]_{\alpha,k}$ the coefficient
of the monomial $x^{\alpha}\boldsymbol{\ell}^{k}$ in the series $f\in\mathcal{L}$. 
\end{notation*}
While the questions of formal embeddings and formal normal forms in
$\mathcal{L}$ and $\mathcal{L}_{\mathfrak{D}}$ can be considered
as natural problems of independent interest, our motivation for this
research lies in fractal analysis of orbits of germs. Given an orbit
$\mathcal{O}$ of a germ, by its fractal analysis we mean understanding
the function $\varepsilon\mapsto A(\varepsilon)$, that assigns to
each $\varepsilon>0$ the Lebesgue measure of the $\varepsilon$-neighborhood
of the orbit $\mathcal{O}$. The question that we pose is if we can
recognize a germ by fractal properties of its realizations (orbits).
Fractal properties of orbits of Poincar{\'e} maps around limit periodic
sets have been studied in \cite{mrz} and \cite{belgproc}. In the
differentiable cases of elliptic points and limit cycles, it was proven
in \cite{belgproc} that fractal analysis of orbits of Poincar{\'e}
maps gives the multiplicity and the cyclicity. As already mentioned,
in the nondifferentiable cases of hyperbolic polycycles, Poincar{\'e}
maps have an expansion in $\mathcal{L}_{\mathfrak{D}}$. Furthermore,
fractal analysis was performed on holomorphic complex germs in \cite{resman}
and \cite{nonlin}. It was shown in \cite{resman} that the function
$\varepsilon\mapsto A(\varepsilon),\ \varepsilon>0,$ characterizes
the formal class of a parabolic germ. The analytic class cannot be
characterized by $A(\varepsilon)$, since it does not have an asymptotic
expansion, as $\varepsilon\to0$, see \cite{nonlin}. In a subsequent
work, we plan to introduce a new definition of the \emph{formal area}
$A(\varepsilon)$, based on the formal embedding theorem proven in
the present paper. With this new definition, we further hope for a
sectorially analytic function which will reveal the analytic class
of a germ. This would give a way to \emph{see} the analytic class
of a germ by \emph{looking} at its orbits.

\subsection{Overview of the results\label{overview}}

Our main results (Theorems A and B) hold for a subclass of elements
of $\mathcal{L}_{\mathfrak{D}}$ and $\mathcal{L}$. We say that an
element $f$ of $\mathcal{L}$ \emph{contains no logarithms in the
leading term} $\mathrm{Lt}(f)$ if $f$ is of the form 
\[
(H)\qquad f(x)=\lambda x^{\alpha}+\sum_{\left(\alpha,0\right)\prec\left(\beta,k\right)}a_{\beta,k}\, x^{\beta}\boldsymbol{\ell}^{k},\hspace{1em}\lambda>0,\ \alpha_{\beta,k}\in\mathbb{R}.
\]
We denote by $\mathcal{L}^{H}$ the subset of transseries from $\mathcal{L}$
that satisfy $(H)$, and by $\mathcal{L}_{\mathfrak{D}}^{H}$ the
intersection $\mathcal{L}_{\mathfrak{D}}\cap\mathcal{L}^{H}$. There
are two reasons for this additional assumption on the leading monomial.
First, we have already mentioned that the Dulac series which are asymptotic
expansions of Poincar{\'e} maps of hyperbolic polycyles belong to
$\mathcal{L}_{\mathfrak{D}}^{H}$. Second, unlike $\mathcal{L}$,
where iterated logarithms may be generated by compositions, the class
$\mathcal{L}^{H}$ is a group for the composition of transseries.

The leading term in the asymptotic expansion at $0$ of a germ indicates
the rate of convergence of its orbits (or backward orbits) towards
$0$. According to the standard terminology used for holomorphic diffeomorphisms
(see for example \cite{abate}, \cite{milnor}), we distinguish three
cases: 
\begin{defn}
Let $f\in\mathcal{L}^{H}$, $f(x)=\lambda x^{\alpha}+\cdots$, $\lambda>0$,
$\alpha>0$. We say that $f$ is 
\begin{enumerate}[1., font=\textup, nolistsep, leftmargin=0.6cm]
\item \emph{strongly} \emph{hyperbolic},\emph{ }if $\alpha\neq1$, 
\item \emph{hyperbolic}, if $\alpha=1$ and $\lambda\neq1$, 
\item \emph{parabolic}, if $\alpha=1$ and $\lambda=1$. 
\end{enumerate}
Additionally, we say that a hyperbolic $f$ is a \emph{hyperbolic
contraction} if $0<\lambda<1$. If $\lambda>1$, we call $f$ a \emph{hyperbolic
expansion}.
\end{defn}
Intuitively, strongly hyperbolic cases $\alpha>1$ correspond to strong
contractions in the first term, and cases $\alpha<1$ to strong expansions.
Hyperbolic cases $0<\lambda<1$ correspond to exponential contractions,
and cases $\lambda>1$ to exponential expansions.\\

We denote by $\mathcal{L}^{0}\subset\mathcal{L}^{H}$ the set of \emph{formal
changes of variables in $\mathcal{L}$}: 
\[
\mathcal{L}^{0}=\{\varphi\in\mathcal{L}^{H}:\,\varphi(x)=ax+\text{h.o.t},\ a>0\}.
\]
We use here the shortcut ``h.o.t.'' for ``higher order terms''.
Similarly, put $\mathcal{L}_{\mathfrak{D}}^{0}=\mathcal{L}_{\mathfrak{D}}\cap\mathcal{L}^{0}$.
Unlike $\mathcal{L}$, the classes $\mathcal{L}^{H},\mathcal{L}_{\mathfrak{D}}^{H},\mathcal{L}^{0}$
and $\mathcal{L}_{\mathfrak{D}}^{0}$ are closed under formal compositions
of transseries and they are groups with respect to this operation.
We say that $f,g\in\mathcal{L}^{H}$ (resp. $\mathcal{L}_{\mathfrak{D}}^{H}$)
are \emph{formally equivalent in $\mathcal{L}^{0}$} (\emph{resp.}
$\mathcal{L}_{\mathfrak{D}}^{0}$) if there exists a change of variables
$\varphi\in\mathcal{L}^{0}$ (\emph{resp.} $\varphi\in\mathcal{L}_{\mathfrak{D}}^{0}$)
transforming $f$ to $g$, $g=\varphi^{-1}\circ f\circ\varphi$.\\

We now recall the definitions needed to state our formal embedding
theorem. In the settings of usual power series, similar definitions
may be found in, for example, \cite{ilya} or \cite{loray}. 
\begin{defn}[The formal flow of a formal vector field]
\label{def:time-one-map} Consider a family $\left(f^{t}\right)_{t\in\mathbb{R}}$
of elements of $\mathcal{L}^{H}$.
\begin{enumerate}[1., font=\textup, nolistsep, leftmargin=0.6cm]
\item We say that $\left(f^{t}\right)$ forms a \emph{one-parameter group}
(we also say for short: defines a \emph{flow}) if $f^{0}=\mathrm{id}$
and $f^{s}\circ f^{t}=f^{s+t}$, for all $s,\, t\in\mathbb{R}$. An
element $f\in\mathcal{L}^{H}$ \emph{embeds} in the flow $\left(f^{t}\right)_{t\in\mathbb{R}}$
if $f=f^{1}$. 
\item The family $\left(f^{t}\right)_{t\in\mathbb{R}}$ is called a \emph{$\mathcal{C}^{1}$-one-parameter
group} or a\emph{ $\mathcal{C}^{1}$-flow} if it defines a flow, and
moreover :

\begin{enumerate}[(i), font=\textup, nolistsep,leftmargin=0.6cm]
\item there exists a well-ordered subset $S$ of $\mathbb{R}_{>0}\times\mathbb{Z}$
such that $\mathcal{S}\left(f^{t}\right)\subseteq S$ for every $t\in\mathbb{R}$,
and
\item for every $\left(\alpha,m\right)\in S$, the function $t\mapsto\left[f^{t}\right]_{\alpha,m}$
is $\mathcal{C}^{1}\left(\mathbb{R}\right)$.
\end{enumerate}
\item \noindent Assume that $\left(f^{t}\right)$ is a $\mathcal{C}^{1}$-flow
and let $\xi:=\frac{\mathrm{d}f^{t}}{\mathrm{d}t}|_{t=0}\in\mathcal{L}$.
Then we say that $\left(f^{t}\right)$ is the \emph{$\mathcal{C}^{1}$-formal
flow} of the vector field $X=\xi\frac{\mathrm{d}}{\mathrm{d}x}$.
In that case, $f^{t}$ is called the \emph{formal $t$-map of $X$},
$t\in\mathbb{R}$.
\end{enumerate}
\end{defn}
\begin{rem}
The third point of the former definition means that, if we write
\[
f^{t}\left(x\right)=\sum_{\alpha,k}\left[f^{t}\right]_{\alpha,k}x^{\alpha}\boldsymbol{\ell}^{k},\quad\forall t\in\mathbb{R},
\]
then 
\[
\xi\left(x\right)=\sum_{\alpha,k}\frac{\mathrm{d}\left[f^{t}\right]_{\alpha,m}}{\mathrm{d}t}\Big|_{t=0}x^{\alpha}\boldsymbol{\ell}^{k}.
\]

\end{rem}
We show in Propositions~\ref{prop: existence-formal-flow-parabolic},
\ref{prop: existence-formal-flow-hyperbolic} and \ref{prop:unic}
in Section~\ref{sub:fivetwo} that a vector field $X=\xi\frac{d}{dx}$,
$\xi\in\mathcal{L}$, such that $\left(1,0\right)\preceq\mathrm{ord}\left(\xi\right)$
admits a \emph{unique} $\mathcal{C}^{1}$-formal flow $(f^{t})_{t}$
in $\mathcal{L}^{H}$, which is given by: 
\begin{equation}
f^{t}=\exp(tX)\cdot\mathrm{id}=\mathrm{id}+t\xi+\frac{t^{2}}{2!}{\xi}'\xi+\frac{t^{3}}{3!}(\xi'\xi)'\xi+\cdots,\ t\in\mathbb{R}.\label{eq:eksp}
\end{equation}

\begin{rem}
\label{rem:two_topologies} We prove the convergence of formula \eqref{eq:eksp}
in Propositions~\ref{prop: existence-formal-flow-parabolic} and
\ref{prop: existence-formal-flow-hyperbolic} in Section~\ref{sub:fivetwo}.
We will actually need two notions of convergence. The first one is
relevant of what we call the \emph{formal topology}, see Section~\ref{sec:topologies}.
To describe it roughly, let us say that the formal topology takes
into account the orders of monomials, but not the size of coefficients.
The series \eqref{eq:eksp} converges in this topology when $\left(1,0\right)\prec\mathrm{ord}\left(\xi\right)$.
Nevertheless, it does not converge when $\mathrm{ord}\left(\xi\right)=\left(1,0\right)$.
Hence, we introduce a coarser \emph{weak topology} (later: the product
topology with respect to the Euclidean topology), in which the coefficients
of the monomials play a role for the convergence of series. In this
weak topology, the series \eqref{eq:eksp} converges even when $\mathrm{ord}\left(\xi\right)=\left(1,0\right)$,
see Proposition~\ref{prop: existence-formal-flow-parabolic}.

Finally, if $\text{ord}(\xi)\prec(1,0)$, the series \eqref{eq:eksp}
does not converge in any of these topologies (Proposition~\ref{prop:not-weakly-well-defined}).
\end{rem}
We now state the two main theorems of this paper. The precise, but
more technical, formulations are given in Sections \ref{sec:proof-theorem-A}
and \ref{sec:proof-theorem-B}.
\begin{namedthm}
{Theorem A} Let $f\in\mathcal{L}^{H}$ (\emph{resp.} $f\in\mathcal{L}_{\mathfrak{D}}^{H}$).
Then: 
\begin{enumerate}[1., font=\textup, nolistsep, leftmargin=0.6cm]
\item $f$ is formally equivalent to a normal form $f_{0}\in\mathcal{L}^{H}$
(\emph{resp.} $f_{0}\in\mathcal{L}_{\mathfrak{D}}^{H}$), given as
a finite sum of power-log monomials. 
\item If $f$ is parabolic or hyperbolic, then $f$ is formally equivalent
to the formal time-one map $\widehat{f}_{0}\in\mathcal{L}^{H}$ (\emph{resp.
}$\widehat{f}_{0}\in\mathcal{L}_{\mathfrak{D}}^{H}$) of a (formal)
vector field $X=\xi\frac{\mathrm{d}}{{\mathrm{d}x}}$, where $\xi\in\mathcal{L}$
(resp. $\xi\in\mathcal{L}_{\mathfrak{D}}$) is a rational function
in power-log monomials. 
\end{enumerate}
\end{namedthm}
The formal normal forms of Theorem A are described by \emph{at most
$4$ scalars}. The actual number of scalars depends on the type (parabolic,
hyperbolic or strongly hyperbolic) of the diffeomorphism.

The proof of Theorem A in Section~\ref{sec:proof-theorem-A} is actually
based on a \emph{transfinite algorithm} which transforms any transseries
$f$ in $\mathcal{L}^{H}$ or $\mathcal{L}_{\mathfrak{D}}^{H}$ into
its finite formal normal form $f_{0}$.
\begin{namedthm}
{Theorem B}Let $f\in\mathcal{L}^{H}$. Then $f$ embeds in a flow
$\left(f^{t}\right)_{t\in\mathbb{R}}$, $f^{t}\in\mathcal{L}^{H}$.
Moreover, if $f$ is parabolic or hyperbolic, $f$ embeds in the $\mathcal{C}^{1}$-flow
of a unique vector field $X=\xi\frac{\mathrm{d}}{{\mathrm{d}x}}$,
$\xi\in\mathcal{L}$ (see Definition~\ref{def:time-one-map}).
\end{namedthm}
For the detailed statements discussing all cases (parabolic, hyperbolic,
strongly hyperbolic) and their proofs, see Sections~ \ref{sec:proof-theorem-A},
\ref{sec:proof-theorem-B} respectively. \\

\section{\label{sub:Hahn-fields}Hahn fields and the structures of $\mathcal{L}$,
$\mathcal{L}^{H}$ and $\mathcal{L}^{0}$ \emph{(resp.} $\mathcal{L}_{\mathfrak{D}}$,
$\mathcal{L}_{\mathfrak{D}}^{H}$ and $\mathcal{L}_{\mathfrak{D}}^{0}$).}

Various descriptions of the notion of \emph{transseries} have been
given in several publications.\emph{ }The detailed study of classical
operations, such as the operations of fields, as well as derivation,
integration or composition, in this setting, is given in detail in
\cite{dries}. The classes of transseries considered in the present
work are proper sub-classes of the general field $\mathbb{R}\left(\left(x^{-1}\right)\right)^{\mathrm{LE}}$
of \emph{logarithmic-exponential series} (or \emph{LE-series}) introduced
in \cite{dries}. Therefore, the operations we have to deal with are
mostly specializations, in our framework, of the similar operations
described there. In particular, the proof of the closure of $\mathcal{L}^{H}$
and $\mathcal{L}^{0}$ under composition can be checked by a careful,
but straightforward, adaptation of the corresponding statement in
$\mathbb{R}\left(\left(x^{-1}\right)\right)^{\mathrm{LE}}$.

Hence, we just provide in this section the vary basic notions needed
to perform the description of our classes $\mathcal{L},\ \mathcal{L}^{H}$
and $\mathcal{L}^{0}$. We use, as in \cite{dries}, the language
of \emph{Hahn fields}.\emph{ }Recall that given a multiplicative ordered
abelian group $\Gamma$ with unit $1$, the \emph{Hahn field $\mathbb{R}\left(\left(\Gamma\right)\right)$}
consists of \emph{generalized series} with real coefficients and monomials
in $\Gamma$. The elements of $\mathbb{R}\left(\left(\Gamma\right)\right)$
are the formal sums 
\[
f=\sum_{\gamma\in\Gamma}f_{\gamma}\gamma,
\]
with coefficients $f_{\gamma}\in\mathbb{R}$, such that $\mathrm{Supp}\left(f\right)=\left\{ \gamma\in\Gamma\colon f_{\gamma}\ne0\right\} $
is a \emph{reverse well-ordered} subset of $\Gamma$. If $f\ne0$
and $\gamma_{0}$ is the biggest element of $\text{Supp}\left(f\right)$,
then the \emph{leading term} $\text{Lt}\left(f\right)$ of $f$ is
$f_{\gamma_{0}}\gamma_{0}$, its \emph{leading monomial} $\text{Lm}\left(f\right)$
is $\gamma_{0}$ and its \emph{leading coefficient} is $f_{\gamma_{0}}$.
\\

One of the most useful tools when dealing with algebraic operations
on Hahn fields is a result due to Neumann \cite{neumann}. Its statement
requires the following notations.
\begin{notation}
\label{not:notation-appendix}Consider an ordered (multiplicative)
abelian group $\Gamma$ and two subsets $A$ and $B$ of $\Gamma$.
We denote: 
\begin{enumerate}[a), font=\textup, nolistsep,leftmargin=0.6cm]
\item $AB:=\left\{ ab\colon a\in A,b\in B\right\} $, 
\item $\left\langle A\right\rangle $: the sub-semigroup of $\Gamma$ generated
by $A$ (i.e. the smallest sub-semigroup of $G$ containing all elements
of $A$),
\item $\Gamma_{<\gamma_{0}}:=\left\{ \gamma\in\Gamma:\,\gamma<\gamma_{0}\right\} $,
$\Gamma_{\leq\gamma_{0}}:=\left\{ \gamma\in\Gamma:\,\gamma\leq\gamma_{0}\right\} $,
$\gamma_{0}\in\Gamma$. Note that $\Gamma_{<1}$ denotes the \emph{infinitesimals}
of the group $\Gamma.$
\end{enumerate}
\end{notation}
\begin{lem}[Neumann's Lemma]
 \label{lem:neumann_lemma} Consider an ordered (multiplicative)
abelian group $\Gamma$ and two reverse well-ordered subsets $A$
and $B$ of $\Gamma$. Then: 
\begin{enumerate}[1., font=\textup, nolistsep, leftmargin=0.6cm]
\item The product $AB$ is reverse well-ordered. 
\item For $g\in AB$, there are only finitely many pairs $\left(a,b\right)\in A\times B$
such that $g=ab$. 
\item If $A\subseteq\Gamma_{<1}=\{g\in\Gamma\,:\, g<1\}$ is reverse well-ordered,
then $\left\langle A\right\rangle $ is also reverse well-ordered.
Moreover, for each $g\in\left\langle A\right\rangle $ there are only
finitely many tuples $\left(a_{1},\ldots a_{n}\right)$ with $n\in\mathbb{N}$,
$a_{1},\ldots,a_{n}\in A$, such that $g=a_{1}\cdots a_{n}$. 
\end{enumerate}
\end{lem}
These series can be added and multiplied in the following way: if
$f=\sum_{\gamma\in\Gamma}f_{\gamma}\gamma$ and $g=\sum_{\gamma\in\Gamma}g_{\gamma}\gamma$
belong to $\mathbb{R}\left(\left(\Gamma\right)\right)$, then 
\[
f+g=\sum_{\gamma\in\Gamma}\left(f_{\gamma}+g_{\gamma}\right)\gamma,\quad f\cdot g=\sum_{\gamma\in\Gamma}\left(\sum_{\lambda\mu=\gamma}f_{\lambda}g_{\mu}\right)\gamma.
\]
Notice that the reverse well-ordering of the supports of $f$ and
$g$ guarantees, thanks to Neuman's Lemma (see \cite{neumann} or
\cite[p. 64]{dries} for example), that the product is well defined.
Moreover, it is known that every nonzero element of $\mathbb{R}\left(\left(\Gamma\right)\right)$
admits a multiplicative inverse, so that $\mathbb{R}\left(\left(\Gamma\right)\right)$
is actually a field. If now $\Gamma'$ is a sub-semigroup of $\Gamma$,
then the set 
\[
\mathbb{R}\left[\left[\Gamma'\right]\right]=\left\{ f\in\mathbb{R}\left[\left[\Gamma\right]\right]\colon\text{Supp}\left(f\right)\subseteq\Gamma'\right\} 
\]
is a subring (actually an $\mathbb{R}$-algebra) of $\mathbb{R}\left(\left(\Gamma\right)\right)$.\\

The LE-series\emph{ }introduced in \cite{dries} are generalized series
in one variable whose monomials involve the logarithm and the exponential
functions. Our classes $\mathcal{L}$ and $\mathcal{L}_{\mathfrak{D}}$
are contained in the field of LE-series (up to the obvious modification
which comes from the fact that the variable $x$ is thought as ``infinitely
big'' there, while it is infinitesimal in our work). Let us show
how $\mathcal{L}$ and $\mathcal{L}_{\mathfrak{D}}$ can be described
by following the above Hahn's construction. Consider the multiplicative
group $G$:

\[
G=\{x^{\alpha}\boldsymbol{\ell}^{k}:\alpha\in\mathbb{R},\ k\in\mathbb{Z}\},
\]
and the multiplicative sub-semigroup $G'=\left\{ x^{\alpha}\boldsymbol{\ell}^{k}:\alpha\in\left(0,\infty\right),k\in\mathbb{Z}\right\} $
of $G$, equipped with the order $\leq$ introduced in Section \ref{sub:description-problem}.
Then the class $\mathcal{L}$ is equal to the ring $\mathbb{R}\left[\left[G'\right]\right]$.
It is a subring of a Hahn field $\mathbb{R}((G))$, which is itself
a subfield of the general field of LE-series. Notice that the support
$\mathcal{S}\left(f\right)$ of a series $f\in\mathcal{L}$ is in
Section \ref{sub:description-problem} defined as a subset of $\mathbb{R}_{>0}\times\mathbb{Z}$.
It differs from the support $\text{Supp}\left(f\right)$ defined above
for elements of general Hahn fields, which would be a set of monomials.
The reason is that, in our situation, it is more convenient to work
directly with exponents than to work with monomials. For the same
reason, we will often refer to the additive version of Neumann's Lemma
adapted to sets of exponents rather than to the multiplicative version
stated above, which is adapted to sets of monomials.

Finally, as a straightforward consequence of Neumann's Lemma, $\mathcal{L}^{H}$
is an additive and multiplicative sub-semigroup of $\mathcal{L}$,
and $\mathcal{L}^{0}$ is an additive sub-semigroup of $\mathcal{L}^{H}$.
Furthermore, $\mathcal{L}_{\mathfrak{D}}^{H}$ is an additive and
multiplicative sub-semigroup of $\mathcal{L}_{\mathfrak{D}}$, and
$\mathcal{L}_{\mathfrak{D}}^{0}$ is an additive sub-semigroup of
$\mathcal{L}_{\mathfrak{D}}^{H}$. \medskip{}

We consider now the operation of composition, as an imported operation
from the general field of LE-series. As mentioned above, it is proved
in \cite{dries} that the field of transseries can be equipped with
a composition operator, and that each nonzero LE-series admits a compositional
inverse. The proof of these facts requires an elaborate construction,
which was previously sketched in \cite[Chapter 4]{ecalleb}. Fortunately,
the action of the restriction of these two operators to our classes
is much simpler, due to the particular shape of the monomials in $G'$.
To be more precise, the composition of two series is understood by
classical \emph{Taylor expansions}. It is mainly based on the following
observations, which are used in almost all subsequent computations
of this paper. Every series $f\in\mathcal{L}^{H}$ can be written
as: 
\[
f\left(x\right)=ax^{\lambda}\left(1+\varepsilon\left(x\right)\right)\text{ with }\varepsilon\in\mathcal{L},\text{ }\varepsilon\left(0\right)=0\text{ and }\lambda>0.
\]
For every real number $\alpha>0$, the composition defined in \cite{dries}
leads to: 
\[
\left(x^{\alpha}\circ f\right)\left(x\right)=\left(f\left(x\right)\right)^{\alpha}=a^{\alpha}x^{\lambda\alpha}\sum_{j=0}^{\infty}\binom{\alpha}{j}\varepsilon\left(x\right)^{j}.
\]
In the same way, if $f$ is positive (that is, if $a>0$), we have:
\[
\log\left(f\left(x\right)\right)=\log a+\lambda\log\left(x\right)+\sum_{j=1}^{\infty}\frac{\left(-1\right)^{j+1}}{j}\varepsilon^{j}\left(x\right).
\]
The analysis made in \cite{dries} shows how these formulas extend
to composition of series in the following way. If $g\left(x\right)=\sum_{\left(\alpha,k\right)}c_{\alpha,k}x^{\alpha}\boldsymbol{\ell}^{k}\in\mathcal{L}$
and $f\in\mathcal{L}^{H}$, then 
\[
\left(g\circ f\right)\left(x\right)=\sum_{\left(\alpha,k\right)}c_{\alpha,k}\big(f\left(x\right)\big)^{\alpha}\Big(\frac{-1}{\log f\left(x\right)}\Big)^{k}
\]
is a well defined element of $\mathcal{L}$. As a consequence of the
results proved in \cite[Section 7]{dries}, the similar conclusion
holds for $\mathcal{L}_{\mathfrak{D}},\mathcal{L}_{\mathfrak{D}}^{H}$
and $\mathcal{L}_{\mathfrak{D}}^{0}$ (the finite type property of
the support is preserved). We summarize the former results in the
following proposition. 
\begin{prop}[Properties of $\mathcal{L}$, $\mathcal{L}^{H}$ and $\mathcal{L}^{0}$]
 \label{prop:properties-of-L} \ 
\begin{enumerate}[1., font=\textup, nolistsep, leftmargin=0.6cm]
\item $\mathcal{L}$ (\emph{resp.} $\mathcal{L}_{\mathfrak{D}}$) is an
$\mathbb{R}$-algebra without unity. 
\item The algebra $\mathcal{L}$ (\emph{resp.} $\mathcal{L}_{\mathfrak{D}}$)
is closed under right compositions with elements from $\mathcal{L}^{H}$
(resp. $\mathcal{L}_{\mathfrak{D}}^{H}$). 
\item The sets $\mathcal{L}^{H}$ and $\mathcal{L}^{0}$ (\emph{resp.} $\mathcal{L}_{\mathfrak{D}}^{H}$
and $\mathcal{L}_{\mathfrak{D}}^{0}$) are groups under composition.
In particular, they are closed under compositional inverses.
\end{enumerate}
\end{prop}
The next section is dedicated to adaptations of standard \emph{Lie
bracket} techniques to our transseries setting. These techniques are
used in the proofs in Sections~\ref{sec:proof-theorem-A} and \ref{sec:proof-theorem-B}.

\section{Lie brackets in search of normal forms}

\noindent \label{sec:lie-brackets-normal-forms}

The method of producing normal forms for analytic or formal diffeomorphisms
is an adaptation of the \emph{Lie bracket} technique for normal forms
of vector fields, which is described for example in \cite{arnold}
or \cite{takens}. As we plan to adapt the method for elements of
our algebra $\mathcal{L}$, we first recall briefly how it works in
the classical case, more precisely, for formal power series in one
variable. 
\begin{rem}
In the sequel, the \emph{h.o.t}., meaning \emph{higher-order terms},
stands for monomials of higher order than the last one written.
\end{rem}

\subsection{The effect of a change of variables on the elements of $\mathbb{R}\left[\left[x\right]\right]$}

Consider a series $f\in\mathbb{R}\left[\left[x\right]\right]$ such
that $f\left(0\right)=0$. In order to transform $f$ into its normal
form, a classical approach consists in describing the effect on $f$
of a change of variables $\varphi\in\mathbb{R}\left[\left[x\right]\right]$,
such that $\varphi\left(x\right)=x+\text{h.o.t}$. The simplest method
consists in considering the leading term $\psi=\mathrm{Lt}(f\circ\varphi-\varphi\circ f)$.
This leading term $\psi$ is the same as the leading term of the difference
$\varphi^{-1}\circ f\circ\varphi-f$. Using Taylor formula, we have:
\begin{align}
f\circ\varphi & =\varphi\circ f+\psi\cdot\left(1+\eta\right),\hspace{1em}\eta\in x\mathbb{R}\left[\left[x\right]\right]\nonumber \\
\varphi^{-1}\circ f\circ\varphi & =\varphi^{-1}\left(\varphi\circ f+\psi\cdot\left(1+\eta\right)\right)\nonumber \\
\varphi^{-1}\circ f\circ\varphi\left(x\right) & =f\left(x\right)+\left(\varphi^{-1}\right)'\left(\varphi\circ f\left(x\right)\right)\cdot\psi\left(x\right)+\text{h.o.t.}\nonumber \\
 & =f\left(x\right)+\psi\left(x\right)+\text{h.o.t.},\label{eq:formule-conjugaison}
\end{align}
since $\varphi'\left(x\right)=1+\text{h.o.t.}$

Recall that the goal of formal normalization is to produce a series
in the class of formal conjugacy of $f$ which contains the smallest
possible number of terms. So, given a term $\tau$ in the expansion
of $f$, the main step consists in removing $\tau$ (if possible)
via an appropriate change of variables $\varphi$. To do this, we
choose the change of variables $\varphi$ such that the leading term
$\mathrm{Lt}(f\circ\varphi-\varphi\circ f)$ is the opposite of $\tau$.
This procedure is then repeated term by term.

In particular, if $f$ is \emph{parabolic}, that is if $f\left(x\right)=x+\varepsilon\left(x\right)=x+ax^{p}+\text{h.o.t}$,
where $p>1$, then we look for a change of variables $\varphi\left(x\right)=x+\eta\left(x\right)=x+cx^{m}$,
$m>1$,\ $c\in\mathbb{R}$. We obtain: 
\begin{align}
\left(f\circ\varphi-\varphi\circ f\right)\left(x\right) & =f\left(x+\eta\left(x\right)\right)-\varphi\left(x+\varepsilon\left(x\right)\right)\nonumber \\
 & =f\left(x\right)+f'\left(x\right)\eta\left(x\right)-\varphi\left(x\right)-\varphi'\left(x\right)\varepsilon\left(x\right)+\text{h.o.t.}\nonumber \\
 & =x+\varepsilon\left(x\right)+\left(1+\varepsilon'\left(x\right)\right)\eta\left(x\right)-x-\eta\left(x\right)-\left(1+\eta'\left(x\right)\right)\varepsilon\left(x\right)+\text{h.o.t.}\nonumber \\
 & =\varepsilon'\left(x\right)\eta\left(x\right)-\varepsilon\left(x\right)\eta'\left(x\right)+\text{h.o.t.}\label{eq:adjoint}
\end{align}
The series $\eta\varepsilon'-\eta'\varepsilon$ is called the \emph{Lie
bracket (the commutator)} of $\eta$ and $\varepsilon$ and is denoted
by $\left\{ \eta,\varepsilon\right\} $. The leading term $\psi$
of $f\circ\varphi-\varphi\circ f$ is given by the Lie bracket $\left\{ cx^{m},ax^{p}\right\} $
of the leading terms of $\eta$ and $\varepsilon$.

\subsection{\label{sub:poisson-brackets-homologic}Lie brackets in $\mathbb{R}[[x]]$
and the homological equation}

The action of the Lie bracket of\emph{ $g$} is given by the following
linear operator on $\mathbb{R}[[x]]$: 
\begin{equation}
\mathrm{ad}_{g}(f)=[f,g],\ f\in\mathbb{R}[[x]].\label{pb}
\end{equation}
Denote by $H_{k}$ the vector space of monomials of degree $k$, $k\in\mathbb{N}$:
\[
H_{k}=\big\{ ax^{k}:\, a\in\mathbb{R}\big\},\ k\geq1.
\]
The \emph{grading} of the space $H_{k}$ is given by the degree $k$
of its monomials.\\

Let $f\left(x\right)=x+ax^{p}+$$\text{h.o.t.}$ be a parabolic element
of $\mathbb{R}\left[\left[x\right]\right]$. It can be reduced to
its formal normal form by solving a series of \emph{Lie bracket $($commutator$)$
equations}, considering the action of the Lie bracket of the leading
monomial of $f-\mathrm{id}$ to spaces $H_{l}$, $l\in\mathbb{N}$.
The idea is to work step by step and, in each step, to eliminate the
monomial of a given degree, if possible. Here we describe a single
step. 

Applying the change of variables $\varphi(x)=x+cx^{m}$, $c\in\mathbb{R},\ m\in\mathbb{N},\ m>1$,
we obtain, according to formula \eqref{eq:adjoint}: 
\begin{align}
\varphi^{-1}\circ f\circ\varphi & =f+\mathrm{ad}_{ax^{m}}(cx^{m})+\mathrm{h.o.t.}\nonumber \\
 & =f+ac(p-m)x^{p+m-1}+\mathrm{h.o.t.}\label{eq:lb}
\end{align}
Since, $\mathrm{ad}_{x^{m}}(H_{l})\subseteq H_{m+l-1}$, the action
of the Lie bracket of any power \emph{preserves the grading}. Moreover,
for $m,\, l\in\mathbb{N}$, we have: 
\[
H_{m+l-1}=\mathrm{ad}_{x^{m}}(H_{l})\oplus G_{m+l-1},\quad G_{m+l-1}=\begin{cases}
\emptyset, & m\neq l,\\
H_{2m-1}, & m=l.
\end{cases}
\]
Here, the spaces $G_{k}$, $k\in\mathbb{N}$, are subspaces of $H_{k}$
that are not in the image of $\mathrm{ad}_{x^{m}}$ (consequently,
these terms cannot be eliminated by changes of variables).\\

Consider now a term $\psi=bx^{r}$ of the expansion of $f$, with
$r>p$. According to \eqref{eq:lb}, in order to remove this term,
we look for a change of variables $\varphi\left(x\right)=x+cx^{m}$
such that 
\begin{equation}
ac\left(p-m\right)x^{p+m-1}=-bx^{r}.\label{eq:homological-classical}
\end{equation}
This equation is classically called the \emph{homological equation}.
We find $m=r-p+1$, so $p-m=2p-r-1$. It can be solved if and only
if $r\ne2p-1$. So the term of degree $r=2p-1$ cannot be removed
from the expansion of $f$. In other words, the subspace $H_{2p-1}$
is not in the image of the Lie bracket action operator of the leading
monomial of $f$. In order to remove all possible monomials of $f$,
we proceed with a sequence of changes of variable of the previous
type. The normal form appears to be a formal limit of a sequence of
elements of $\mathbb{R}\left[\left[x\right]\right]$, and to have
the form: 
\[
f_{0}\left(x\right)=x+ax^{p}+\beta x^{2p-1},\hspace{1em}\beta\in\mathbb{R}.
\]
This procedure is an adaptation of a similar algorithm from \cite{takens2}
for reducing vector fields to their normal forms.

\subsection{Lie brackets in $\mathcal{L}$}

\label{lie-brackets-L} The general idea of the proof of Theorem A
is to mimic the former method for the elements of algebra $\mathcal{L}$.
However, because of the presence of logarithms in monomials, several
complications are to be expected. Let us first explain the action
of Lie brackets in our framework.\\

By $H_{\gamma,m}\subset\mathcal{L}$ we denote the one-dimensional
vector spaces spanned by the monomial $x^{\gamma}\ell^{m}$, $\gamma\geq1,\ m\in\mathbb{Z}$.
We introduce the \emph{grading of $H_{\gamma,m}$} by the order $(\gamma,m)$
of its monomials. Notice that, for $\left(\alpha,k\right)\in\left(0,\infty\right)\times\mathbb{Z}$,
we have 
\[
\left(x^{\alpha}\boldsymbol{\ell}^{k}\right)'=\alpha x^{\alpha-1}\boldsymbol{\ell}^{k}+kx^{\alpha-1}\boldsymbol{\ell}^{k+1}.
\]
Hence, the action of the Lie bracket, as defined in \eqref{pb}, of
a monomial $x^{\alpha}\boldsymbol{\ell}^{k}$ on a space $H_{\beta,l}$
is given by: 
\begin{align}
\mathrm{ad}_{x^{\alpha}\boldsymbol{\ell}^{k}}(cx^{\beta}\boldsymbol{\ell}^{l}) & =\left[cx^{\beta}\boldsymbol{\ell}^{l},x^{\alpha}\boldsymbol{\ell}^{k}\right]\nonumber \\
 & =c(\alpha-\beta)x^{\alpha+\beta-1}\boldsymbol{\ell}^{k+l}+c(l-k)x^{\alpha+\beta-1}\boldsymbol{\ell}^{k+l+1}.
\end{align}
We conclude that, on spaces $H_{\gamma,m}$, the action of the Lie
bracket of a power-log monomial does not preserve the \emph{grading},
as in power series case. Therefore, we introduce the appropriate quotient
spaces. By $K_{\gamma,m}^{0}$ and $K_{\gamma,m}$, we denote the
direct sums: 
\[
K_{\gamma,m}^{0}=\bigoplus_{(\gamma,m)\preceq(\gamma',m')}H_{\gamma',m'},\quad K_{\gamma,m}=\bigoplus_{(\gamma,m)\prec(\gamma',m')}H_{\gamma',m'}.
\]
Recall that the order $\preceq$ (\emph{resp}. $\prec$) is the lexicographic
order (\emph{resp.} strict lexicographic order ) on $\mathbb{R}\times\mathbb{Z}$.
We define the quotient spaces: 
\[
J_{\gamma,m}=\frac{K_{\gamma,m}^{0}}{K_{\gamma,m}}.
\]
Note that the quotient space $J_{\gamma,m}$ can be identified with
the vector space $H_{\gamma,m}$ of monomials of order $(\gamma,m)$.
The grading of $J_{\gamma,m}$ is given by the order $(\gamma,m)$
of any representative. Based on these remarks, we can state the next
proposition which claims that the grading is preserved on quotient
spaces $J_{\gamma,m}$. 
\begin{prop}[Action of the Lie bracket operator on quotient spaces $J_{\gamma,m}$]
\label{prop1} Let 
\[
T_{\alpha,k}\in J_{\alpha,k},\ (\alpha,k)\in\mathbb{R}_{>0}\times\mathbb{Z},\ (1,0)\prec(\alpha,k),
\]
be an element of the class $J_{\alpha,k}$ of the monomial $x^{\alpha}\boldsymbol{\ell}^{k}$.
Let $(\gamma,m)\in\mathbb{R}_{>0}\times\mathbb{Z},\ (1,0)\prec(\gamma,m)$.
The operator $\mathrm{ad}{}_{T_{\alpha,k}}$ acts on the quotient
space $J_{\gamma,m}$ by the following rule: 
\begin{align}
\begin{cases}
 & J_{\alpha+\gamma-1,k+m}=\mathrm{ad}{}_{T_{\alpha,k}}(J_{\gamma,m}),\quad\gamma\neq\alpha,\\
 & J_{2\alpha-1,k+m+1}=\mathrm{ad}{}_{T_{\alpha,k}}(J_{\alpha,m})\oplus G_{2\alpha-1,k+l+1},
\end{cases}\label{taa}
\end{align}
where 
\[
G_{2\alpha-1,k+m+1}=\begin{cases}
\emptyset, & \ m\neq k,\\
J_{2\alpha-1,2k+1}, & \ m=k.
\end{cases}
\]

\end{prop}
Obviously, this different behavior of the action of the Lie bracket
compared to its behavior in $\mathbb{R}\left[\left[x\right]\right]$
will induce a different treatment of the homological equation. These
aspects are examined in details in the next section, where we give
the precise form and the proof of Theorem A.

\section{Proof of Theorem A}

\label{sec:proof-theorem-A}

In this Section we construct changes of variables that transform a
transseries from $\mathcal{L}^{H}$ or $\mathcal{L}_{\mathfrak{D}}^{H}$
to its formal normal form. These changes of variables will be obtained
via \emph{transfinite compositions} of elementary changes of variables.
This important difference with the classical case comes from the fact
that the supports of the elements of $\mathcal{L}$ are not any more
contained in the set of positive integers, but are well-ordered subsets
of $\left(0,\infty\right)\times\mathbb{Z}$. In order to define properly
the notion of a transfinite composition, we recall (for a non-specialized
reader) in the next section a few well known facts about well-ordered
sets and about transfinite sequences.

\subsection{\label{sub:well-ordered-ordinals}Basic properties of ordinal numbers,
well-ordered sets and transfinite sequences}

A set is an \emph{ordinal number} (or an \emph{ordinal} for short)
if it is transitive and well-ordered by $\in$. Recall that a set
$X$ is called \emph{transitive} if every element of $X$ is also
a subset of $X$. Usually, the class of all ordinals is denoted by
$\text{\textbf{On}}$. It is totally ordered (moreover, well-ordered)
by the relation: $\alpha<\beta$ if and only if $\alpha\in\beta$.
Recall the \emph{von Neumann} construction of the class $\text{\textbf{On}}$.
The empty set is the \emph{smallest} ordinal, denoted by $0$. Every
ordinal $\alpha$ coincides with the set of all ordinals smaller than
$\alpha$, that is $\alpha=\left\{ \beta\in\mathbf{On}\colon\beta<\alpha\right\} $.

\noindent There are two \emph{types} of ordinals:

(1) \emph{The successor ordinal}: The \emph{successor} of an ordinal
$\alpha$, denoted by $\alpha+1$, is the ordinal $\alpha\cup\left\{ \alpha\right\} $.

(2) \emph{The limit ordinal}: If $\alpha$ is not a successor ordinal,
then $\alpha=\sup\left\{ \beta\colon\beta<\alpha\right\} $. Such
$\alpha$ is called a \emph{limit ordinal}.

\noindent The smallest limit ordinal is the set of non-negative integers,
usually denoted by $\omega$.\\

The classical principle of induction is generalized by the following
principle, called the \emph{principle of transfinite induction}. Consider
a class $C$ of ordinals, such that: 
\begin{enumerate}[1., font=\textup, nolistsep, leftmargin=0.6cm]
\item $0\in C$; 
\item if $\alpha\in C$ then $\alpha+1\in C$ (\emph{non-limit case}); 
\item if $\alpha$ is a nonzero limit ordinal and $\beta\in C$ for all
$\beta<\alpha$, then $\alpha\in C$ (\emph{limit case}). 
\end{enumerate}
Then $C$ is the class $\text{\textbf{On}}$ of \emph{all} ordinals.\\

Consider now a set $A$. A \emph{transfinite sequence} (or \emph{$\theta$-sequence})
of elements of $A$ is a function that takes values in $A$ and whose
domain is an ordinal $\theta\in\mathbf{On}$. We denote such sequence
by $(a_{\beta})_{\beta<\theta}$, $a_{\beta}\in A$. Suppose that
$A$ is a topological space. We say that the $\theta$-sequence $\left\{ a_{\beta}\colon\beta<\theta\right\} $
of elements of $A$ \emph{converges to} $a\in A$ when $\beta$ goes
to $\theta$ if, for every neighborhood $U$ of $a$, there exists
an ordinal $\beta_{0}<\theta$ such that $a_{\beta}\in U$ for all
$\beta$ such that $\beta_{0}<\beta<\theta$. We put $a:=\lim_{\beta\rightarrow\theta}a_{\beta}$
or $a:=\lim a_{\beta}$ for short.\\

Recall that two totally ordered sets $\left(P,<\right)$ and $\left(Q,<\right)$
are called \emph{isomorphic} if there exists an order-preserving one-to-one
function $f\colon P\rightarrow Q$. Finally, the strong connection
between well-ordered sets and ordinals is established by the following
result: \emph{every well-ordered set is isomorphic to a unique ordinal
number.} This ordinal number will be called \emph{the order type}
of the well-ordered set. It implies that the elements of a well-ordered
set can be enumerated as an increasing $\theta$-sequence (transfinite
sequence). The ordinal $\theta$ is then its order type. Also conversely:
the elements of a well-ordered set can be used as the indices of a
transfinite sequence. In particular, given a well-ordered set $W$
and a sequence $\left(a_{w}\right)_{w\in W}$ of elements of a topological
space $A$, we say that $\left(a_{w}\right)$ converges to $a\in A$
if, for every neighborhood $U$ of $a$, there exists $w_{0}\in W$
such that, for every $w\in W$, $w_{0}<w$, it holds that $a_{w}\in U$.
We denote this limit by $\lim_{w\in W}a_{w}.$

Notice that, due to the density of the set of rational numbers in
$\mathbb{R}$, every well-ordered subset of $\mathbb{R}$ or of $\mathbb{R}\times\mathbb{Z}$
is countable.\\

In the sequel, we build transfinite sequences of elements of $\mathcal{L}$
algorithmically, and we study their convergence in $\mathcal{L}$.

\subsection{\label{sub:transfinite-sequences-elements-L}Transfinite sequences
of elements of $\mathcal{L}$}

\label{sec:topologies}

In order to study convergence of (transfinite) sequences of elements
of $\mathcal{L}$, we endow $\mathcal{L}$ with the following topologies,
introduced in order of the decreasing strength.\\

1. \emph{The formal topology} on $\mathcal{L}$. Consider $f\in\mathcal{L}$
and $\left(\alpha,k\right)\in\mathbb{R}_{>0}\times\mathbb{Z}$. Then
the (open) \emph{ball} $B\left(f,\left(\alpha,k\right)\right)$ centered
at $f$ is the set 
\[
\left\{ g\in\mathcal{L}\colon\text{ord}\left(g-f\right)\succ\left(\alpha,k\right)\right\} .
\]
Given two different balls $B_{1}$ and $B_{2}$ centered at $f\in\mathcal{L}$,
either $B_{1}\subset B_{2}$ or $B_{2}\subset B_{1}$. Hence, the
collection of balls centered at $f$ form a \emph{fundamental system
of neighborhoods}. The family of all balls generates a Hausdorff topology
on $\mathcal{L}$. 

Consider now an ordinal $\theta$ and a transfinite sequence $\left(f_{\mu}\right)_{\mu<\theta}$
of elements of $\mathcal{L}$. Then the sequence $\left(f_{\mu}\right)$
\emph{converges} to $f\in\mathcal{L}$ when $\mu$ goes to $\theta$
in the formal topology if, for every $\left(\alpha,k\right)\in\mathbb{R}_{>0}\times\mathbb{Z}$,
there exists an ordinal $\mu_{0}<\theta$ such that $\mathrm{ord}\left(f-f_{\mu}\right)\succ\left(\alpha,k\right)$
for every $\mu_{0}<\mu<\theta$.\\

2. \emph{The product topology on $\mathcal{L}$ with respect to the
discrete topology on $\mathbb{R}$.} Let us endow $\mathbb{R}$ with
the discrete topology, and the product $\mathbb{R}^{\mathbb{R}_{>0}\times\mathbb{Z}}$
with the product topology. Each transseries $f\in\mathcal{L}$ is
understood as a function $f:\mathbb{R}_{>0}\times\mathbb{Z}\to\mathbb{R}$,
which assigns to each pair $(\alpha,k)$ the coefficient of the monomial
$x^{\alpha}\boldsymbol{\ell}^{k}$ in $f$. We will denote that coefficient
by $[f]_{\alpha,k}$. Hence, we can consider $\mathcal{L}$ as a subspace
of $\mathbb{R}^{\mathbb{R}_{>0}\times\mathbb{Z}}$, equipped with
the induced topology. 

Let $(f_{\mu})_{\mu<\theta}$ be a transfinite sequence of elements
from $\mathcal{L}$. In this product topology, the sequence $(f_{\mu})$
converges to $f\in\mathcal{L}$ when $\mu\to\theta$ if, for every
$\left(\alpha,k\right)\in\mathbb{R}_{>0}\times\mathbb{Z}$, there
exists an ordinal $\mu_{0}<\theta$ such that the coefficient $[f_{\mu}]_{\alpha,k}$
equals the coefficient $[f]_{\alpha,k}$, for every $\mu_{0}<\mu<\theta$.\\

3. \emph{The weak topology on $\mathcal{L}$ (i.e. the product topology
with respect to the Euclidean topology on $\mathbb{R}$)}. The topology
is similar to the one described in 2. The only difference is that
we endow $\mathbb{R}$ with the Euclidean topology instead of the
discrete one. The sequence $(f_{\mu})$ converges to $f\in\mathcal{L}$
when $\mu\to\theta$ in the weak topology if, for every $\left(\alpha,k\right)\in\mathbb{R}_{>0}\times\mathbb{Z}$
and $\varepsilon>0$, there exists an ordinal $\mu_{0}<\theta$ such
that $\big[f-f_{\mu}\big]_{\alpha,k}\in(-\varepsilon,\varepsilon)$,
for every $\mu_{0}<\mu<\theta$. \\

The three topologies introduced above will be used in this work. We
need the \emph{product topology with respect to the discrete topology}
in the proof of Theorem A. In the proof of Theorem B, depending on
the type of elements of $\mathcal{L}$ considered (parabolic or hyperbolic),
we use \emph{formal} or \emph{weak topology}.
\begin{rem}
As has already been mentioned, the above topologies are ordered by
their strength. For example, the sequence $(f_{n})_{n\in\mathbb{N}}\in\mathcal{L}$,
\[
f_{n}(x)=x^{2-\frac{1}{n}},
\]
converges to $f\equiv0$ in the product topology with respect to the
discrete topology, but not in the formal topology.

Likewise, the sequence 
\[
f_{n}(x)=\frac{1}{n}x,
\]
converges to $f\equiv0$ in the weak topology, but not in the product
topology with respect to the discrete topology nor in the formal topology.
\end{rem}
\bigskip{}
In all three cases, we set $f=\lim_{\mu\to\theta}f_{\mu}$, with an
indication of the topology to which we refer. From now on, unless
explicitely stated otherwise, we endow $\mathcal{L}$ with \emph{the
product topology with respect to the discrete topology}, so every
limit or convergence to be mentioned in the sequel is understood with
respect to this topology.
\begin{rem}
\label{rem:two-types-limit}Given a well-ordered subset $W$ of $\mathbb{R}_{>0}\times\mathbb{Z}$,
we can define in the same way, if it exists, the limit $f=\lim_{\left(\alpha,k\right)\in W}f_{\alpha,k}$
of a transfinite sequence $\left(f_{\alpha,k}\right)$ of elements
of $\mathcal{L}$. In the rest of this article, we will deal indifferently
with sequences indexed by ordinals or by elements of a well-ordered
subset of $\mathbb{R}_{>0}\times\mathbb{Z}$. 
\end{rem}
We define the \emph{elementary changes of variables} in $\mathcal{L}^{0}$
by: 
\begin{equation}
\begin{cases}
\varphi_{1,0}(x)=ax, & a\in\mathbb{R},\ a>0,\ a\ne1,\\
\varphi_{1,m}(x)=x+cx\boldsymbol{\ell}^{m}, & m\in\mathbb{N},\ m\neq0,\ c\in\mathbb{R},\\
\varphi_{\beta,m}(x)=x+cx^{\beta}\boldsymbol{\ell}^{m}, & \beta>1,\ m\in\mathbb{Z},\ c\in\mathbb{R}.
\end{cases}\label{eq:elementary-change}
\end{equation}
Notice that $\text{ord}(\varphi_{\beta,m}-\text{id})=(\beta,m)$.

We now define the notion of a composition of a transfinite sequence
in $\mathcal{L}$. We will apply this notion to transfinite compositions
of elementary changes of variables in $\mathcal{L}^{0}$.
\begin{defn}
\label{def:transfinite-composition}Consider an ordinal $\theta$
and a transfinite sequence $\left(\varphi_{\mu}\right)_{\mu<\theta}$
of elements from $\mathcal{L}^{0}$. We say that the \emph{transfinite
composition} $\circ_{\mu<\theta}\,\varphi_{\mu}$ \emph{exists} and
is equal to $\varphi\in\mathcal{L}^{0}$ if the following conditions
are satisfied: 
\begin{enumerate}[1., font=\textup, nolistsep, leftmargin=0.6cm]
\item We can define a sequence $\left(\psi_{\mu}\right)_{\mu<\theta}$
of elements of $\mathcal{L}^{0}$ (which we call the \emph{partial
compositions}) in the following way:

\begin{enumerate}[(a), font=\textup, nolistsep, leftmargin=0.6cm]
\item $\psi_{0}:=\mathrm{id}$; 
\item If $\mu=\nu+1$ is a \emph{successor ordinal}, then $\psi_{\nu+1}:=\varphi_{\nu}\circ\psi_{\nu}$
(\emph{non-limit case}); 
\item If $\mu<\theta$ is a \emph{limit ordinal}, the sequence $\left(\psi_{\nu}\right)_{\nu<\mu}$
converges to $\psi_{\mu}\in\mathcal{L}^{0}$ when $\nu$ goes to $\mu$
(\emph{limit case}). 
\end{enumerate}
\item The sequence $\left(\psi_{\mu}\right)_{\mu<\theta}$ converges to
$\varphi\in\mathcal{L}^{0}$ when $\mu$ goes to $\theta$. 
\end{enumerate}
We write: $\varphi=\circ_{\mu<\theta}\,\varphi_{\mu}$.\end{defn}
\begin{prop}[Convergence of \emph{partial normal forms} $(f_{\mu})_{\mu<\theta}\in\mathcal{L}^{H}$]
\label{prop:changement-transfini} Let $f\in\mathcal{L}^{H}$. Let
$\left(\varphi_{\mu}\right)_{\mu<\theta}$, $\varphi_{\mu}\in\mathcal{L}^{0},$
be a transfinite sequence of changes of variables such that the composition
$\psi=\circ_{\mu<\theta}\,\varphi_{\mu}$ exists in $\mathcal{L}^{0}$
(as the limit of the transfinite sequence $\left(\psi_{\mu}\right)_{\mu<\theta}$
introduced in the former definition). Let $\left(f_{\mu}\right)_{\mu<\theta}$
be a transfinite sequence in $\mathcal{L}^{H}$, defined by: 
\[
f_{\mu}:=\psi_{\mu}^{-1}\circ f\circ\psi_{\mu},\ \mu\leq\theta,\ \text{with }\psi_{\theta}:=\psi.
\]
Then $f_{\mu}\to f_{\theta}$, as $\mu\to\theta$. 
\end{prop}
In the proof, we use the following auxiliary lemma. Consider a topology
$\mathcal{T}$ on $\mathcal{L}$. We say that an application $F\colon\mathcal{L}^{H}\rightarrow\mathcal{L}^{H}$
is \emph{transfinitely sequentially continuous with respect to $\mathcal{T}$}
if, for every transfinite sequence $\left(g_{\mu}\right)_{\mu<\theta}$
in $\mathcal{L}^{H}$ \emph{such that the supports of all the $g_{\mu}$
are contained in a well-ordered subset of $\mathbb{R}_{>0}\times\mathbb{Z}$,}
and such that $g_{\mu}\rightarrow g$ with respect to $\mathcal{T}$,
then $F\left(g_{\mu}\right)\rightarrow F\left(g\right)$ with respect
to $\mathcal{T}$.
\begin{lem}[\emph{Transfinite sequential continuity}]
\emph{}\label{lem:conti}Assume $\mathcal{L}$ equipped with the
product topology (the discrete case).\textup{ }
\begin{enumerate}[1., font=\textup, nolistsep, leftmargin=0.6cm]
\item Let $h\in\mathcal{L}^{H}$. The applications defined on $\mathcal{L}^{0}$
by
\[
\text{\emph{(i)} }g\longmapsto g\circ h,\text{ }g\longmapsto h\circ g,\quad\text{\emph{(ii)} }g\longmapsto g^{-1}
\]
are transfinitely sequentially continuous.
\item Consider two transfinite sequences $(h_{\mu})_{\mu<\theta}$ in $\mathcal{L}^{H}$,
and $\left(g_{\mu}\right)_{\mu<\theta}$ in $\mathcal{L}^{0}$, such
that $h_{\mu}\to0$ as $\mu\to\theta$ and the supports of all the
$h_{\mu}$ and $g_{\mu}$ are contained in a common well-ordered subset
of $\mathbb{R}_{>0}\times\mathbb{Z}$. Then $h_{\mu}\circ g_{\mu}\to0$
as $\mu\to\theta$. 
\end{enumerate}
\end{lem}
\begin{proof}
All these statements can be proven by analyzing the supports of the
composition and of the inverse as in \eqref{eq:composition-by-phi-1}
and applying Neumann's Lemma \ref{lem:neumann_lemma}.3. Concluding
similarly as in \eqref{eq:composition-by-phi-1}, for two transseries
$g,\, h\in\mathcal{L}^{H}$, such that $\text{ord}(g)=(\alpha_{0},0),$
we obtain: 
\[
\mathcal{S}(h\circ g)\subset\mathcal{S}(h)\cup H.
\]
Here, $H$ is a sub-semigroup of $\mathbb{R}_{\geq0}\times\mathbb{Z}$
generated by $(\alpha_{0}\beta,k)$ for $(\beta,k)\in\mathcal{S}(h)$,
$(\alpha-\alpha_{0},m)$ for $(\alpha,m)\in\mathcal{S}(g)$ and $(0,1)$.
Moreover, every coefficient of the composition is a sum of \emph{finitely
many} finite products of coefficients of $h$ and $g$ by Neumann's
lemma. Additionally, each product contains exactly one coefficient
from $h$ among its factors. (1) (i) and (2) follow.

To prove (1) (ii), due to (1), it suffices to prove the easier statement:
if $g_{\mu}\to\mathrm{id}$ as $\mu\to\theta,$ then $g_{\mu}^{-1}\to\mathrm{id}$,
as $\mu\to\theta$. It can be checked that the coefficients of $g_{\mu}^{-1}-\mathrm{id}$
are sums of finitely many finite products of coefficients of $g_{\mu}-\mathrm{id}$,
which eventually vanish, by Neumann's lemma. Therefore, the coefficients
of $g_{\mu}^{-1}-\mathrm{id}$ eventually vanish.
\end{proof}

\begin{proof}[Proof of Proposition~\ref{prop:changement-transfini}.]
\emph{} Let $f,\ (f_{\mu})_{\mu<\theta}\in\mathcal{L}^{H}$ be as
defined in the proposition. Knowing that $\psi_{\mu}\to\psi$, we
prove that $\psi_{\mu}^{-1}\circ f\circ\psi_{\mu}\to\psi^{-1}\circ f\circ\psi$,
as $\mu\to\theta$ (in the product topology, the discrete case).

Since $\psi_{\mu}\to\psi$, it follows by Lemma~\ref{lem:conti}
(1)$(i)$ that $\psi_{\mu}\circ\psi^{-1}\to\mathrm{id}$, and further
by (1)$(ii)$ that $\psi\circ\psi_{\mu}^{-1}\to\mathrm{id}$. By (1)$(i)$,
$f\circ\psi\circ\psi_{\mu}^{-1}\to f\circ\mathrm{id}=f$ and $\psi\circ\psi_{\mu}^{-1}\circ f\to\mathrm{id}\circ f=f$.
Therefore, 
\begin{align*}
f\circ\psi\circ\psi_{\mu}^{-1}-\psi\circ\psi_{\mu}^{-1}\circ f & \to0,\\
\stackrel{(2),\,\psi_{\mu}\to\psi}{\Longrightarrow}\, f\circ\psi-\psi\circ\psi_{\mu}^{-1}\circ f\circ\psi_{\mu} & \to0,\\
\psi\circ\psi_{\mu}^{-1}\circ f\circ\psi_{\mu} & \to f\circ\psi,\\
\psi_{\mu}^{-1}\circ f\circ\psi_{\mu} & \to\psi^{-1}\circ f\circ\psi.
\end{align*}

\end{proof}
As we did above for a composition of a transfinite sequence of elements
of $\mathcal{L}^{0}$, we can define in particular, if it exists,
a composition of a sequence $\left(\varphi_{\beta,m}\right)$ of elementary
changes of variables indexed by elements of a well-ordered subset
$W\subset\mathbb{R}_{>0}\times\mathbb{Z}$, via the sequence $\left(\psi_{\beta,m}\right)$
of partial compositions. Again, we have to consider non-limit cases
and limit cases.

The following proposition gives an important characterization of elements
of $\mathcal{L}^{0}$ or $\mathcal{L}_{\mathfrak{D}}^{0}$ in terms
of transfinite compositions of elementary changes of variables. 
\begin{prop}[Characterization of changes of variables in $\mathcal{L}^{0}$ or
$\mathcal{L}_{\mathfrak{D}}^{0}$]
 \label{prop:characterization-changes-variables}

Let $\mathcal{L}^{0}$ be endowed with the product topology with respect
to the discrete topology.
\begin{enumerate}[1., font=\textup, nolistsep, leftmargin=0.6cm]
\item Let $W\subset\mathbb{R}_{>0}\times\mathbb{Z}$ be well-ordered. Let
$\left(\varphi_{\alpha,m}\right)_{\left(\alpha,m\right)\in W}$ be
a transfinite sequence of elementary changes of variables, such that
the sequence of orders $\mathrm{ord}(\varphi_{\alpha,m}-\text{id})=\left(\alpha,m\right)$
is strictly increasing. Then the transfinite composition $\varphi=\circ_{\left(\alpha,m\right)\in W}\varphi_{\alpha,m}$
is well defined in $\mathcal{L}^{0}$. Moreover, if $W\subset\mathbb{R}_{>0}\times\mathbb{Z}$
is of finite type, then $\varphi\in\mathcal{L}_{\mathfrak{D}}^{0}$. 
\item For every transseries $\varphi\in\mathcal{L}^{0}$ (\emph{resp.} $\varphi\in\mathcal{L}_{\mathfrak{D}}^{0}$)
there exist a well-ordered subset (\emph{resp.} a subset of finite
type) $W\subset\mathbb{R}_{>0}\times\mathbb{Z}$ and a transfinite
sequence $\left(\varphi_{\alpha,m}\right)_{\left(\alpha,m\right)\in W}$
of elementary changes of variables such that $\varphi=\circ_{\left(\alpha,m\right)\in W}\varphi_{\alpha,m}$. 
\end{enumerate}
\end{prop}
\begin{proof}
(1) We first give a preliminary computation which describes the change
in the support of an element of $\mathcal{L}^{0}$ after composition
with an elementary change of variables, and prove implicitely that
the composition remains in $\mathcal{L}^{0}$. Let $h=\text{id}+\varepsilon\in\mathcal{L}^{0}$
(i.e. $\text{ord}(\varepsilon)\succ\left(1,0\right)$). Consider an
elementary change of variables $\varphi_{\beta,\ell}\left(x\right)=x+cx^{\beta}\boldsymbol{\ell}^{\ell}$,
$c\in\mathbb{R}$, $(\beta,\ell)\succ(1,0)$. A straightforward computation
shows that, for every integer $p\ge1$, the support of the $p$-th
derivative $\left(x^{\beta}\boldsymbol{\ell}^{\ell}\right)^{\left(p\right)}$
is contained in the set $\left\{ \left(\beta-p,\ell\right),\left(\beta-p,\ell+1\right),\ldots,\left(\beta-p,\ell+p\right)\right\} $.
Hence, it follows from Taylor formula that 
\begin{align}
\left(\varphi_{\beta,\ell}\circ h\right)\left(x\right) & =\left(x+cx^{\beta}\boldsymbol{\ell}^{\ell}\right)\circ\left(x+\varepsilon\left(x\right)\right)\nonumber \\
 & =h(x)+cx^{\beta}\boldsymbol{\ell}^{\ell}+\sum_{p=1}^{\infty}\sum_{j_{p}=0}^{p}b_{p,j_{p}}x^{\beta-p}\boldsymbol{\ell}^{\ell+j_{p}}\varepsilon\left(x\right)^{p},\quad b_{p,j_{p}}\in\mathbb{R}.\label{eq:composition-by-phi-1}
\end{align}
Notice that, for each $p\ge1$, every element $\left(\gamma,r\right)$
of the support of $\varepsilon\left(x\right)^{p}$ has the form 
\[
\left(\gamma,r\right)=\left(\alpha_{i_{1}}+\cdots+\alpha_{i_{p}},k_{i_{1}}+\cdots+k_{i_{p}}\right),
\]
where the exponents $\left(\alpha_{i_{s}},k_{i_{s}}\right)$, $s=1,\ldots,p$,
belong to $\mathcal{S}\left(\varepsilon\right)$. Hence, every element
of the support of the double sum in the formula \eqref{eq:composition-by-phi-1}
has the form 
\begin{equation}
\left(\beta,\ell\right)+\left(\alpha_{i_{1}}-1+\cdots+\alpha_{i_{p}}-1,k_{i_{1}}+\cdots+k_{i_{p}}\right)+\left(0,j_{i_{p}}\right),\label{eq:suppor-composition-by-monomial}
\end{equation}
where $p\ge1$, $\left(\alpha_{i_{s}},k_{i_{s}}\right)\in\mathcal{S}\left(h\right)$
and $j_{i_{p}}\in\left\{ 0,\ldots,p\right\} $.\\

Two main facts follow from this computation. Denote by $H$ the (additive)
sub-semigroup of $\mathbb{R}_{\geq0}\times\mathbb{Z}$ generated by
$\left(\beta,\ell\right)$, the elements $\left(\alpha-1,k\right)$
for $\left(\alpha,k\right)\in\mathcal{S}\left(h\right)$ and $\left(0,1\right)$.
First, by \eqref{eq:composition-by-phi-1}, the support $\mathcal{S}\left(\varphi_{\beta,\ell}\circ h\right)$
of the composition is contained in the union $\mathcal{S}\left(h\right)\cup H$.
Since $\mathcal{S}\left(h\right)$ is well-ordered, and since, by
Neumann's Lemma \ref{lem:neumann_lemma}, $H$ is well-ordered, the
support $\mathcal{S}\left(\varphi_{\beta,\ell}\circ h\right)$ is
also well-ordered. Second, by Neumann's Lemma \ref{lem:neumann_lemma},
the composition $\varphi_{\beta,\ell}\circ h$ is well-defined, meaning
that every monomial in the support of $\varphi_{\beta,\ell}\circ h$
has a well-defined coefficient. This in particular implies that $\varphi_{\beta,\ell}\circ h\in\mathcal{L}^{0}$.
Now, assume additionally that $h\in\mathcal{L}_{\mathfrak{D}}^{0}$,
so $\mathcal{S}\left(h\right)$ is contained in a (additive) sub-semigroup
of $\mathbb{R}_{>0}\times\mathbb{Z}$ generated by finitely many elements
$\left(\gamma_{1},p_{1}\right),\ldots,\left(\gamma_{r},p_{r}\right)$
of $\mathbb{R}_{>0}\times\mathbb{Z}$. For $n\in\mathbb{N}$, we can
write $\left(n\gamma_{i}-1,np_{i}\right)=\left(n-1\right)\left(\gamma_{i},p_{i}\right)+\left(\gamma_{i}-1,p_{i}\right)$.
Hence, $\mathcal{S}\left(\varphi_{\beta,\ell}\circ h\right)$ is contained
in the (additive) sub-semigroup of $\mathbb{R}_{\geq0}\times\mathbb{Z}$
generated by $\left(\beta,\ell\right)$, $\left(0,1\right)$ and the
elements $\left(\gamma_{i},p_{i}\right)$, $\left(\gamma_{i}-1,p_{i}\right)$
for $i=1,\ldots,r$. In particular, $\varphi_{\beta,\ell}\circ h\in\mathcal{L}_{\mathfrak{D}}^{0}$.\\

Consider now the sequence $\left(\varphi_{\alpha,m}\right)_{\left(\alpha,m\right)\in W}$
given in the statement of the proposition. We prove the existence
of the composition $\circ_{(\alpha,m)\in W}\,\varphi_{\alpha,m}$
by transfinite induction. Let $(\alpha_{0},m_{0})$ be the smallest
element of $W$. Put $\psi_{\alpha_{0},m_{0}}:=\varphi_{(\alpha_{0},m_{0})}\in\mathcal{L}^{0}$.
Consider the sub-semigroup $\overline{W}\subset\mathbb{R}_{\geq0}\times\mathbb{Z}$
generated by the elements of $W$, the elements $\left(\alpha-1,p\right)$
for $\left(\alpha,p\right)\in W$, and $\left(0,1\right)$. We already
know that $\overline{W}$ is well-ordered, and of finite type if $W$
is of finite type.

Existence of partial compositions in $\mathcal{L}^{0}$ in the \emph{non-limit
case} follows directly by the above considerations. Moreover, the
support of the partial compositions is contained in $\overline{W}$.
Consider the \emph{limit case}. Suppose $(\alpha_{\theta},m_{\theta})$
is a limit ordinal (or the order type of $W$), and for every $(\beta,l)\in W,\ (\beta,l)\prec(\alpha_{\theta},m_{\theta})$,
it holds that $\psi_{\beta,l}\in\mathcal{L}^{0}$ and $\mathcal{S}(\psi_{\beta,l})\subset\overline{W}$.
We prove that $\psi_{\beta,l}$ converge in $\mathcal{L}^{0}$ in
the product topology, as $(\beta,l)\to(\alpha_{\theta},m_{\theta})$,
and that the support of the limit belongs to $\overline{W}$.

By \eqref{eq:composition-by-phi-1}, $\mathcal{S}(\psi)\subset\mathcal{S}(\varphi_{\beta,l}\circ\psi)\subseteq\overline{W}$
for every partial sum $\psi$ and every change of variables $\varphi_{\beta,l}$,
$(\beta,l)\in W$. Thus, if $(\gamma,k)\in\mathcal{S}(\psi_{\beta,l})$,
then $(\gamma,k)\in\mathcal{S}(\psi_{\alpha,m})$, for all $(\beta,l)\prec(\alpha,m)\prec(\alpha_{\theta},m_{\theta})$.
To prove the convergence in the product topology, we have to prove
that the coefficient of monomial $x^{\gamma}\ell^{k}$ eventually
stabilizes in the sequence of partial sums $(\psi_{\beta,l})_{(\beta,l)\prec(\alpha_{\theta},m_{\theta})}$.
Since $\left(\beta,l\right)$ is a summand of \eqref{eq:suppor-composition-by-monomial}
and the sequence $\left(\beta,l\right)\in W$ is strictly increasing,
it follows from Neumann's Lemma \ref{lem:neumann_lemma}.3. that each
$\left(\gamma,k\right)\in\overline{W}$ is realized at most finitely
many times as sum of the type \eqref{eq:suppor-composition-by-monomial}.
That is, the coefficient of monomial $x^{\gamma}\ell^{k}$ in the
support of $(\psi_{\beta,l})$ changes only \emph{finitely many} times
in the course of compositions $\circ_{(\beta,l)\prec(\alpha_{\theta},k_{\theta})}\varphi_{\beta,l}$.
This guarantees the convergence in $\mathcal{L}^{0}$ of partial compositions
in the limit case, for the product topology. The limit is the partial
composition for the limit ordinal $(\alpha_{\theta},m_{\theta})$:
\[
\psi_{\alpha_{\theta},m_{\theta}}=\circ_{(\beta,l)\prec(\alpha_{\theta},m_{\theta})}\,\varphi_{\beta,l}:=\lim_{(\beta,l)\to(\alpha_{\theta},m_{\theta})}\psi_{\beta,l},
\]
with $\mathcal{S}(\psi_{\alpha_{\theta},m_{\theta}})\subseteq\overline{W}$
by construction. \\

$(2)$ Let $\varphi\in\mathcal{L}^{0}$, $\varphi(x)=ax+\mathrm{h.o.t},$
$a\in\mathbb{R}$. Obviously, $\varphi(x)=ax\circ\varphi_{0}(x)$,
where $\varphi_{0}(x)=x+ax^{\alpha_{0}}\boldsymbol{\ell}^{m_{0}}+\mathrm{h.o.t.}$
tangent to the identity. By Neumann's lemma \ref{lem:neumann_lemma},
the sub-semigroup $W$ of $\mathbb{R}_{>0}\times\mathbb{Z}$ generated
by $\mathcal{S}(\varphi-a\cdot\mathrm{id})$ is well-ordered, with
$(\alpha_{0},m_{0})$ its smallest element. We prove that $\varphi_{0}$
can be decomposed in a transfinite composition of elementary changes
of variables $\left(\varphi_{\beta,m}\right)_{\left(\beta,m\right)\in V}$,
for some $V\subseteq W$. More precisely: we build, by transfinite
induction, a sequence $\left(\varphi_{\beta,m}\right)_{(\beta,m)\in V}$,
$V\subseteq W$, of elementary changes of variables, such that, for
every $\left(\alpha,k\right)\in W$, there exists $\left(\beta(\alpha,k),m(\alpha,k)\right)\in V$
and a partial composition 
\[
\psi_{\beta(\alpha,k),m(\alpha,k)}=\circ_{(\beta,m)\prec\left(\beta(\alpha,k),m(\alpha,k)\right)}\varphi_{(\beta,m)}
\]
with $\text{ord}\left(\varphi-\psi_{\beta(\alpha,k),m(\alpha,k)}\right)\succ\left(\alpha,k\right)$.
Moreover, the function $(\alpha,m)\in W\mapsto\left(\beta(\alpha,k),m(\alpha,k)\right)\in V$
is increasing, but not necessarily strictly. Since $W$ contains arbitrarily
big elements with respect to the order topology, it means that the
sequence $\left(\psi_{\beta,m}\right)_{\left(\beta,m\right)\in V}$
converges towards $\varphi$ in the formal topology. In particular,
it converges towards $\varphi$ in the product topology with respect
to the discrete topology.

Put $\psi_{\alpha_{0},m_{0}}(x):=\varphi_{\alpha_{0},m_{0}}(x)=x+ax^{\alpha_{0}}\ell^{m_{0}}$.
Consider $\left(\alpha,k\right)\in W$. By the induction hypothesis,
for all $\left(\gamma,r\right)\in W$, $\left(\gamma,r\right)\prec\left(\alpha,k\right)$,
there exists a transfinite composition $\psi_{\beta(\gamma,r),m(\gamma,r)}\in\mathcal{L}^{0}$,
such that $\left(\gamma,r\right)\prec\text{ord}(\varphi-\psi_{\beta(\gamma,r),m(\gamma,r)})$.
We prove that then there exists a transfinite composition $\psi_{\beta(\alpha,k),m(\alpha,k)}$,
such that $\left(\alpha,k\right)\prec\text{ord}(\varphi-\psi_{\beta(\alpha,k),m(\alpha,k)})$.\\

In the \emph{non-limit case}, consider the predecessor $\left(\alpha',k'\right)$
of $\left(\alpha,k\right)$ in $W$, and the partial composition $\psi_{\beta(\alpha',k'),m(\alpha',k')}\in\mathcal{L}^{0}$.
Then either $\psi_{\beta(\alpha',k'),m(\alpha',k')}=\psi_{\beta(\alpha,k),m(\alpha,k)}$,
or there exists an elementary change of variables $\varphi_{\alpha,k}(x)=x+cx^{\alpha}\boldsymbol{\ell}^{k}$,
with $c\in\mathbb{R}$ such that the term of order $\left(\alpha,k\right)$
from $\varphi-\psi_{\beta(\alpha',k'),m(\alpha',k')}$ is cancelled
in $\varphi-\psi_{\beta(\alpha,k),m(\alpha,k)}$: 
\[
\varphi-\psi_{\beta(\alpha,k),m(\alpha,k)}=\varphi-\varphi_{\alpha,k}\circ\psi_{\beta(\alpha',k'),m(\alpha',k')}=(\varphi-\psi_{\beta(\alpha',k'),m(\alpha',k')})-cx^{\alpha}\boldsymbol{\ell}^{k}+\text{h.o.t.}
\]
Obviously, $\psi_{\beta(\alpha,k),m(\alpha,k)}=\varphi_{\alpha,k}\circ\psi_{\beta(\alpha',k'),m(\alpha',k')}\in\mathcal{L}^{0}$.
Hence, the claim is proved in the non-limit case.

In the \emph{limit case}, when $(\alpha,k)$ is a limit ordinal, we
put $\psi_{\beta(\alpha,k),m(\alpha,k)}=\lim_{\left(\gamma,r\right)\prec(\alpha,k)}\psi_{\beta(\gamma,r),m(\gamma,r)}$,
as in Definition \ref{def:transfinite-composition}. By $(1)$, this
limit exists and belongs to $\mathcal{L}^{0}$.

We conclude the proof by noticing that, if $\varphi\in\mathcal{L}_{\mathfrak{D}}^{0}$,
that is, if $\mathcal{S}\left(\varphi\right)$ is of finite type,
so is the set $W$. 
\end{proof}
Propositions \ref{prop:changement-transfini} and \ref{prop:characterization-changes-variables}
will be used in the proof of Theorem A to derive the formal normal
forms of $f\in\mathcal{L}^{H}$ by \emph{transfinite induction}: eliminating
terms from $f$, step by step, by elementary changes of variables,
when possible.\\

By Definition~\ref{def:transfinite-composition}, the composition
of a transfinite sequence $(f_{\mu})_{\mu<\theta}\in\mathcal{L}^{0}$
\emph{exists}, if the transsequence of partial compositions $(\psi_{\mu})_{\mu<\nu}$
at any limit ordinal $\nu\leq\theta$ converges in the product topology
in $\mathcal{L}^{0}$. That is if, for every $(\beta,l)\in\mathbb{R}_{>0}\times\mathbb{Z}$,
there exists an index $\mu_{\beta,l}$ such that, for $\mu_{\beta,l}<\mu<\nu$,
the coefficient $[\psi_{\mu}]_{\beta,l}$ remains constant. In the
proof of Proposition~\ref{prop:characterization-changes-variables},
for transfinite compositions of \emph{elementary changes of variables}
we have proved (by Neumann's lemma) even more: for every $(\beta,l)\in\mathbb{R}_{>0}\times\mathbb{Z}$,
the coefficient $[\psi_{\mu}]_{\beta,l}$ changes in the sequence
of partial compositions $(\psi_{\mu})_{\mu<\nu}$ at most at \emph{finitely
many} indices.

\subsection{\label{sub:Precise-form-Theorem A}The precise form of Theorem A}

We now give the precise statement of Theorem A, which was given with
less details on page \pageref{eq:eksp}.

\begin{namedthm}
{Theorem A \rm{(Formal normal forms)}}Let $f\in\mathcal{L}^{H}$
(\emph{resp.} $f\in\mathcal{L}_{\mathfrak{D}}^{H}$). 
\begin{enumerate}[1., font=\textup, nolistsep, leftmargin=0.6cm]
\item $f$ is formally equivalent in $\mathcal{L}^{0}$ (resp. in $\mathcal{L}_{\mathfrak{D}}^{0}$)
to the finite normal form $f_{0}\in\mathcal{L}^{H}$ (actually in
$\mathcal{L}_{\mathfrak{D}}^{H}$):

\begin{enumerate}[(a), font=\textup, nolistsep,leftmargin=0.6cm]
\item \emph{(parabolic case)} 
\begin{align*}
f\left(x\right) & =x+ax^{\alpha}\boldsymbol{\ell}^{k}+\mathrm{h.o.t},\ \alpha\geq1,\ k\in\mathbb{Z},\ (\alpha,k)\succ(1,0);\ a\in\mathbb{R},\ a\neq0;\\
f_{0}\left(x\right) & =x+ax^{\alpha}\boldsymbol{\ell}^{k}+bx^{2\alpha-1}\boldsymbol{\ell}^{2k+1},\ b\in\mathbb{R}.
\end{align*}

\item \emph{(hyperbolic case)} 
\begin{align*}
f\left(x\right) & =\lambda x+ax\boldsymbol{\ell}+\mathrm{h.o.t},\ a\in\mathbb{R},\ \lambda>0,\ \lambda\ne1;\\
f_{0}\left(x\right) & =\lambda x+ax\boldsymbol{\ell}.
\end{align*}

\item \emph{(strongly hyperbolic case)} 
\begin{align*}
f(x)=\lambda x^{\alpha}+\mathrm{h.o.t},\ \lambda>0,\ \alpha\neq1;\qquad f_{0}(x)=x^{\alpha}.
\end{align*}

\end{enumerate}
\item Let $f$ be hyperbolic or parabolic. Then $f$ is formally equivalent
in $\mathcal{L}^{0}$ (\emph{resp}. $\in\mathcal{L}_{\mathfrak{D}}^{0}$)
to $\widehat{f}_{0}\in\mathcal{L}^{H}$ (\emph{resp}. in $\mathcal{L}_{\mathfrak{D}}^{H}$),
given as the formal time-one map of the following vector fields:

\begin{enumerate}[(a), font=\textup, nolistsep,leftmargin=0.6cm]
\item \emph{(parabolic case)} 
\begin{align*}
 & \qquad\quad f_{0}(x)=\exp(X_{\alpha,k,a,b}).\mathrm{id},\\
 & \qquad\quad X_{\alpha,k,a,b}=\frac{ax^{\alpha}\boldsymbol{\ell}^{k}}{1+\frac{a\alpha}{2}x^{\alpha-1}\boldsymbol{\ell}^{k}-(\frac{ak}{2}+\frac{b}{a})x^{\alpha-1}\boldsymbol{\ell}^{k+1}}\frac{\mathrm{d}}{\mathrm{d}x}.
\end{align*}

\item \emph{(hyperbolic case)} 
\begin{align*}
 & \ \widehat{f}_{0}(x)=\exp(X_{\lambda,a}).\mathrm{id},\quad X_{\lambda,a}=\frac{\log\lambda\cdot x}{1+\frac{a}{2(\lambda-1)}\boldsymbol{\ell}}\frac{\mathrm{d}}{\mathrm{d}x}.
\end{align*}

\end{enumerate}
\end{enumerate}
\noindent In the parabolic case, the formal normal forms are described
by the quadruples: 
\[
(\alpha,k,a,b);\ \alpha\geq1,\ k\in\mathbb{Z},\ b\in\mathbb{R},\ a\neq0.
\]
Additionally, if $\alpha>1$, $a$ can be replaced by $\mathrm{sgn}(a)$,
up to a linear change. 
\end{namedthm}
\noindent It is worth recalling that in the \emph{hyperbolic case},
the series $\exp\left(X_{\lambda,a}\right)\cdot\mathrm{id}$ does
not converge in $\mathcal{L}$ neither for the formal topology nor
for the product topology with respect to the discrete topology. Nevertheless,
it converges for the weak topology (the product topology with respect
to the Euclidean topology), which takes into account not only the
supports, but also the size of their coefficients. For details, see
the proof of Proposition~\ref{prop: existence-formal-flow-parabolic}.

Note that the formal normal form of a \emph{strongly hyperbolic} transseries
cannot be expressed as the formal time-one map of a vector field in
$\mathcal{L}$. The exponential of a parabolic vector field does not
converge in $\mathcal{L}$ in any of the three topologies that we
mentioned on page~\pageref{eq:elementary-change}. The formula \eqref{eq:eksp}
for the formal-time map of a parabolic field does not make sense in
$\mathcal{L}$. The detailed description of this phenomenon is given
in Proposition~\ref{prop:not-weakly-well-defined}.

Furthermore, notice that if $f\in\mathbb{R}[[x]]$ is a parabolic
formal power series, its formal normal form $f_{0}$ in $\mathcal{L}^{0}$
is different from the standard formal normal form (recalled in Section
\ref{sub:poisson-brackets-homologic}). Indeed, due to the fact that
we use a wider class of changes of variables, the residual term can
also be eliminated. See Example~\ref{fps} for details.

\subsection{\label{sub:proof-precise-form-A}Proof of the precise form of Theorem
A}

The proof is divided in three parts. Let $f\in\mathcal{L}^{H}$ (resp.
$f\in\mathcal{L}_{\mathfrak{D}}^{H}$). 
\begin{enumerate}[1., font=\textup, leftmargin=0.6cm]
\item \emph{Part }1 is \emph{the step} of the algorithm. We describe a
process which allows, by an appropriate elementary change of variables
as defined on page \pageref{eq:elementary-change}, to eliminate the
smallest possible monomial of $\mathcal{S}\left(f\right)$. \\

\item \emph{Part }2 is \emph{the convergence} of the algorithm. We prove
that the collection of consecutive changes of variables made in \emph{Part
1} is actually a \emph{transfinite sequence}, which can be indexed
by a well-ordered subset of $\mathbb{R}_{>0}\times\mathbb{Z}$. The
main difficulty here is the following one: each execution of a local
step of the algorithm, while eliminating a single monomial of the
support of the transseries to which it is applied, may at the same
time add infinitely many new monomials to the support. Hence, we have
to prove that, nevertheless, all the monomials which appear during
the process (except at most finitely many of them) will be ultimately
eliminated by a transfinite sequence of elementary changes of variables.
\\

\item Finally, in \emph{Part 3}, we show how to obtain another normal form,
which is the formal time-one map of a vector field in the sense of
Definition~\ref{def:time-one-map}. 
\end{enumerate}
\noindent \textbf{Part 1 (the step of the algorithm).}

\noindent Let $f\in\mathcal{L}^{H}$. We examine three possible situations:
$f$ parabolic, hyperbolic, or strongly hyperbolic. In each case,
we show how to construct an elementary change of variables which will
eliminate the smallest possible monomial of $\mathcal{S}\left(f\right)$.
We give a detailed description for $f$ parabolic. The other cases
follow the similar scheme. \\

(a) $f$ \emph{parabolic.} Let us write: 
\begin{align*}
f(x)=x+ax^{\alpha}\boldsymbol{\ell}^{k} & +a_{1}x^{\alpha}\boldsymbol{\ell}^{k+1}+\text{h.o.t},\\
 & \ a\in\mathbb{R},\ a\neq0,\ (\alpha,k)\succ(1,0).
\end{align*}
In \emph{Part 2} of the proof we want to eliminate the term of smallest
possible order in the expansion of $f$, and proceed by induction.
To see which terms can be eliminated, we examine the action of an
elementary change of variables $\varphi_{\beta,m}$: 
\begin{align*}
\varphi_{\beta,m}(x)=x+ & cx^{\beta}\boldsymbol{\ell}^{m},\quad c\in\mathbb{R},\ c\neq0,\ (\beta,m)\in\mathbb{R}_{>0}\times\mathbb{Z},\ (1,0)\prec(\beta,l).
\end{align*}
We apply the method described in Section \ref{sec:lie-brackets-normal-forms}.
Recall from \eqref{eq:formule-conjugaison} that if $\psi$ is the
leading term of the difference $f\circ\varphi_{\beta,m}-\varphi_{\beta,m}\circ f$,
then 
\begin{equation}
\varphi_{\beta,m}^{-1}\circ f\circ\varphi_{\beta,m}=f+\psi+\text{h.o.t}.\label{eq:lt}
\end{equation}
We prove that $\psi$ is exactly the leading term of $ad_{ax^{\alpha}\boldsymbol{\ell}^{k}}(cx^{\beta}\boldsymbol{\ell}^{m})$,
that is, 
\[
\psi=\text{Lt}\Big(ad_{ax^{\alpha}\boldsymbol{\ell}^{k}}(cx^{\beta}\boldsymbol{\ell}^{m})\Big),
\]
where $\text{Lt}()$ denotes the leading term of expression in brackets.
Indeed, write 
\[
f\left(x\right)=x+\varepsilon\left(x\right),\quad\varphi_{\beta,m}\left(x\right)=x+\eta\left(x\right),
\]
with $\left(1,0\right)\prec\text{ord}\left(\varepsilon\right)$ and
$\left(1,0\right)\prec\text{ord}\left(\eta\right)$. We obtain, using
Taylor formula: 
\begin{align*}
\big(\varphi_{\beta,m}\circ f-f\circ\varphi_{\beta,m}\big)\left(x\right) & =\varphi_{\beta,m}\left(x+\varepsilon\left(x\right)\right)-f\left(x+\eta\left(x\right)\right)\\
 & =\varphi_{\beta,m}\left(x\right)+\varphi_{\beta,m}'\left(x\right)\varepsilon\left(x\right)+\frac{1}{2}\varphi_{\beta,m}''\left(x\right)\varepsilon^{2}\left(x\right)\\
 & \quad-f\left(x\right)-f'\left(x\right)\eta\left(x\right)-\frac{1}{2}f''\left(x\right)\eta^{2}\left(x\right)+\text{h.o.t.}\\
 & =x+\eta\left(x\right)+\left(1+\eta'\left(x\right)\right)\varepsilon\left(x\right)+\frac{1}{2}\eta''\left(x\right)\varepsilon^{2}\left(x\right)\\
 & \quad-x-\varepsilon\left(x\right)-\left(1+\varepsilon'\left(x\right)\right)\eta\left(x\right)-\frac{1}{2}\varepsilon''\left(x\right)\eta^{2}\left(x\right)+\text{h.o.t.}\\
 & =\eta'\left(x\right)\varepsilon\left(x\right)-\eta\left(x\right)\varepsilon'\left(x\right)\\
 & \quad+\frac{1}{2}\left(\eta''\left(x\right)\varepsilon^{2}\left(x\right)-\varepsilon''\left(x\right)\eta^{2}\left(x\right)\right)+\text{h.o.t.}
\end{align*}
The expansion of this expression gives 
\begin{align*}
\big(\varphi_{\beta,m}\circ f-f\circ\varphi_{\beta,m}\big)\left(x\right)= & ca\left(\beta-\alpha\right)x^{\alpha+\beta-1}\boldsymbol{\ell}^{m+k}\\
 & +ca_{1}\left(\beta-\alpha\right)x^{\alpha+\beta-1}\boldsymbol{\ell}^{m+k+1}+ca\left(m-k\right)x^{\alpha+\beta-1}\boldsymbol{\ell}^{m+k+1}\\
 & +\frac{1}{2}\left(ca^{2}\beta\left(\beta-1\right)x^{\alpha+\beta-1+\left(\alpha-1\right)}\boldsymbol{\ell}^{m+2k}-c^{2}a\alpha\left(\alpha-1\right)x^{\alpha+\beta-1+\left(\beta-1\right)}\boldsymbol{\ell}^{2m+k}\right)\\
 & +\text{h.o.t.}
\end{align*}
Let us examine various possibilities for the leading term, depending
on $\alpha,\ \beta,\ k,\ m$. If $\beta\ne\alpha$, then the leading
term of this expression is $ca\left(\beta-\alpha\right)x^{\alpha+\beta-1}\boldsymbol{\ell}^{m+k}$.
If $\beta=\alpha$, notice that one of the terms of the second line
could contribute to the order $\left(\alpha+\beta-1,m+k+1\right)$,
when $\alpha=1$, or when $\beta=1$. But in this case, since $\alpha=\beta=1$,
the coefficients of these terms vanish. So, in any case, the leading
term $\psi$ of the former expression is exactly the leading term
of $\mathrm{ad}{}_{ax^{\alpha}\boldsymbol{\ell}^{k}}(cx^{\beta}\boldsymbol{\ell}^{m})=ca\left(\beta-\alpha\right)x^{\alpha+\beta-1}\boldsymbol{\ell}^{m+k}+ca\left(m-k\right)x^{\alpha+\beta-1}\boldsymbol{\ell}^{m+k+1}$.
By \eqref{eq:lt}, we now have: 
\[
\varphi_{\beta,m}^{-1}\circ f\circ\varphi_{\beta,m}=f+\text{Lt}\Big(\mathrm{ad}{}_{ax^{\alpha}\boldsymbol{\ell}^{k}}(cx^{\beta}\boldsymbol{\ell}^{m})\Big)+\mathrm{h.o.t.}
\]
In order to find the change of variables $\varphi_{\beta,m}$ whose
action would eliminate a given monomial $dx^{\gamma}\boldsymbol{\ell}^{l}$
in the expansion of $f$, we need to solve the \emph{homological equation}:
\begin{equation}
\text{Lt}\left(\mathrm{ad}{}_{ax^{\alpha}\boldsymbol{\ell}^{k}}(cx^{\beta}\boldsymbol{\ell}^{m})\right)=dx^{\gamma}\boldsymbol{\ell}^{r}.\label{solv}
\end{equation}
That is, as in the proof of Proposition~\ref{prop1}, we want to
see which monomials are not in the image of $\mathrm{ad}{}_{ax^{\alpha}\boldsymbol{\ell}^{k}}(J_{\beta,m})$
for any $(\beta,\, m)\succ(1,0)$, since these cannot be eliminated
by elementary changes of variables. The homological equation \eqref{solv}
leads to $\alpha+\beta-1=\gamma$. That is, to $\beta=\gamma-\alpha+1$.
We have three possibilities: 
\begin{enumerate}[(i), font=\textup, nolistsep,leftmargin=0.6cm]
\item If $\beta\ne\alpha$, i.e. if $\gamma\ne2\alpha-1$, then we put
$m=r-k$ and we can solve the homological equation. 
\item (If $\beta=\alpha$, i.e. if $\gamma=2\alpha-1$, then the homological
equation becomes $ca\left(m-k\right)x^{\gamma}\boldsymbol{\ell}^{m+k+1}=dx^{\gamma}\boldsymbol{\ell}^{r}$.
This equation leads to $m+k+1=r$, that is, $m=r-k-1$. If $m\ne k$,
i.e. if $r\ne2k+1$, the homological equation can be solved. 
\item If $r=2k+1$, then the homological equation cannot be solved, so the
term $dx^{\gamma}\boldsymbol{\ell}^{r}$ cannot be eliminated from
$f$. We have 
\[
dx^{\gamma}\boldsymbol{\ell}^{r}\notin\mathrm{ad}{}_{ax^{\alpha}\boldsymbol{\ell}^{k}}\Big(\bigcup_{(1,0)\prec(\beta,m)}J_{\beta,m}\Big).
\]

\end{enumerate}
Note that the assumption $(1,0)\prec(\beta,m)$ on order of elementary
changes of variables $\varphi_{\beta,m}-\text{id}$ is necessary so
that $\varphi_{\beta,m}\in\mathcal{L}^{0}$. \\

If $\alpha>1$, it follows from our computations that all the terms
in the expansion of $f$ can be eliminated except for the first term
$ax^{\alpha}\boldsymbol{\ell}^{k}$, and the \emph{residual term}
$dx^{2\alpha-1}\boldsymbol{\ell}^{2k+1}$.

If $\alpha=1$, along with these two terms, we observe that the term
$a_{1}x\boldsymbol{\ell}^{k+1}$ is not in the image of $\mathrm{ad}{}_{ax\boldsymbol{\ell}^{k}}$.
Indeed, to solve the homological equation, we need a change of variables
$\varphi_{1,0}$, which is impossible by the comment above. Nevertheless,
in that case, as initial step we apply the appropriate linear change
of variables, $\varphi_{1,0}\left(x\right)=cx$, $c\neq0$. The action
of the linear change of variables on the first terms of $f$ is explained
in the following computation: 
\begin{align}
\varphi_{1,0}^{-1}\circ f\circ & \varphi_{1,0}=x+ac^{\alpha-1}x^{\alpha}\boldsymbol{\ell}^{k}+\nonumber \\
 & +c^{\alpha-1}\big(a_{1}-ka\log c\big)x^{\alpha}\boldsymbol{\ell}^{k+1}+\text{h.o.t. }\label{pomm}
\end{align}
Hence, if $\left(1,k+1\right)\in\mathcal{S}\left(f\right)$, we eliminate
the term $a_{1}x\boldsymbol{\ell}^{k+1}$ with the linear change of
variables $\varphi_{1,0}(x)=e^{\frac{a_{1}}{k\cdot a}}x$.\\

Notice that, if $\alpha>1$, we can use the linear change of variables
$\varphi_{1,0}(x)=cx$, with $c=|a|^{-\frac{1}{\alpha-1}}$, to \emph{normalize}
the coefficient $a$ to $\text{sign}\left(a\right)$.\\

(b) \emph{$f$ hyperbolic.} Let 
\[
f(x)=\lambda x+ax\boldsymbol{\ell}+\text{h.o.t.}
\]
Applying the change of variables $\varphi_{\beta,m}(x)=x+cx^{\beta}\boldsymbol{\ell}^{m}$,
$c\in\mathbb{R}$, $(1,0)\prec(\beta,m)$, we obtain: 
\begin{align*}
\qquad\varphi_{\beta,m}\circ f & -f\circ\varphi_{\beta,m}=\\
 & =c(\lambda^{\beta}-\lambda)x^{\beta}\boldsymbol{\ell}^{m}+\\
 & \quad-\big(ac(1-\beta\lambda^{\beta-1})+cm\lambda^{\beta}\log\lambda\big)\, x^{\beta}\boldsymbol{\ell}^{m+1}+\text{h.o.t.}
\end{align*}
Then we proceed as in (a) above: every term can be eliminated, except
for the terms of order $(1,0)$ and $\left(1,1\right)$. \\

(c) \emph{$f$ strongly hyperbolic.} First, by the linear elementary
change of variables $\varphi_{1,0}(x)=\lambda^{-\frac{1}{\alpha-1}}x$,
we \emph{normalize} the first term. Then, as in the parabolic case,
we want to remove the other monomials by appropriate elementary changes
of variables $\varphi_{\beta,m}(x)=x+cx^{\beta}\boldsymbol{\ell}^{m}$,
$c\in\mathbb{R}$, $(1,0)\prec(\beta,m)$. Let: 
\begin{align*}
f(x)=x^{\alpha}+dx^{\gamma}\boldsymbol{\ell}^{r} & +\text{h.o.t},\quad(\alpha,0)\prec(\gamma,r),\ d\in\mathbb{R},\ d\neq0.
\end{align*}
As above, we consider the difference $\left(\varphi_{\beta,m}\circ f-f\circ\varphi_{\beta,m}\right)\left(x\right)=t_{\varphi_{\beta,m}}(x)+\text{h.o.t}$.
Here, the leading monomial $t_{\varphi_{\beta,m}}$ is given by: 
\begin{equation}
\qquad\ t_{\varphi_{\beta,m}}(x)=\begin{cases}
c(\alpha^{m}-\alpha)x^{\alpha}\boldsymbol{\ell}^{m}, & \beta=1,\\
-c\alpha x^{\alpha+\beta-1}\boldsymbol{\ell}^{m}, & \alpha>1,\ \beta\neq1,\\
c\alpha^{m}x^{\alpha\beta}\boldsymbol{\ell}^{m}, & \alpha<1,\ \beta\neq1.
\end{cases}\label{analy}
\end{equation}
By the change of variables $\varphi_{\beta,m}$ which solves the equation
$t_{\varphi_{\beta,m}}(x)=dx^{\gamma}\boldsymbol{\ell}^{r}$, we eliminate
the term $x^{\gamma}\boldsymbol{\ell}^{r}$ from $f$. Notice that,
unlike in the former cases, in the strongly hyperbolic case \emph{all
the monomials except for the first one can be eliminated}.\\

\noindent \textbf{Part 2 (the convergence of the algorithm).} Let
$f\in\mathcal{L}^{H}$. We repeatedly apply to $f$ the changes of
variables built in local \emph{Part 1} of the proof. This step by
step process leads to some \emph{collection} $(\varphi_{\mu})_{\mu\in I}$
of elementary changes of variables from $\mathcal{L}_{0}$, indexed
by some initial segment $I$ of the ordinals: 
\[
f\underset{\varphi_{0}}{\longrightarrow}f_{1}=\varphi_{0}^{-1}\circ f\circ\varphi_{0}\underset{\varphi_{1}}{\longrightarrow}f_{2}=\varphi_{1}^{-1}\circ f_{1}\circ\varphi_{1}\rightarrow\cdots
\]
For each step $\mu$, the change of variables $\varphi_{\mu}$ is
designed to eliminate the smallest possible monomial of the support
$\mathcal{S}\left(f_{\mu}\right)$. We have to prove that the collections
$\left(\varphi_{\mu}\right)$ and $\left(f_{\mu}\right)$ obtained
in the process are \emph{transfinite sequences}. That is, that there
exists a bounding ordinal $\theta$ such that $I$ is the set of ordinals
$\{\mu<\theta\}$. The idea is to analyze the orders of elementary
changes of variables used in step by step eliminations of all possible
monomials from $f$. The analysis for $f$ parabolic (other two cases
can be done similarly) is given in Section~\ref{sec:Appendix}, Subsections~\ref{sub:one}
and \ref{sub:two}. For $f(x)=x+ax^{\alpha}\boldsymbol{\ell}^{k}+h.o.t$,
$a\neq0$, we prove that the supports of all $(f_{\mu}-\textrm{id})$
belong to the set $\mathcal{R}\subset\mathbb{R}_{>0}\times\mathbb{Z}$:
\[
\mathcal{R}=\Big\langle\mathcal{S}(f-\mathrm{id})\setminus\{\alpha,k\}-\left(\alpha,k+1\right)\Big\rangle+\mathbb{N}_{*}\left(\alpha-1,k\right)+\left\{ 1\right\} \times\mathbb{N}_{*}.
\]
The orders of the elementary changes of variables used for normalization
thus belong to the set $\mathcal{R}_{1}\subset\mathbb{R}_{>0}\times\mathbb{Z}$
explicitely obtained from $\mathcal{R}$: 
\[
\mathcal{R}_{1}=\Big\langle\mathcal{S}(f-\mathrm{id})\setminus\{\alpha,k\}-\left(\alpha,k+1\right)\Big\rangle+\mathbb{N}\left(\alpha-1,k\right)+\left\{ 0\right\} \times\mathbb{N}.
\]
Both $\mathcal{R}$ and $\mathcal{R}_{1}$ are well-ordered by Neumann's
lemma~\ref{lem:neumann_lemma}, since $\mathcal{S}(f-\mathrm{id})$
is well-ordered. The computations in \emph{Part 1} of the proof show
that, in each step, not only the monomial of smallest order in $f_{\mu}$
is eliminated, but no other monomial of smaller order is added to
the support, so that the orders of $(f_{\mu}-\mathrm{id})$ and accordingly
of $(\varphi_{\mu}-\mathrm{id})$ strictly increase and at the same
time stay inside well-ordered sets $\mathcal{R}$ resp. $\mathcal{R}_{1}$.
The steps of eliminations can be carried through, since we know from
Propositions~\ref{prop:changement-transfini} and \ref{prop:characterization-changes-variables}
that the partial compositions $\psi_{\nu}$ and $f_{\nu}$ at any
limit ordinal $\nu$ do exist in $\mathcal{L}^{0}$.

Let us now index the collection $(\varphi_{\mu})$ of elementary changes
of variables by the orders $\text{ord}(\varphi_{\mu}-\mathrm{id})$.
The orders form a \emph{strictly increasing}, \emph{well-ordered}
subset $W$ of $\mathbb{R}_{\ge0}\times\mathbb{Z}$. Therefore, we
obtain a transfinite sequence of elementary changes of variables and
we use the notation $\left(\varphi_{\beta,m}\right)_{\left(\beta,m\right)\in W}$.
According to Proposition \ref{prop:characterization-changes-variables},
the transfinite composition $\varphi=\circ_{\left(\beta,m\right)\in W}\varphi_{\beta,m}$
is a well-defined element of $\mathcal{L}^{0}$. On the other hand,
Proposition~\ref{prop:changement-transfini} guarantees that our
transfinite process of eliminations \emph{ends}, that is, \emph{converges
to an element from $\mathcal{L}^{H}$}. By construction in the algorithm,
the limit is the normal form $f_{0}\in\mathcal{L}^{H}$. \\

If additionally $f\in\mathcal{L}_{\mathfrak{D}}^{H}$, that is if
$f$ is of finite type, we prove that the normalizing change of variables
$\varphi$ is also of finite type. The proof of this fact is quite
long and technical. It is described in detail for the parabolic case
in Section \ref{sec:Appendix}: Appendix. The proof of the hyperbolic
and the strongly hyperbolic case follow the same lines and are left
to the reader.\\

\noindent \textbf{Part 3 (the 2nd normal form).}

\emph{(i) }Take the vector field $X$ as in $(2.a)$ or $(2.b)$ of
the theorem. Expanding the \emph{coefficient} $\xi(x)$ of the vector
field in the geometric series, we see that $\xi\in\mathcal{L}$ and
that $\left(1,0\right)\preceq\mathrm{ord}\left(\xi\right)$. By Propositions
\ref{prop: existence-formal-flow-parabolic} and \ref{prop: existence-formal-flow-hyperbolic}
of Sections \ref{sub:fivetwo} and \ref{sub:fivefour}, the exponential
of $X$ (formula \eqref{eq:eksp}) converges in $\mathcal{L}$ and
gives a normal form as the formal time-one map $\widehat{f}_{0}\in\mathcal{L}^{H}$.
It should be mentioned that the appropriate topology for the convergence
of the series in \eqref{eq:eksp} depends on whether $\mathrm{ord}\left(\xi\right)$
is equal to or bigger than $\left(1,0\right)$. This is why the proof
of \ref{prop: existence-formal-flow-parabolic} is split between Subsection
\ref{sub:fivetwo} and \ref{sub:fivefour}. Finally, we simply observe
from this expansion that $\widehat{f}_{0}=f_{0}+\text{h.o.t.}$

\smallskip{}

\emph{(ii) }Let $f\in\mathcal{L}^{H}$, and let $f_{0}$ and $\widehat{f}_{0}$
be as in the statements $1.$ and $2.$ of the theorem. We show that
there exists a change of variables $\widehat{\varphi}\in\mathcal{L}^{0}$
that conjugates $f$ to $\widehat{f}_{0}$. Indeed, by Theorem~A$(1)$,
there exists $\varphi\in\mathcal{L}^{0}$ such that 
\begin{equation}
f=\varphi\circ f_{0}\circ\varphi^{-1}.\label{cvv}
\end{equation}
On the other hand, by $(i)$ above, $\widehat{f}_{0}\in\mathcal{L}^{H}$
and $\widehat{f}_{0}=f_{0}+\mathrm{h.o.t.}$ Applying the transfinite
algorithm described in \emph{Parts 1-2} of the proof to $\widehat{f}_{0}\in\mathcal{L}^{H}$,
we obtain an element $\psi\in\mathcal{L}^{0}$ such that: 
\begin{equation}
\widehat{f}_{0}=\psi\circ f_{0}\circ\psi^{-1}.\label{cvv1}
\end{equation}
By \eqref{cvv} and \eqref{cvv1}, it follows that $f$ can be transformed
into $\widehat{f}_{0}$ by composition $\widehat{\varphi}=\varphi\circ\psi^{-1},\ \widehat{\varphi}\in\mathcal{L}^{0}$.
That is, $f$ is conjugated to $\widehat{f}_{0}$ in $\mathcal{L}^{0}$.
Note that by Proposition~\ref{prop:characterization-changes-variables}(2)
$\widehat{\varphi}$ can be considered as a transfinite step by step
process of elementary changes of variables applied to $f$. \hfill{}$\Box$ 
\begin{rem}
Let $f_{0}\in\mathcal{L}^{H}$ be already in the form as in Theorem
A$(1)$, $(a),(b)\text{ or }(c)$. Let the beginning of $f\in\mathcal{L}^{H}$
coincide with $f_{0}$: 
\[
f=f_{0}+\text{h.o.t.}
\]
It is easy to see that the algorithm described in the proof of Theorem
A transforms $f$ to $f_{0}$ itself. In other words, every transseries
from $\mathcal{L}^{H}$ which begins by one of the normal forms $f_{0}$
from Theorem A has $f_{0}$ itself as its normal form in $\mathcal{L}$.
\end{rem}

\section{Proof of Theorem B}

\label{sec:proof-theorem-B}

In this section we state and prove the precise form of Theorem B.
It turns out that, unlike in the proof of Theorem A, the techniques
involved depend strongly on the nature (parabolic, hyperbolic, or
strongly hyperbolic) of the element $f\in\mathcal{L}^{H}$. Hence,
we divide the statement and the proof in different subsections. In
Subsection~\ref{sub:fiveone}, we recall and state some useful facts
about \emph{linear} operators on $\mathcal{L}$, more specifically
about isomorphisms and derivations on $\mathcal{L}$. An important
part is the relationship between elements of $\mathcal{L}^{H}$ and
linear operators acting on $\mathcal{L}$. Subsection~\ref{sub:fivetwo}
is dedicated to vector fields, that present a particular class of
derivations on $\mathcal{L}$. Finally, Subsections~\ref{sub:fivethree},\,\ref{sub:fivefour},\,\ref{sub:fivefive}
contain the statement and the proof of Theorem~B respectively in
parabolic, hyperbolic and strongly hyperbolic case.

\subsection{\label{sub:operators_on_LH}Operators acting on $\mathcal{L}$, isomorphisms
and derivations}

\label{sub:fiveone}\
 Some notions considered in this section are similar to \cite[Chapter I.3]{ilya}
for formal power series.

By an \emph{operator on $\mathcal{L}$ (respectively $\mathcal{L}_{\mathfrak{\mathfrak{\mathfrak{D}}}}$)},
we denote a strongly linear map $B\colon\mathcal{L}\ (\text{resp. }\mathcal{L}_{\mathfrak{D}})\rightarrow\mathcal{L}\ (\text{resp. }\mathcal{L}_{\mathfrak{D}})$.
By \emph{strongly linear}, we mean that 
\[
B\left(\sum_{\alpha,k}c_{\alpha,k}x^{\alpha}\boldsymbol{\ell}^{k}\right)=\sum_{\alpha,k}c_{\alpha,k}B\left(x^{\alpha}\boldsymbol{\ell}^{k}\right),\quad c_{\alpha,k}\in\mathbb{R},
\]
for every transseries $\sum_{\alpha,k}c_{\alpha,k}x^{\alpha}\boldsymbol{\ell}^{k}\in\mathcal{L}$\ (resp.
$\mathcal{L}_{\mathfrak{D}}$).

We denote by $L(\mathcal{L})$, respectively $L(\mathcal{L}_{\mathcal{\mathfrak{D}}})$,
the set of all operators $B:\mathcal{L}\ (\text{resp. }\mathcal{L}_{\mathfrak{D}})\to\mathcal{L}\ (\text{resp. }\mathcal{L}_{\mathfrak{D}})$.

For an operator $B\in L(\mathcal{L})$ and an element $f\in\mathcal{L}$,
we denote indifferently by $B\cdot f$ or $B\left(f\right)$ the image
of $f$ under $B$. The identity operator will be denoted by $\mathrm{Id}$.
\begin{defn}[Operators defined as a series of operators]
 Let $(B_{j})_{j\in\mathbb{N}}$ be a sequence of operators in $L(\mathcal{L})$
(resp. in $L(\mathcal{L}_{\mathfrak{D}})$).
\begin{enumerate}[1., font=\textup, nolistsep, leftmargin=0.6cm]
\item We say that the operator $B\in L(\mathcal{L})$ (resp. $B\in L(\mathcal{L}_{\mathfrak{D}})$)
is \emph{well-defined by the series $\sum_{j=0}^{\infty}B_{j}$} if,
for every $f\in\mathcal{L}$ (resp. $f\in\mathcal{L}_{\mathfrak{D}}$),
the sequence $\sum_{j=0}^{N}B_{j}\cdot f$ converges towards $B\cdot f$
in the \emph{formal topology}, as $N\to\infty$. 
\item If for every $f\in\mathcal{L}$ (resp. $\mathcal{L}_{\mathfrak{D}}$)
the sequence $\sum_{j=0}^{N}B_{j}\cdot f$ converges towards $B\cdot f$
in the \emph{weak topology}, we say that $B$ is \emph{weakly well-defined
by $\sum_{j=0}^{\infty}B_{j}$}. 
\end{enumerate}
In both cases, we write $B:=\sum_{j=0}^{\infty}B_{j}$.
\end{defn}
The notion \emph{weakly} used throughout the article indicates relation
to the weak topology on $\mathcal{L}$, see also the Definition~\ref{def:sows}
of the \emph{small operator in the weak sense} in Section~\ref{sub:fivefour}.

Note that \emph{well-defined} is a stronger notion than \emph{weakly
well-defined}, since it relates to the stronger formal topology. That
is, if an operator $B$ is well-defined by a series of operators,
then it is also weakly well-defined by the same series. Note also
that the operator series defines an operator in $L(\mathcal{L})$
\emph{as soon as the convergence of the series is weak}.
\begin{defn}[Formal differential operator in $L(\mathcal{L})$]
\label{def:form_diff} We say that an operator $B\in L(\mathcal{L})$
(resp. $B\in L(\mathcal{L}_{\mathfrak{D}})$) is a \emph{formal differential
operator} if there exists a sequence $(h_{j})_{j\in\mathbb{N}}$ of
elements of $\mathcal{L}$ (resp. $\mathcal{L}_{\mathfrak{D}}$) such
that $B$ is (weakly) well-defined by the series 
\[
B=\sum_{j=0}^{\infty}h_{j}\frac{\mathrm{d}^{j}}{{\mathrm{d}x}^{j}}.
\]
 
\end{defn}
The following definition of a \emph{small operator} is inspired by
\cite[Section 1.3]{dries}. 
\begin{defn}
\label{def:small_operator}An operator $B:\mathcal{L}\to\mathcal{L}$
is \emph{small} if there exists a well-ordered set $R\subseteq\mathbb{R}_{\ge0}\times\mathbb{Z}$
of exponents strictly bigger than $\left(0,0\right)$ such that $\mathcal{S}\left(B.f\right)\subseteq\mathcal{S}\left(f\right)+R$,
for every $f\in\mathcal{L}$. An operator $B:\mathcal{L}_{\mathfrak{D}}\to\mathcal{L}_{\mathfrak{D}}$
is small if the set $R$ is in addition of finite type. \end{defn}
\begin{prop}
\label{lem:series_small_operator} Let $B$ be a small operator on
$\mathcal{L}$ (resp. $\mathcal{L}_{\mathfrak{D}}$) and let $\left(c_{k}\right)_{k\in\mathbb{N}\cup\{0\}}$
be a sequence of real numbers. The sum 
\begin{equation}
S:=\sum_{k=0}^{\infty}c_{k}B^{k}\label{eq:seri}
\end{equation}
is a well-defined operator on $\mathcal{L}$ (resp. $\mathcal{L}_{\mathfrak{D}}$).
Here, $B^{k}$ denotes the $k^{\mathrm{th}}$ iterate of $B$.
\end{prop}
Proposition~\ref{lem:series_small_operator} is a special case of
a more general fact used repeatedly in \cite{dries}. Note that the
proof is based on the \emph{smallness property} of operator $B$.
It implies indeed that $\mathcal{S}(S.f)\subseteq\mathcal{S}(f)+\langle R\rangle$,
where $\langle R\rangle$ denotes the (additive) sub-semigroup of
$\mathbb{R}_{\geq0}\times\mathbb{Z}$ generated by $R$. Furthermore,
for any $f\in\mathcal{L}$, the order $\text{ord}(B^{k}.f)$ strictly
increases as $k$ increases, by at least $\min\{R\}$ in every step.
Consequently, the series $S_{n}:=\sum_{k=0}^{n}c_{k}B^{k}.f\in\mathcal{L}$
converges to $S\in\mathcal{L}$ in the \emph{formal topology}. We
omit the details of the proof. 
\begin{prop}
\label{def:exponential_logarithm_operator}Let $B$ be a small operator
on $\mathcal{L}$ (resp. $\mathcal{L}_{\mathfrak{D}}$). Then 
\[
\exp\left(B\right):=\sum_{k=0}^{\infty}\frac{B^{k}}{k!},\quad\log\left(\mathrm{Id}+B\right):=\sum_{k=1}^{\infty}\frac{\left(-1\right)^{k+1}B^{k}}{k},
\]
$\exp\big(\log(\mathrm{Id}+B)\big)$ and $\log\exp(B)$ are well-defined
operators on $\mathcal{L}$ (resp. $\mathcal{L}_{\mathfrak{D}}$).
Moreover, 
\begin{equation}
\exp\big(\log(\mathrm{Id}+B)\big)=\mathrm{Id}+B\text{ and }\log\exp(B)=B.\label{eq:equal}
\end{equation}
\end{prop}
\begin{proof}
Since $B$ is a small operator, by Proposition~\ref{lem:series_small_operator},
$\log(\mathrm{Id}+B)$ and $\exp{B}$ are well-defined operators on
$\mathcal{L}$ ($\mathcal{L}_{\mathfrak{D}}$). Moreover, by Definition~\ref{def:small_operator}
of small operators on $\mathcal{L}$ (resp. $\mathcal{L}_{\mathfrak{D}}$),
we obtain inductively: 
\[
\mathcal{S}(B.f)\subseteq\mathcal{S}(f)+R,\ \mathcal{S}(B^{k}\cdot f)\subseteq\mathcal{S}(f)+\langle R\rangle,\ k\in\mathbb{N}_{0},
\]
where $R$ is as in Definition~\ref{def:small_operator}. Therefore,
\[
\mathcal{S}(\log\left(\mathrm{Id}+B\right)),\ \mathcal{S}(\exp\left(B\right))\subseteq\mathcal{S}(f)+\langle R\rangle.
\]
The operators $\exp{B}$ and $\log\left(\mathrm{Id}+B\right)$ are
small in $L(\mathcal{L})$ (resp. $L(\mathcal{L}_{\mathfrak{D}})$).
It follows from Proposition~\ref{lem:series_small_operator} that
$\exp\big(\log(\mathrm{Id}+B)\big)$ and $\log\exp(B)$ are also well-defined
operators in $L(\mathcal{L})$ (resp. $L(\mathcal{L}_{\mathfrak{D}})$).
The equality \eqref{eq:equal} now follows by symbolic computation
from the standard properties of formal exp-log series, similarly as
in the proof of Proposition~\ref{prop:derivative}.\end{proof}
\begin{defn}
Let $B\colon\mathcal{L}\rightarrow\mathcal{L}$ be an operator on
$\mathcal{L}$. 
\begin{enumerate}[1., font=\textup, nolistsep, leftmargin=0.6cm]
\item We say that $B$ is a \emph{derivation} if it satisfies the usual
Leibniz's rule. 
\item We say that $B:\mathcal{L}\to\mathcal{L}$ is a \emph{morphism} if
it satisfies the \emph{morphism property}: $B\left(f\cdot g\right)=B\left(f\right)\cdot B\left(g\right)$,\ $f,\, g\in\mathcal{L}$.
\item We say that $B:\mathcal{L}\to\mathcal{L}$ is an \emph{isomorphism}
of $\mathcal{L}$ if $B$ is a bijective morphism. 
\end{enumerate}
\end{defn}
\begin{rem}[Isomorphisms associated with $f\in\mathcal{L}^{H}$ parabolic or hyperbolic]
 \label{ex:automorphism_associated} Let $f\in\mathcal{L}^{H}$ be
parabolic or hyperbolic. The map $F\colon\mathcal{L}\rightarrow\mathcal{L}$
defined by 
\begin{equation}
F(g)=g\circ f,\ g\in\mathcal{L},\label{eq:aut}
\end{equation}
is an isomorphism of $\mathcal{L}$. Moreover, the same conclusion
holds in finitely generated case. If $f\in\mathcal{L}_{\mathfrak{D}}^{H}$
is parabolic or hyperbolic, then $F$ defined by \eqref{eq:aut} is
an isomorphism of $\mathcal{L}_{\mathfrak{D}}^{H}$.

We call such $F$ the \emph{isomorphism associated with $f$} and
denote it by 
\[
F=\text{iso}(f).
\]
The morphism property is easily checked. Moreover, since a parabolic
(resp. hyperbolic) element $f\in\mathcal{L}^{H}$ admits a parabolic
(resp. hyperbolic) compositional inverse $f^{-1}\in\mathcal{L}^{H}$,
then $F$ is bijective, with the inverse $F^{-1}:\mathcal{L}\to\mathcal{L}$,
$F^{-1}(g)=g\circ f^{-1}$, $g\in\mathcal{L}$. \end{rem}
\begin{lem}
\label{lem:expo-log_and_log-exp} Let $f\in\mathcal{L}^{H}$ $($resp.
$f\in\mathcal{L}_{\mathfrak{D}}^{H})$ be parabolic or hyperbolic
contraction. Let the operator $F$ be defined as in \eqref{eq:aut}.
Then the formal operators $\log F\in L(\mathcal{L}),\ \exp\log F\in L(\mathcal{L})$
$\big($resp. $L(\mathcal{L}_{\mathfrak{D}})\big)$ are weakly well-defined.
Moreover, 
\[
\exp\log F=F.
\]
Finally, if $f$ is parabolic, these operators are well-defined.\end{lem}
\begin{proof}[Proof of Lemma \ref{lem:expo-log_and_log-exp} ]
We prove here the lemma for $f$ parabolic. For $f$ hyperbolic,
the proof is postponed to Section~\ref{sub:fivefour}. We write $f\left(x\right)=x+\varepsilon\left(x\right)$,
with $\text{ord}\left(\varepsilon\right)\succ\left(1,0\right)$. By
Taylor expansion, for every $g\in\mathcal{L}$, we have: 
\begin{align}
g\left(f\left(x\right)\right) & =g\left(x+\varepsilon\left(x\right)\right)\nonumber \\
 & =g\left(x\right)+\sum_{k=1}^{\infty}\dfrac{g^{\left(k\right)}\left(x\right)}{k!}\varepsilon\left(x\right)^{k}.\label{eq:act}
\end{align}
Hence, we can write $F=\text{Id}+P=\mathrm{Id}+\sum_{k=1}^{\infty}\frac{\varepsilon\left(x\right)^{k}}{k!}\frac{\mathrm{d}^{k}}{\mathrm{d}x^{k}}$.
Obviously, $P\cdot g=g\circ f-g\in\mathcal{L}$, $g\in\mathcal{L}$.
We show that $P$ is a small operator. By \eqref{eq:act}, we have:
\[
\mathcal{S}(P\cdot g)=\bigcup_{k\in\mathbb{N}}\mathcal{S}\big(g^{(k)}\varepsilon^{k}\big).
\]
The support $\mathcal{S}\big(g^{(k)}\varepsilon^{k}\big)$ contains
pairs of the form: 
\[
\Big((\beta_{1}-1)+\cdots+(\beta_{k}-1)+\alpha,\ l_{1}+\cdots+l_{k}+m+j\Big),
\]
where $\left(\beta_{i},l_{i}\right)\in\mathcal{S}\left(\varepsilon\right)$,
$i=1,\ldots,k$, $\left(\alpha,m\right)\in\mathcal{S}\left(g\right)$
and $j\in\{0,\ldots,k\}$. Therefore, 
\begin{equation}
\mathcal{S}(P\cdot g)\subseteq\mathcal{S}(g)+R,\label{eq:prv}
\end{equation}
where $R$ is a sub-semigroup $R\subseteq\mathbb{R}_{\geq0}\times\mathbb{Z}$
generated by elements $\left(\beta-1,l\right)$ for $\left(\beta,l\right)\in\mathcal{S}\left(\varepsilon\right)$,
and $(0,1)$. By Neumann's lemma and since $(1,0)\prec\text{ord}(\varepsilon)$,
$R$ is well ordered and its elements are of order strictly greater
than $\left(0,0\right)$. Therefore, the operator $P$ is small. By
Proposition~\ref{def:exponential_logarithm_operator}, the operators
$\log F$ and $\exp(\log F):\mathcal{L}\to\mathcal{L}$ are well-defined.\\

It remains to be proven that 
\begin{equation}
\exp(\log F).f=F.f,\ f\in\mathcal{L}.\label{symb}
\end{equation}
But once formal convergence is proven, this property follows from
the well-known formal identities concerning $\exp-\log$ series.

In the finitely generated case ($f\in\mathcal{L}_{\mathfrak{D}}$),
the semigroup $R$ above is in addition of finite type (a subset of
a finitely generated sub-semigroup of $\mathbb{R}_{\geq0}\times\mathbb{Z}$),
for details see the ``finite part'' of the proof of Proposition~\ref{prop:characterization-changes-variables}.
The operator $P$ is small in $L(\mathcal{L}_{\mathfrak{D}})$, and
the result follows by Proposition~\ref{def:exponential_logarithm_operator}.
\end{proof}
\noindent We suspect that the next statement is already known, but
we could not find it in the literature. Therefore, we give a short
proof. 
\begin{prop}
\noindent \label{prop:derivative} Let $A:\mathcal{L}\to\mathcal{L}$
be a linear morphism. Assume that the operator $\log A:\mathcal{L}\to\mathcal{L}$
is (weakly) well-defined. Then $\log A$ is a derivation.\end{prop}
\begin{proof}[Proof of Proposition \ref{prop:derivative}]
\noindent   Take any $f,\ g\in\mathcal{L}$. We prove the Newton-Leibniz
rule, that is, 
\[
\log A(fg)=\log A\left(f\right)g+f\log A\left(g\right).
\]
Put $H=A-Id$. Using the fact that $A$ is a morphism acting on $\mathcal{L}$,
we compute: 
\begin{align}
H(fg)=A(fg)-fg=A\left(f\right)A\left(g\right)-fg= & (f+H\left(f\right))\cdot(g+H\left(g\right))-fg=\nonumber \\
= & H\left(f\right)g+fH\left(g\right)+H\left(f\right)H\left(g\right).\label{upp}
\end{align}
Using the linearity of $H$ and \eqref{upp}, we compute $H^{2}(fg)$:
\[
H^{2}(fg)=H^{2}\left(f\right)g+2H^{2}\left(f\right)H\left(g\right)+2H\left(f\right)H\left(g\right)+2H\left(f\right)H^{2}\left(g\right)+fH^{2}\left(g\right).
\]
We proceed by symbolic computation. We substitute 
\[
x^{i}\text{ for }H^{i}\left(f\right),\ y^{i}\text{ for }H^{i}\left(g\right),\ i\in\mathbb{N}_{0}.
\]
By induction, the symbolic computation allows to substitute 
\[
(x+xy+y)^{k}\text{ for }H^{k}(fg),\ k\in\mathbb{N}_{0}.
\]
Hence, we have: 
\begin{align*}
\log A(fg)= & \text{(substitution})=\sum_{i=1}^{\infty}\frac{(-1)^{i+1}}{i}(x+xy+y)^{i}=\\
= & \log(1+x+xy+y)=\log\big((1+x)(1+y)\big)=\\
= & \log(1+x)+\log(1+y)=\\
= & (\text{substitution})=\log A\left(f\right)g+f\log A\left(g\right),
\end{align*}
which proves that $\log A$ is a derivation.
\end{proof}

\subsection{Vector fields and differential operators}

\label{sub:fivetwo}

We focus in this subsection on a special type of derivations. We denote
by $\frac{\mathrm{d}}{\mathrm{d}x}$ the usual derivation on germs
of functions. Note that, by strong linearity, the derivation $\frac{\mathrm{d}}{\mathrm{d}x}$
can be extended as an \emph{operator on $\mathcal{L}$.} 
\begin{defn}
An operator $B$ on $\mathcal{L}$ is a \emph{vector field} if there
exists $\xi\in\mathcal{L}$ such that $B=\xi\frac{\mathrm{d}}{\mathrm{d}x}$.
\end{defn}
Notice that there is an important difference here between $\mathcal{L}$
and $\mathbb{R}\left[\left[x\right]\right]$. A vector field is determined
by its value on the element $x\in\mathcal{L}$. But since $\mathcal{L}$
contains infinitely many elements, which are, on $\mathbb{R}$, algebraically
independent of $x$ (such as, for example, the powers $x^{\alpha},\alpha\in\mathbb{R}_{>0}\setminus\mathbb{Q}$),
then all the derivations on $\mathcal{L}$ cannot be vector fields.\\

Theorem B discusses the possible embedding of an element $f\in\mathcal{L}^{H}$
(in the three cases) in a formal flow of a vector field from $\mathcal{L}$.
Let us recall the definition of a formal flow, adapted from the standard
definition in the usual setting of formal power series to our class
$\mathcal{L}$. The next discussions follow the lines of similar results
for usual power series (see \cite[Chapter I.3]{ilya}, for example).
\begin{prop}[The existence of a formal flow of a formal vector field in $\mathcal{L}$,
the parabolic case]
\label{prop: existence-formal-flow-parabolic} Let $X=\xi\frac{\mathrm{d}}{{\mathrm{d}x}}$,
$\xi\in\mathcal{L}$, be a vector field in $\mathcal{L}$ such that
$(1,0)\prec\mathrm{ord}\left(\xi\right)$. Then the vector field $X$
admits the $\mathcal{C}^{1}$-formal flow $\{h^{t}\in\mathcal{L}^{0}:\ t\in\mathbb{R}\}$
defined by $h^{t}:=H^{t}\cdot\text{id}$, where $\{H^{t}\in L(\mathcal{L}):\ t\in\mathbb{R}\}$
is the one-parameter group of isomorphisms of $\mathcal{L}$ well-defined
by: 
\begin{equation}
H^{t}:=\exp(tX)=\sum_{k=0}^{\infty}\frac{t^{k}}{k!}X^{k}.\label{eq:exponential-operator}
\end{equation}
Moreover, $H^{t}$ are the isomorphisms associated to $h_{t},\ t\in\mathbb{R}$,
in the sense of Remark~\ref{ex:automorphism_associated}.

If, in addition, $\xi\in\mathcal{L}_{\mathfrak{D}}$, then $H^{t}\in L(\mathcal{L}_{\mathfrak{D}})$,
$h^{t}\in\mathcal{L}_{\mathfrak{D}}^{0}$, \textup{$t\in\mathbb{R}$.}
\end{prop}

\begin{prop}[The existence of a formal flow of a formal vector field, the hyperbolic
case]
\label{prop: existence-formal-flow-hyperbolic}Let $X=\xi\frac{\mathrm{d}}{{\mathrm{d}x}}$,
$\xi\in\mathcal{L}$, be a vector field in $\mathcal{L}$ such that
$\mathrm{ord}(\xi)=(1,0)$. Then the statements of Proposition~\ref{prop: existence-formal-flow-parabolic}
hold in this case as well, with the difference that $H^{t}$ is just
\emph{weakly} well-defined by \eqref{eq:exponential-operator}.
\end{prop}
\noindent Note that the time-$t$ map $h^{t}$ of $X$ in Propositions~\ref{prop: existence-formal-flow-parabolic}
and \ref{prop: existence-formal-flow-hyperbolic} is given by the
following formula: 
\begin{equation}
h^{t}=H^{t}\cdot\text{id}=\text{id}+t\xi+\frac{t^{2}}{2!}\xi'\xi+\frac{t^{3}}{3!}\big(\xi'\xi\big)'\xi+\cdots\label{con}
\end{equation}
Note also that in the case $(1,0)\prec\text{ord}(\xi)$, $h^{t}\in\mathcal{L}$
are \emph{parabolic}, while in the case $\text{ord}(\xi)=(1,0)$ they
are \emph{hyperbolic}. Moreover, the formula \eqref{con} converges
in the formal topology if $(1,0)\prec\text{ord}(\xi)$, and in the
weak topology if $\text{ord}(\xi)=(1,0)$. 
\begin{proof}[Proof of Proposition~\ref{prop: existence-formal-flow-parabolic}\emph{.}]
The assumption $\mathrm{ord}(\xi)$ guarantees that $X=\xi\frac{\mathrm{d}}{{\mathrm{d}x}}$
is a \emph{small operator} in the sense of Definition~\ref{def:small_operator}.
It is easy to check that 
\begin{equation}
\mathcal{S}(X.g)=\mathcal{S}(\xi\cdot g')\subseteq\mathcal{S}(g)+R,\ g\in\mathcal{L},\label{eq:s}
\end{equation}
where $R$ is a sub-semigroup of $\mathbb{R}_{\geq0}\times\mathbb{Z}$
generated by elements $(\beta-1,l)$, $(\beta,l)\in\mathcal{S}(\xi)$,
and $(0,1)$. All elements of $R$ are of order strictly bigger than
$(0,0)$. Hence, the sum \eqref{eq:exponential-operator} gives by
Proposition~\ref{lem:series_small_operator} a \emph{well-defined}
operator $H^{t}$ for all $t\in\mathbb{R}$.

The statement in $\mathcal{L}_{\mathfrak{D}}$ follows as in the proof
of finite part of Lemma~\ref{lem:expo-log_and_log-exp}. By \eqref{eq:s},
we have that $\mathcal{S}(H^{t}.g)\subseteq\mathcal{S}(g)+R,\ g\in\mathcal{L},\ t\in\mathbb{R}$.

Finally, the proof of the morphism property of operators $H^{t}$,
$t\in\mathbb{R}$, and the proof that the family $\left(h^{t}\right)_{t}$
is a flow of $X$ (see Definition~\ref{def:time-one-map}) are routine,
following the lines of similar results for formal power series, see
for example \cite[Chapter I.3]{ilya}.

To prove that $H^{t}.f=f\circ h_{t}$, $f\in\mathcal{L},$ we combine
Lemma~\ref{lem:correspondance-diffeo-automorphism} below in this
section and Proposition~\ref{rem:dif}.
\end{proof}

The proof of Proposition~\ref{prop: existence-formal-flow-hyperbolic}
(the case $\mathrm{ord}\left(\xi\right)=\left(1,0\right)$) is more
involved and is postponed to Subsection \ref{sub:fivefour}. In fact,
in this case $X$ is not a small operator in the sense of Definition~\ref{def:small_operator}.
Therefore, $H^{t}$ is not well-defined by \eqref{eq:exponential-operator}.
Nevertheless, we prove in Subsection \ref{sub:fivefour} that it is
\emph{weakly well-defined} by \eqref{eq:exponential-operator}.
\begin{prop}[Uniqueness of the $\mathcal{C}^{1}$-formal flow of a vector field]
\label{prop:unic} Consider a vector field $X=\xi\frac{d}{dx}$,
$\xi\in\mathcal{L}$.
\begin{enumerate}[1., font=\textup, nolistsep,leftmargin=0.6cm]
\item If $\left(1,0\right)\preceq\mathrm{ord}\left(\xi\right)$, then there
exists a \emph{unique} $\mathcal{C}^{1}$-formal flow $(f^{t})_{t\in\mathbb{R}}$
of $X$, $f^{t}\in\mathcal{L}^{H}$, in the sense of Definition~\ref{def:time-one-map}.
Moreover:

\begin{enumerate}[(i), font=\textup, nolistsep,leftmargin=0.6cm]
\item if $\left(1,0\right)\prec\mathrm{ord}(\xi)$, then the $f^{t}\in\mathcal{L}^{H}$
are parabolic; 
\item if $\mathrm{ord}(\xi)=(1,0)$, then the $f^{t}\in\mathcal{L}^{H}$
are hyperbolic.
\end{enumerate}
\item If $\mathrm{ord}\left(\xi\right)\prec\left(1,0\right)$, then $X$
does not admit any $\mathcal{C}^{1}$-flow.
\end{enumerate}
\end{prop}
\begin{proof}
1. The existence of a $\mathcal{C}^{1}$-flow for the vector field
$X$ is shown by an explicit construction in Propositions \ref{prop: existence-formal-flow-parabolic}
and \ref{prop: existence-formal-flow-hyperbolic}. Suppose now that
$X$ admits two $\mathcal{C}^{1}$-flows $\left(f^{t}\right)_{t\in\mathbb{R}}$
and $\left(g^{t}\right)_{t\in\mathbb{R}}$ in $\mathcal{L}^{H}$.
Let $S$ be a well-ordered subset of $\mathbb{R}_{>0}\times\mathbb{Z}$
such that $\mathcal{S}\left(f^{t}\right)$ and $\mathcal{S}\left(g^{t}\right)$
are contained in $S$ for all $t\in\mathbb{R}$. Let $\left(\alpha,m\right)$
be the smallest element of $S$ (which exists since $S$ is well-ordered)
such that the coefficient $h\left(t\right)$ of $x^{\alpha}\boldsymbol{\ell}^{m}$
in $f^{t}\left(x\right)-g^{t}\left(x\right)$ does not vanish identically.
Since $\left(f^{t}\right)$ and $\left(g^{t}\right)$ are both $\mathcal{C}^{1}$-formal
flows of $X=\xi\frac{\mathrm{d}}{\mathrm{d}x}$, we have the integral
equation in $\mathcal{L}$:
\begin{equation}
f^{t}\left(x\right)-g^{t}\left(x\right)=\int_{0}^{t}\big(\xi\left(f^{s}\right)\left(x\right)-\xi\left(g^{s}\right)\left(x\right)\big)\mathrm{d}s,\quad\forall t\in\mathbb{R},\label{eq:flow-integral}
\end{equation}
where the integral on \eqref{eq:flow-integral} is applied on each
coefficient of the integrand. The coefficient of the monomial $x^{\alpha}\boldsymbol{\ell}^{m}$
on the left-hand side of \eqref{eq:flow-integral} is $h\left(t\right)$.
Hence, in order to estimate the coefficient of the same monomial on
the right-hand side of this equation, we write, based on the definition
of $\left(\alpha,m\right)$:
\begin{align*}
f^{s}\left(x\right) & =M\left(s;x\right)+h_{1}\left(s\right)x^{\alpha}\boldsymbol{\ell}^{m}+\mathrm{h.o.t.}\\
g^{s}\left(x\right) & =M\left(s;x\right)+h_{2}\left(s\right)x^{\alpha}\boldsymbol{\ell}^{m}+\mathrm{h.o.t.}
\end{align*}
Here, $M\left(s;x\right)=b\left(s\right)x^{\beta}+\mathrm{h.o.t.}$,
$\left(\beta,0\right)\prec\left(\alpha,m\right)$, is a transseries
with monomials in $S$ and coefficients in $\mathcal{C}^{1}\left(\mathbb{R}\right)$
such that $b$ is not identically zero, and $h_{1}$, $h_{2}$ are
$\mathcal{C}^{1}$-functions. Obviously, $h=h_{1}-h_{2}$.

Let $ax^{\gamma}\boldsymbol{\ell}^{n}$ be the leading term of $\xi$.
We see that the leading term of the difference $\xi\left(f^{s}\right)-\xi\left(g^{s}\right)$
is
\[
\begin{aligned}a\left(\frac{1}{\beta}\right)^{n}\gamma b\left(s\right)^{\gamma-1} & \left(h_{1}\left(s\right)-h_{2}\left(s\right)\right)x^{\alpha+\beta\left(\gamma-1\right)}\boldsymbol{\ell}^{m+n}\\
 & =a\left(\frac{1}{\beta}\right)^{n}\gamma b\left(s\right)^{\gamma-1}h\left(s\right)x^{\alpha+\beta\left(\gamma-1\right)}\boldsymbol{\ell}^{m+n}.
\end{aligned}
\]
If $\left(1,0\right)\prec\left(\gamma,n\right)$, the order of the
right-hand side is bigger than $\left(\alpha,m\right)$. It would
imply $h\equiv0$, which is a contradiction. On the other hand, if
$\left(\gamma,n\right)=\left(1,0\right)$, by comparing the coefficients
of $x^{\alpha}\boldsymbol{\ell}^{m}$ on both sides of \eqref{eq:flow-integral},
we see that:
\[
h\left(t\right)=a\int_{0}^{t}h\left(s\right)\mathrm{d}s\text{, so }\left|h\left(t\right)\right|\leq\left|a\right|\int_{0}^{t}\left|h\left(s\right)\right|\mathrm{d}s,\quad a\in\mathbb{R}.
\]
It follows from \emph{Gronwall's lemma} applied to $\left|h\right|$
that $h\equiv0$, which is again a contradiction.

The points (i) and (ii) follow by uniqueness on one hand, and by the
explicit construction of the flow done in the proof of Propositions
\ref{prop: existence-formal-flow-parabolic} and \ref{prop: existence-formal-flow-hyperbolic}
on the other hand.

2. Assume now that $\left(\beta,m\right):=\mathrm{ord}\left(\xi\right)\prec\left(1,0\right)$
and that $X$ admits a $\mathcal{C}^{1}$-flow $\left(f^{t}\right)_{t\in\mathbb{R}}$.
We show that this assumption leads to a contradiction. As above, let
$S$ be a well-ordered subset of $\mathbb{R}_{>0}\times\mathbb{Z}$
such that $\mathcal{S}\left(f^{t}\right)\subseteq S$ for all $t\in\mathbb{R}$.
Let $\left(\alpha,m\right)$ be the smallest element of $S$ such
that the coefficient $h\left(t\right)$ of $x^{\alpha}\boldsymbol{\ell}^{m}$
in $f^{t}$ does not vanish identically. Let $t_{0}\in\mathbb{R}$
with $h\left(t_{0}\right)\ne0$. In particular, $\mathrm{ord}\left(f^{t_{0}}\left(x\right)\right)=\left(\alpha,m\right)$.
We have:
\[
\frac{\mathrm{d}f^{t}}{\mathrm{d}t}\Big|_{t=t_{0}}\left(x\right)=\xi\left(f^{t_{0}}\left(x\right)\right).
\]
The order of the left-hand side of this equation is bigger than or
equal to $\left(\alpha,m\right)$. But since $\mathrm{ord}\left(\xi\right)\prec\left(1,0\right)$,
the order of the right-hand side is strictly smaller than $\left(\alpha,m\right)$,
and we get a contradiction.\end{proof}
\begin{cor}
\label{cor:formula}Let $X=\xi\frac{d}{dx}$, $\xi\in\mathcal{L}$
(resp. $\mathcal{L}_{\mathfrak{D}}$) and let $\left(1,0\right)\preceq\text{ord}(\xi)$.
Then its $\mathcal{C}^{1}$-flow $(f^{t})_{t}$, $f^{t}\in\mathcal{L}^{H}$
(resp. $\mathcal{L}_{\mathfrak{D}}^{H}$), is given uniquely by the
formula: 
\[
f^{t}:=\exp(tX).\mathrm{id},\ t\in\mathbb{R}.
\]
\end{cor}
\begin{proof}
The proof follows by Propositions~\ref{prop: existence-formal-flow-parabolic}
and \ref{prop: existence-formal-flow-hyperbolic} and the uniqueness
result in Proposition~\ref{prop:unic}.\end{proof}
\begin{lem}
\label{lem:uniq} Let $X=\xi\frac{\mathrm{d}}{\mathrm{d}x}$, $\xi\in\mathcal{L}$
(resp. $\mathcal{L}_{\mathfrak{D}}$), be such that $\left(1,0\right)\prec\mathrm{\text{ord}}(\xi)$
or $\xi(x)=\lambda x+\mathrm{h.o.t.}$ with $\lambda<0$. The operators
$\exp(X)$ and $\log\exp(X)$ are weakly well-defined in $L(\mathcal{L})$
(resp. $L(\mathcal{L}_{\mathfrak{D}})$) and 
\[
\log\,\exp(X)=X.
\]
Moreover, in the case $\text{ord}(\xi)\succ(1,0)$, the operators
are well-defined.\end{lem}
\begin{proof}
The result in the case $\text{ord}(\xi)\succ(1,0)$ follows directly
from Proposition~\ref{def:exponential_logarithm_operator}, since
$X$ is a small operator in this case. The case $\text{ord}(\xi)=(1,0)$
is proven in Section~\ref{sub:fivefour}.\end{proof}
\begin{prop}[The convergence of the Taylor expansion]
\label{prop:taylor} Let $f\in\mathcal{L}^{H}$ (resp. $\mathcal{L}_{\mathfrak{D}}^{H}$)
be parabolic or hyperbolic contraction. Let $F=\mathrm{iso}(f)\in L(\mathcal{L})$
(resp. $L(\mathcal{L}_{\mathfrak{D}})$). Put $f=\mathrm{id}+\varepsilon$.
Then $F$ is weakly well-defined as the formal differential operator:
\[
F=\mathrm{Id}+\sum_{k=1}^{\infty}\frac{\varepsilon^{k}}{k!}\frac{\mathrm{d}^{k}}{\mathrm{d}x^{k}},
\]
Moreover, if $f$ is parabolic, then $F$ is well-defined by the above
series.
\end{prop}
Note that Proposition~\ref{prop:taylor} claims that in parabolic
and hyperbolic cases the \emph{Taylor expansions converge in $\mathcal{L}$}
(in the respective topologies). That is, for every $g\in\mathcal{L}$,
we can write: 
\begin{equation}
F.g\left(x\right)=g\circ f(x)=g\left(x+\varepsilon\left(x\right)\right)=g\left(x\right)+\sum_{k=1}^{\infty}\frac{\varepsilon\left(x\right)^{k}}{k!}g^{\left(k\right)}\left(x\right).\label{taylor}
\end{equation}

\begin{proof}
$(i)$ $f$ \emph{parabolic.} Since $\text{ord}(\varepsilon)\succ(1,0)$,
the Taylor expansion \eqref{taylor} converges in the formal topology
to $g\circ f$. Indeed, the orders $\text{ord}(\varepsilon^{k}g^{\left(k\right)})$
strictly increase by a fixed value $\left(0,0\right)\prec\text{ord}(h)-(1,0)$,
as $k$ increases. \\

$(ii)$ $f$ \emph{a hyperbolic contraction.} We prove that if $\text{ord}(\varepsilon)=(1,0)$
the Taylor expansion \eqref{taylor} converges in the weak topology.
Additionally, we prove that \emph{the coefficients of respective monomials
converge absolutely}.

Let $f(x)=\lambda x+\mathrm{h.o.t.}$ be hyperbolic, with $0<\lambda<1$.
Then $\varepsilon(x)=f(x)-x=(\lambda-1)x+\Psi(x)$, $\left(1,0\right)\prec\text{ord}(\Psi)$.
We prove that the Taylor expansion \eqref{taylor} for monomials $g(x)=x^{\alpha}$,
$\alpha>0,$ and $g(x)=\big(\frac{1}{-\log x}\big)^{m}$, $m\in\mathbb{Z}$,
converges in the weak topology. The convergence is then deduced for
all elements $g\in\mathcal{L}$, since products of absolutely convergent
series converge absolutely.

1. $g(x)=x^{\alpha}$. By definition of compositions in $\mathcal{L}$,
see Section~\ref{sub:Hahn-fields}, we have: 
\begin{equation}
\big(f(x)\big)^{\alpha}=\big(\lambda x+\Psi(x)\big)^{\alpha}=\lambda^{\alpha}x^{\alpha}\Big(1+\frac{\Psi(x)}{\lambda x}\Big)^{\alpha}=\lambda^{\alpha}x^{\alpha}\sum_{k=0}^{\infty}{\alpha \choose k}\lambda^{-k}\Big(\frac{\Psi(x)}{x}\Big)^{k}.\label{first}
\end{equation}
Since $\text{ord}(\frac{\Psi(x)}{x})\succ(1,0)$, the above series
converges in the formal topology.

Consider the series corresponding to the Taylor expansion \eqref{taylor}:
\begin{align}
\sum_{k=0}^{\infty}\frac{\big(x^{\alpha}\big)^{(k)}}{k!}\big((\lambda-1)x+\Psi(x)\big)^{k}= & \sum_{k=0}^{\infty}\frac{\big(x^{\alpha}\big)^{(k)}\cdot x^{k}}{k!}(\lambda-1)^{k}\Big(1+\frac{\Psi(x)}{(\lambda-1)x}\Big)^{k}=\nonumber \\
 & =\sum_{k=0}^{\infty}\frac{\alpha(\alpha-1)\cdots(\alpha-k+1)\cdot x^{\alpha}}{k!}(\lambda-1)^{k}\Big(1+\frac{\Psi(x)}{(\lambda-1)x}\Big)^{k}=\nonumber \\
 & =\sum_{k=0}^{\infty}{\alpha \choose k}x^{\alpha}(\lambda-1)^{k}\Big[\sum_{l=0}^{k}{k \choose l}\Big(\frac{\Psi(x)}{(\lambda-1)x}\Big)^{l}\Big].\label{second}
\end{align}
We show that the series converges in $\mathcal{L}$ in the weak topology
to $\big(f(x))^{\alpha}$ above. It can be easily seen that both the
monomials of \eqref{first} and of \eqref{second} belong to $S=\bigcup_{k\in\mathbb{N}_{0}}\mathcal{S}\Big(\big(\frac{\Psi(x)}{x}\big)^{k}x^{\alpha}\Big)$.
For every $k_{0}\in\mathbb{N}_{0}$, $x^{\alpha}\Big(\frac{\Psi(x)}{x}\Big)^{k_{0}}$
is present in infinitely many elements of \eqref{second}, but, due
to the fact that $|\lambda-1|<1$, its coefficient converges to the
coefficient of $x^{\alpha}\Big(\frac{\Psi(x)}{x}\Big)^{k_{0}}$ in
\eqref{first}: 
\[
\sum_{k=k_{0}}^{\infty}{\alpha \choose k}(\lambda-1)^{k-k_{0}}{k \choose k_{0}}=\sum_{k=k_{0}}^{\infty}{\alpha \choose k_{0}}{\alpha-k_{0} \choose k-k_{0}}(\lambda-1)^{k-k_{0}}={\alpha \choose k_{0}}(1+\lambda-1)^{\alpha-k_{0}}={\alpha \choose k_{0}}\lambda^{\alpha-k_{0}}.
\]
Moreover, the convergence is absolute.

On the other hand, since $\left(1,0\right)\prec\text{ord}\big(\frac{\Psi(x)}{x}\big)$,
for every monomial $x^{\beta}\boldsymbol{\ell}^{n}\in S$ there exists
$N\in\mathbb{N}$ such that $x^{\beta}\boldsymbol{\ell}^{n}\notin\mathcal{S}\Big(\big(\frac{\Psi(x)}{x}\big)^{k}x^{\alpha}\Big)$
for $k>N$. Together with the above analysis, this proves the convergence
of coefficients of every monomial in \eqref{second} to its coefficient
in \eqref{first}. That is, the Taylor expansion \eqref{second} converges
in the weak topology.

2. $g(x)=\big(\frac{1}{-\log x}\big)^{m}$. The proof of convergence
of Taylor expansion for $g(x)=\big(\frac{1}{-\log x}\big)^{m},\ m\in\mathbb{Z},$
follows the same idea, so we omit it.
\end{proof}
The next Lemma~\ref{lem:correspondance-diffeo-automorphism} is a
weaker version of the well-known diffeomorphism-isomorphism correspondence
for the algebra $\mathbb{C}[[x]]$ of formal power series, see \cite[Section 3A]{ilya}.
It is used, together with Proposition~\ref{rem:dif} below, to finish
the proof of Propositions \ref{prop: existence-formal-flow-parabolic}
and \ref{prop: existence-formal-flow-hyperbolic} concerning the correspondence
$h_{t}\leftrightarrow H^{t}$. Their Corollary~\ref{cor} is used
to prove uniqueness in Theorem B. 
\begin{lem}[Formal diffeomorphism - isomorphism correspondence for $\mathcal{L}$]
 \label{lem:correspondance-diffeo-automorphism}Let $B\in L(\mathcal{L})$
(resp. $L(\mathcal{L}_{\mathfrak{D}})$) be a morphism which is also
a (weakly) well-defined formal differential operator (in the sense
of Definition~\ref{def:form_diff}), and such that $h:=B\cdot\mathrm{id}\in\mathcal{L}^{0}$
is parabolic or a hyperbolic contraction. Then $B=\mathrm{iso}(h)$,
$h\in\mathcal{L}^{0}$ (resp. $\mathcal{L}_{\mathfrak{D}}^{0}$).\end{lem}
\begin{proof}
Since $B$ is a formal differential operator, put 
\begin{equation}
B=\mathrm{Id}+h_{1}\frac{\mathrm{d}}{\mathrm{d}x}+h_{2}\frac{\mathrm{d}^{2}}{\mathrm{d}x^{2}}+\cdots,\quad h_{i}\in\mathcal{L}\ (\text{resp. }\mathcal{L}_{\mathfrak{D}}),\ i\in\mathbb{N}.\label{field}
\end{equation}
Given an integer $p>1$, we compute $B\left(x^{p}\right)$ in two
different ways and compare. First, by \eqref{field}, we have: 
\[
B\left(x^{p}\right)=x^{p}+\sum_{n=1}^{p}h_{n}(x)\frac{\mathrm{d}^{n}(x^{p})}{\mathrm{d}x^{n}}=x^{p}+\sum_{n=1}^{p}h_{n}(x)n!\binom{p}{n}x^{p-n}.
\]
Since $B$ is a morphism, we have: 
\[
B\left(x^{p}\right)=\left(B\cdot x\right)^{p}=\left(x+h_{1}(x)\right)^{p}=x^{p}+\sum_{n=1}^{p}\binom{p}{n}h_{1}^{n}(x)x^{p-n}.
\]
Identifying these two expressions for every integer $p>1$, we see
that 
\begin{equation}
h_{n}=\frac{h_{1}^{n}}{n!},\ n\in\mathbb{N}\cup\{0\}.\label{eq:ha}
\end{equation}
Let $h=B\cdot\text{id}=\text{id}+h_{1}$, $h\in\mathcal{L}$ (resp.
$\mathcal{L}_{\mathfrak{D}}$). Let $H=\text{iso}(h)$, as defined
in Remark~\ref{ex:automorphism_associated}. It follows from Proposition~\ref{prop:taylor}
that: 
\[
H=\textrm{Id}+\sum_{n=1}^{\infty}\frac{h_{1}^{n}}{n!}\frac{\mathrm{d}^{n}}{\mathrm{d}x^{n}}.
\]
By \eqref{field} and \eqref{eq:ha}, $B=H$. Note that $B$ is additionally
well-defined by differential series \eqref{field} if $f$ is parabolic.\end{proof}
\begin{prop}
\label{rem:dif} Let $X=\xi\frac{\mathrm{d}}{\mathrm{d}x}$, $\xi\in\mathcal{L}$
(resp. $\mathcal{L}_{\mathfrak{D}}$) with $\left(1,0\right)\preceq\mathrm{ord}(\xi)$.
The operators $H^{t}=\exp(tX)$ from Propositions~\ref{prop: existence-formal-flow-parabolic}
and \ref{prop: existence-formal-flow-hyperbolic} are weakly well-defined
formal differential operators. If moreover $\mathrm{ord}(\xi)=(1,0)$,
they are well-defined formal differential operators.\end{prop}
\begin{proof}
Let $f\in\mathcal{L}$. Then by \eqref{eq:exponential-operator} we
have: 
\begin{align}
H^{t}\cdot f & =\exp(tX)\cdot f=f+t\xi f'+\frac{t^{2}}{2!}(\xi f')'\xi+\frac{t^{3}}{3!}\big((\xi f')'\xi\big)'\xi+\cdots=\nonumber \\
 & =f+t\xi f'+\frac{t^{2}}{2!}(\xi\xi'f'+\xi^{2}f'')+\frac{t^{3}}{3!}\big(\xi(\xi')^{2}f'+\xi^{2}\xi''f'+\xi^{2}\xi'f''+2\xi^{2}\xi'f''+\xi^{3}f'''\big)+\cdots\label{eq:difi}
\end{align}
We prove in both cases \big($\left(1,0\right)\prec\mathrm{ord}(\xi)$
and $\mathrm{ord}(\xi)=(1,0)$\big) that we can \emph{change the
order of the summation} in the respective topologies so that we group
the terms multiplying $f$, $f'$, $f''$, etc: 
\begin{align}
H^{t}\cdot f & =f+\big(t\xi+\frac{t^{2}}{2!}\xi\xi'+\frac{t^{3}}{3!}\xi(\xi')^{2}+\frac{t^{3}}{3!}\xi^{2}\xi''+\cdots\big)f'+\big(\frac{t^{2}}{2!}\xi^{2}+\frac{t^{3}}{3!}3\xi^{2}\xi'+\cdots\big)f''+\big(\frac{t^{3}}{3!}\xi^{3}+\cdots\big)f'''+\cdots\nonumber \\
 & =f+h_{1}\, f'+h_{2}f''+\mathrm{h.o.t.}\label{eq:regroup}
\end{align}
Obviously, by \eqref{eq:difi}, $h_{1}=H^{t}.\mathrm{id}\in\mathcal{L}$,
$h_{2}=\frac{1}{2}(H^{t}.x^{2}-x^{2})-xh_{1}\in\mathcal{L}$, etc.
Thus, $h_{n}\in\mathcal{L}$, $n\in\mathbb{N}$. \\

$(i)$ $\left(1,0\right)\prec\text{ord}(\xi)$. The orders of the
summands in \eqref{eq:regroup} increase by the fixed value $\text{ord}(\xi)-(1,0)\succ(0,0)$,
so \eqref{eq:regroup} converges in the formal topology in $\mathcal{L}$
to an element of $\mathcal{L}$. Moreover, it converges to the same
limit as \eqref{eq:difi}, since the difference of partial sums of
\eqref{eq:difi} and \eqref{eq:regroup} converges to zero in the
formal topology. Indeed, by Proposition~\ref{prop: existence-formal-flow-parabolic}
\eqref{eq:difi} the order of summands increases by the fixed value
$\text{ord}(\xi)-(1,0)\succ(0,0)$ also in \eqref{eq:difi}.\\

$(ii)$ $\text{ord}(\xi)=(1,0)$. Let us represent $H^{t}.f$ by the
following grid: 
\begin{equation}
\begin{array}{lllll}
f\\
*f'\\
*f' & *f''\\
*f' & *f'' & *f'''\\
*f' & *f'' & *f''' & *f^{(4)}\\
\vdots & \ldots
\end{array}\label{eq:rowcol}
\end{equation}
Here, $*$ denotes the \emph{coefficients} (transseries in $\xi$)
of the respective powers of $f$ in \eqref{eq:difi}. The first row
represents the first bracket in \eqref{eq:difi}, the second row the
second bracket in \eqref{eq:difi} etc.

Let us fix a monomial from the support $\mathcal{S}(H^{t}.f)$. The
order of terms remains the same by rows and by columns, in contrast
with the parabolic case. Therefore, a fixed monomial may appear in
every term of every row and of every column of \eqref{eq:rowcol}.
Nevertheless, we have proven in Proposition~\ref{prop: existence-formal-flow-hyperbolic}
that \eqref{eq:difi} converges in the weak topology, meaning exactly
that the coefficients of the given monomial converge when summation
is done by rows. We prove that we can change the order of the summation
of coefficients of the chosen monomial from summation by rows as in
\eqref{eq:difi} to summation by columns as in \eqref{eq:regroup}.
It would give us the convergence of \eqref{eq:regroup} in the weak
topology in $\mathcal{L}$ (to the same limit as \eqref{eq:difi}).

By the proof of Proposition~\ref{prop: existence-formal-flow-hyperbolic},
the coefficient of a fixed monomial of the support $\mathcal{S}(H^{t}.f)$
converges \emph{absolutely} in \eqref{eq:difi}. Moreover, we see
that each row of \eqref{eq:rowcol} contains only \emph{finitely many}
elements of $\mathcal{L}$. Consequently, a fixed monomial can appear
only finitely many times in each row. By the \emph{Moore-Osgood theorem}
stated on p. \pageref{eq:tii-2h} (or see \cite[Theorem 8.3]{rudin}),
we are allowed to change the order of the summation in our double
sum and to sum coefficients by columns.

The finitely generated case follows easily.\end{proof}
\begin{cor}
\label{cor} Let $X=\xi\frac{\mathrm{d}}{\mathrm{d}x}$, $\xi\in\mathcal{L}$
(resp. $\mathcal{L}_{\mathfrak{D}}$), be such that $\left(1,0\right)\prec\mathrm{ord}(\xi)$
or $\xi(x)=\lambda x+\mathrm{h.o.t.}$ with $\lambda<0$. Then, for
any $t\neq0$, the following two statements are equivalent: 
\begin{enumerate}[1., font=\textup, nolistsep, leftmargin=0.6cm]
\item $\exp(tX)\cdot\mathrm{id}=f$, 
\item $\exp(tX)\cdot h=h\circ f$, $h\in\mathcal{L}$ (resp. $\mathcal{L}_{\mathfrak{D}}$). 
\end{enumerate}
\end{cor}
\begin{proof} By Proposition~\ref{rem:dif}, the operator $\exp(tX)$
is a (weakly) well-defined formal differential operator. $(2)\Rightarrow(1)$
is obvious. We prove $(1)\Rightarrow(2)$. Suppose $(1)$ holds. By
Lemma~\ref{lem:correspondance-diffeo-automorphism}, $\exp(tX)$
is the isomorphism associated with $\exp(tX)\cdot\mathrm{id}=f$,
which proves $(2)$. \end{proof}

\subsection{Theorem B in the parabolic case}

\label{sub:fivethree} This section is dedicated to the precise statement
and the proof of Theorem B for parabolic elements of $\mathcal{L}$.
\begin{thm*}[Precise form of Theorem B for parabolic elements]
 Let $f\in\mathcal{L}^{H}$ (\emph{resp.} $f\in\mathcal{L}_{\mathfrak{D}}^{H}$)
be parabolic. Then there exists a unique formal vector field 
\[
X=\xi\frac{\mathrm{d}}{\mathrm{d}x},\ \xi\in\mathcal{L}\ (\text{ \emph{resp.} }\xi\in\mathcal{L}_{\mathfrak{D}}),
\]
such that $f$ embeds in its $\mathcal{C}^{1}$-flow. Moreover, 
\[
f=\exp(X)\cdot\mathrm{id}.
\]
Here, $\left(1,0\right)\prec\text{ord}(\xi)$, and $\exp(X)$ is well-defined
in $L(\mathcal{L})$ (resp. $L(\mathcal{L}_{\mathfrak{D}})$). 
\end{thm*}
Let $f\in\mathcal{L}$ (resp. $\mathcal{L}_{\mathfrak{D}}$) be a
parabolic element as in the statement of Theorem B. Let $F=\text{iso}(f)$:
\[
F.h=h\circ f,\ h\in\mathcal{L}\text{ (resp. \ensuremath{\mathcal{L}_{\mathfrak{D}}})}.
\]
We prove that the vector field $X$ is given by $X=\log F$. Note
that by Lemma~\ref{lem:expo-log_and_log-exp} and Proposition~\ref{prop:derivative},
the operator $X=\log F$ is a well-defined operator on $\mathcal{L}$
(resp. $\mathcal{L}_{\mathfrak{D}}$) and a derivation. The proof
is now given in three steps:
\begin{enumerate}
\item We prove in Lemma \ref{lem:X-differential-operator} that $X=\log F$
is a formal differential operator $\sum_{k}h_{k}\frac{\mathrm{d}^{k}}{\mathrm{d}x^{k}}$,
$h_{k}\in\mathcal{L}$ (resp. $\mathcal{L}_{D}$). 
\item Since $X$ is a derivation and at the same time a formal differential
operator of the above form, we prove that $X$ is necessarily a vector
field. Moreover, we prove that $f$ is the time-one map of $X$. 
\item We prove the uniqueness of the formal vector field whose time-one
map is $f$. \end{enumerate}
\begin{lem}
\label{lem:H-differential-operator} Let $f\in\mathcal{L}^{0}$ (resp.
$\mathcal{L}_{\mathfrak{D}}^{0}$) be parabolic and let $F=\text{iso}(f)\in L(\mathcal{L})$
(resp. $L(\mathcal{L}_{\mathfrak{D}})$). Let $H=F-\mathrm{Id}$.
Then all the iterates $H^{k}$ can be written as well-defined formal
differential operators on $\mathcal{L}$ (resp. $\mathcal{L}_{\mathfrak{D}}$):
\begin{equation}
H^{k}=\sum_{\ell=1}^{\infty}h_{\ell}^{k}\frac{\mathrm{d}^{\ell}}{{\mathrm{d}x}^{\ell}},\ h_{\ell}^{k}\in\mathcal{L}\text{ (resp. \ensuremath{\mathcal{L}_{\mathfrak{D}}})},\ k\in\mathbb{N}.\label{eq:hha}
\end{equation}
\end{lem}
\begin{proof}
Let $f=\mathrm{id}+h$, $h\in\mathcal{L}$ with $\left(1,0\right)\prec\text{ord}(h)$.
The lemma is proven by induction. The induction basis ($k=1$) follows
easily by Taylor expansion: 
\[
H\cdot g=g\circ f-g=\sum_{\ell=1}^{\infty}\frac{h^{\ell}}{\ell!}\frac{\mathrm{d}^{\ell}g}{{\mathrm{d}x}^{\ell}},\ g\in\mathcal{L}.
\]
Thus, $H=\sum_{\ell=1}^{\infty}h_{\ell}^{0}\frac{\mathrm{d}^{\ell}}{{\mathrm{d}x}^{\ell}},$
with the coefficients $h_{\ell}^{0}=\frac{h^{\ell}}{\ell!}\in\mathcal{L}$,
$\ell\in\mathbb{N}$. Assume that the operators $H^{m}$, $m\leq k$
can be written in the form \eqref{eq:hha}, with formal convergence
on $\mathcal{L}$. Note that the formal convergence of series $H^{m}\cdot g$
from \eqref{eq:hha} is equivalent to asking that the orders of summands
$\text{ord}(h_{\ell}^{m}g^{(\ell)})$ infinitely increase as $\ell\to\infty$.
We prove \eqref{eq:hha} for the operator $H^{k+1}$. By Taylor expansion,
we have: 
\begin{align}
H^{k+1}\cdot g\,(x)=H(H^{k}\cdot g)\,(x)=H^{k}\cdot g\,(x+h(x))-H^{k}\cdot g\,(x) & =\sum_{i=1}^{\infty}\frac{h(x)^{i}}{i!}\frac{\mathrm{d}^{i}(H^{k}.g)}{{\mathrm{d}x}^{i}}\nonumber \\
 & =\sum_{i=1}^{\infty}\frac{h^{i}}{i!}\frac{\mathrm{d}^{i}}{\mathrm{d}x^{i}}\Big(\sum_{\ell=1}^{\infty}h_{\ell}^{k}\frac{\mathrm{d}^{\ell}g}{{\mathrm{d}}x^{\ell}}\Big)=\sum_{i=1}^{\infty}\Big(\sum_{\ell=1}^{\infty}h_{i\ell}^{k}\frac{\mathrm{d}^{\ell}g}{{\mathrm{d}x}^{\ell}}\Big),\label{eq:tii-1}
\end{align}
with $h_{i\ell}^{k}\in\mathcal{L}$, $i,\,\ell\in\mathbb{N}$. We
represent the double sum by the following grid: 
\begin{equation}
H^{k+1}\cdot g\ :\qquad\begin{array}{ccccc}
 & \stackrel{\ell}{\rightarrow}\\
i\downarrow & h_{11}\frac{\mathrm{d}g}{{\mathrm{d}x}} & h_{12}\frac{\mathrm{d}^{2}g}{{\mathrm{d}x}^{2}} & h_{13}\frac{\mathrm{d}^{3}g}{{\mathrm{d}x}^{3}} & \ldots\\[0.1cm]
 & h_{21}\frac{\mathrm{d}g}{\mathrm{d}x} & h_{22}\frac{\mathrm{d}^{2}g}{{\mathrm{d}x}^{2}} & h_{23}\frac{\mathrm{d}^{3}g}{{\mathrm{d}x}^{3}} & \ldots\\[0.1cm]
 & h_{31}\frac{\mathrm{d}g}{\mathrm{d}x} & h_{32}\frac{\mathrm{d}^{2}g}{{\mathrm{d}x}^{2}} & h_{33}\frac{\mathrm{d}^{3}g}{{\mathrm{d}x}^{3}} & \ldots\\[0.1cm]
 & \vdots & \vdots & \vdots
\end{array}\label{eq:gri}
\end{equation}
The order of the summation in \eqref{eq:tii-1} is by rows. Since
$f$ is parabolic, the Taylor expansion in \eqref{eq:tii-1} converges
in the formal topology. Moreover, we assumed formal convergence of
the differential expansion of $H^{k}.g$. Therefore, the order of
the terms increases indefinitely along the rows and along the columns
of \eqref{eq:gri}. The monomials up to some fixed order exist only
in finitely many first rows and columns. Consequently, we are allowed
to change the order of the summation from summation by rows to summation
by columns, and the following sum converges in $\mathcal{L}$ in the
formal topology: 
\begin{equation}
\sum_{\ell=1}^{\infty}\Big(\sum_{i=1}^{\infty}h_{i\ell}^{k}\frac{\mathrm{d}^{\ell}g}{{\mathrm{d}x}^{\ell}}\Big)=\sum_{\ell=1}^{\infty}h_{\ell}^{k+1}\frac{\mathrm{d}^{\ell}g}{{\mathrm{d}x}^{\ell}}.\label{eq:tii-2}
\end{equation}
Here, $h_{\ell}^{k+1}:=\sum_{i=1}^{\infty}h_{i\ell}^{k}$. By increasing
orders, we immediately obtain $h_{\ell}^{k+1}\in\mathcal{L}$. The
difference of partial sums of \eqref{eq:tii-1} and \eqref{eq:tii-2}
converges to zero in the formal topology, so the limit of both series
is the same, that is, $H^{k+1}\cdot g$. Thus we have: 
\[
H^{k+1}=\sum_{\ell=1}^{\infty}h_{\ell}^{k+1}\frac{\mathrm{d}^{\ell}}{{\mathrm{d}x}^{\ell}},\ h_{\ell}^{k+1}\in\mathcal{L},
\]
with the formal convergence.

Note additionally that from \eqref{eq:hha} we have that $h_{1}^{k}=H^{k}\cdot\mathrm{id}$,
$h_{2}^{k}=\frac{1}{2}H^{k}\cdot x^{2}-xh_{1}^{k}$, etc. by induction,
$k\in\mathbb{N}$. The finitely generated case follows directly. \end{proof}
\begin{lem}
\label{lem:X-differential-operator} Let $f\in\mathcal{L}^{0}$ (resp.
$\mathcal{L}_{D}^{0}$) be parabolic. Let $F=\text{iso}(f)\in L(\mathcal{L})$
(resp. $L(\mathcal{L}_{\mathfrak{D}})$) and $H=F-\mathrm{Id}$. Let
$X=\log F=\log(\mathrm{Id}+H)$. Then $X$ can be written as a well-defined
formal differential operator on $\mathcal{L}$ (resp. $\mathcal{L}_{\mathfrak{D}}$):
\begin{equation}
X=\sum_{\ell=1}^{\infty}h_{\ell}\frac{\mathrm{d}^{\ell}}{{\mathrm{d}x}^{\ell}},\ h_{\ell}\in\mathcal{L}\text{ (resp. \ensuremath{\mathcal{L}_{\mathfrak{D}}})}.\label{termw}
\end{equation}
\end{lem}
\begin{proof}[Proof of Step (1) ]
By Lemma~\ref{lem:H-differential-operator}, all operators $H^{k}$,
$k\in\mathbb{N}$, can be written as differential operators, with
convergence in the formal topology in $\mathcal{L}$. By Lemma \ref{lem:expo-log_and_log-exp},
the operator $X\cdot g$ given by the logarithmic series 
\begin{align}
X\cdot g= & \log(\text{Id}+H)\cdot g=H\cdot g-\frac{1}{2}H^{2}\cdot g+\frac{1}{3}H^{3}\cdot g+\cdots\label{s}
\end{align}
also converges in the formal topology. We put the convergent expansions
\eqref{eq:hha} for $H^{k}\cdot g$ in \eqref{s}. Due to the formal
convergence of all series, proceeding exactly as in Lemma~\ref{lem:H-differential-operator},
we are allowed to change the order of the summation \emph{from rows
to columns}, that is, to group together the terms in front of the
same derivative of $g$. The new sum again converges formally to the
same limit: 
\begin{align*}
X\cdot g & =H\cdot g-\frac{1}{2}H^{2}\cdot g+\frac{1}{3}H^{3}\cdot g+\cdots=\\
 & =\big(h_{1}^{1}g'+h_{2}^{1}g''+h_{3}^{1}g'''+\cdots\big)+\big(h_{1}^{2}g'+h_{2}^{2}g''+h_{3}^{2}g'''+\cdots\big)+\big(h_{1}^{3}g'+h_{2}^{3}g''+h_{3}^{3}g'''+\cdots\big)+\cdots=\\
 & =\big(h_{1}^{1}+h_{1}^{2}+h_{1}^{3}+\cdots\big)\cdot g'+\big(h_{2}^{1}+h_{2}^{2}+h_{2}^{3}+\cdots\big)\cdot g''+\big(h_{3}^{1}+h_{3}^{2}+h_{3}^{3}+\cdots\big)\cdot g'''+\cdots=\sum_{\ell=1}^{\infty}h_{\ell}g^{(\ell)}.
\end{align*}
Here, $h_{\ell}:=\sum_{k=1}^{\infty}h_{\ell}^{k}\in\mathcal{L}$,
since the orders of the terms increase indefinitely. 
\end{proof}

\begin{proof}[Proof of Step (2)]
\emph{}We now finish the proof of Theorem B. By Lemma~\ref{lem:X-differential-operator},
we have that $X$ is a formal differential operator: 
\begin{equation}
X=\sum_{\ell=1}^{\infty}h_{\ell}\frac{\mathrm{d}^{\ell}}{{\mathrm{d}x}^{\ell}},\ h_{\ell}\in\mathcal{L}\text{ (resp. \ensuremath{\mathcal{L}_{\mathfrak{D}}})}.\label{eq:fiveone}
\end{equation}
We now prove that, due to the Leibniz's property of $X$ ($X$ is
a derivation by Proposition~\ref{prop:derivative}), all the $h_{\ell}$
except $h_{1}$ \emph{vanish}. We apply \eqref{eq:fiveone} successively
to \emph{test monomials} $x^{n},\ n\in\mathbb{N}$, and use the Leibniz's
rule. We deduce from \eqref{eq:fiveone} applied to $g=\mathrm{id}$
that: 
\[
h_{1}=X\cdot\mathrm{id}.
\]
We then apply \eqref{eq:fiveone} to $g(x)=x^{2}$. It follows from
Leibniz's rule that: 
\[
X\cdot x^{2}=2xh_{1}(x)=h_{1}(x)\cdot2x+2h_{2}(x).
\]
It follows that $h_{2}\equiv0$, and, by induction, that $h_{i}\equiv0$,
$i\geq2$. Putting $\xi:=h_{1}$, \eqref{eq:fiveone} becomes: 
\[
X=\xi\frac{\mathrm{d}}{\mathrm{d}x},\ \xi\in\mathcal{L},
\]
which is the desired vector field. By Lemma~\ref{lem:expo-log_and_log-exp},
we have: 
\[
\exp(X)\cdot\mathrm{id}=\exp(\log F)\cdot\mathrm{id}=F\cdot\mathrm{id}=f,
\]
so $f$ is the time-one map of $X$ in the sense of Definition~\ref{def:time-one-map}.
The finitely generated case follows easily. 
\end{proof}

\begin{proof}[Proof of Step $(3)$]
 Let $X=\xi\frac{d}{dx}$, $\xi\in\mathcal{L}$, be any vector field
in whose $\mathcal{C}^{1}$-flow $f$ embeds. Since $f$ is parabolic,
it follows from Proposition~\ref{prop:unic} that $\left(1,0\right)\prec\text{ord}(\xi)$.
By Proposition~\ref{prop: existence-formal-flow-parabolic}, $\exp(tX)\cdot\mathrm{id}$
defines a $\mathcal{C}^{1}$-flow of $X$. Since by Proposition~\ref{prop:unic}
the $\mathcal{C}^{1}$-flow of $X$ is unique, it follows that 
\[
f=\exp(X)\cdot\mathrm{id}.
\]
By Corollary~\ref{cor}, it now necessarily follows that 
\[
\exp(X)\cdot h=h\circ f=F\cdot h,\ h\in\mathcal{L},
\]
so $\exp(X)=F$. By Lemma~\ref{lem:uniq}, $X$ is uniquely given
by $X=\log F$.
\end{proof}

\subsection{Theorem B in the hyperbolic case}

\label{sub:fivefour}
\begin{thm*}[Precise form of Theorem B for hyperbolic elements]
 Let $f\in\mathcal{L}^{H}$ (\emph{resp.} $f\in\mathcal{L}_{\mathfrak{D}}^{H}$)
be hyperbolic. Then there exists a unique formal vector field on $\mathbb{R}$,
\[
X=\xi\frac{\mathrm{d}}{\mathrm{d}x},\ \xi\in\mathcal{L}\text{ (\emph{resp.} }\xi\in\mathcal{L}_{\mathfrak{D}}),
\]
such that $f$ embeds in the flow of $X$ as its time-one map in the
sense of Definition~\ref{def:time-one-map}. Moreover, 
\[
f=\exp(X)\cdot\mathrm{id}.
\]
Here, $\mathrm{ord}(\xi)=(1,0)$ and $\exp(X)$ is a weakly well-defined
operator in $\mathcal{L}$ (resp. $\mathcal{L}_{\mathfrak{D}}$).
\end{thm*}
Let $f(x)=\lambda x+\mathrm{h.o.t.}\in\mathcal{L}^{H}$, $\lambda>0,\ \lambda\neq1$.
In the proof of the theorem we suppose without loss of generality
that $f$ is a formal contraction, that is $0<\lambda<1$. If $\lambda>1$
(a formal expansion), we consider its formal inverse $f^{-1}\in\mathcal{L}^{H}$,
which is a formal contraction. Obviously, $f^{-1}$ embeds in the
flow of $X$ (in the sense of Theorem B) if and only if $f$ embeds
in the flow of $-X$.

In the previous section, the standard notion of a \emph{small operator}
was used to prove the convergence of operator power series in $L(\mathcal{L})$,
see Proposition~\ref{def:exponential_logarithm_operator}. The convergence
of the series was in \emph{formal topology} on $\mathcal{L}$. In
the parabolic case, this notion was sufficient to obtain the embedding
result of Theorem B. Here, we introduce the definition of \emph{small
operator in the weak sense} with the aim of giving a meaning to operator
power series in $L(\mathcal{L})$ under weaker assumptions, needed
for  case. The convergence of operator power series will be in the
\emph{weak topology} on $\mathcal{L}$.
\begin{defn}[Small operator \emph{in the weak sense} with respect to a sequence]
\label{def:sows} An operator $B:\mathcal{L}\to\mathcal{L}$ is \emph{small
in the weak sense with respect to the sequence $(c_{k})_{k\in\mathbb{N}_{0}}$
of real numbers} if: 
\begin{enumerate}[1., font=\textup, nolistsep, leftmargin=0.6cm]
\item there exists a well-ordered set $R\subseteq\mathbb{R}_{\ge0}\times\mathbb{Z}$
of exponents \emph{equal to} or strictly bigger than $\left(0,0\right)$
such that $\mathcal{S}\left(B\cdot f\right)\subseteq\mathcal{S}\left(f\right)+R$,
for every $f\in\mathcal{L}$;
\item For every $f\in\mathcal{L}$ and for every $(\alpha,m)\in\mathcal{S}(f)+\langle R\rangle$,
there exists a sequence $(C_{\alpha,m}^{k})_{k\in\mathbb{N}_{0}}$
of strictly positive numbers such that: 
\begin{equation}
\big|[B^{k}\cdot f]_{\alpha,m}\big|\leq C_{\alpha,m}^{k},\label{eq:so}
\end{equation}
and such that the series 
\begin{equation}
\sum_{k=0}^{\infty}c_{k}C_{\alpha,m}^{k}\label{eq:ser_conv}
\end{equation}
converges absolutely. Here, $[B^{k}\cdot f]_{\alpha,m}$ denotes the
\emph{coefficient} of monomial $x^{\alpha}\boldsymbol{\ell}^{m}$
in $B^{k}\cdot f$ (the notation introduced in Section~\ref{sec:topologies}). 
\end{enumerate}
An operator $B:\mathcal{L}_{\mathcal{D}}\to\mathcal{L}_{\mathcal{D}}$
is \emph{small in the weak sense with respect to the sequence $(c_{k})_{k}$}
if the set $R$ is in addition of finite type.
\end{defn}
Notice that from (1) we obtain by induction that $\mathcal{S}(B^{k}\cdot f)\subseteq\mathcal{S}(f)+\langle R\rangle,\ k\in\mathbb{N}$,
so $(2)$ makes sense.
\begin{prop}[A version of Proposition~\ref{lem:series_small_operator} in the
weak sense]
\label{prop:small_weak_series} Let $(c_{k})_{k\in\mathbb{N}_{0}}$
be a sequence of real numbers and let $B\in L(\mathcal{L})$ (resp.
$L(\mathcal{L}_{D})$) be a small operator in the weak sense with
respect to the sequence $(c_{k})$. Then an operator $B\in L(\mathcal{L})$
($L(\mathcal{L}_{D})$) is \emph{weakly well-defined} by the series
\begin{equation}
B:=\sum_{k=0}^{\infty}c_{k}B^{k}.\ \label{eq:seri1}
\end{equation}
\end{prop}
\begin{proof}
The proof is straightforward by \eqref{eq:so} and the absolute convergence
of the series $\sum_{k=0}^{\infty}c_{k}C_{\alpha,m}^{k}$. 
\end{proof}
We state and prove in this section the analogous of Lemmas~\ref{lem:expo-log_and_log-exp},
\ref{lem:uniq}, \ref{lem:H-differential-operator} and \ref{lem:X-differential-operator}.
We prove Proposition~\ref{prop: existence-formal-flow-hyperbolic}.
All of them are needed for the proof of Theorem B in the hyperbolic
case. Then the proof in the hyperbolic case follows the same steps
as the proof in the parabolic case, but using the corresponding weak
notions.
\begin{lem}[Lemma~\ref{lem:expo-log_and_log-exp} for hyperbolic elements]
\label{lem:expo_log_hyp} Let $f=\lambda x+\mathrm{h.o.t.}\in\mathcal{L}^{H}$
(resp. $f\in\mathcal{L}_{\mathfrak{D}}^{H}$) be a hyperbolic contraction
$(0<\lambda<1)$. Let $F=\mathrm{iso}(f)\in L(\mathcal{L})$ (resp.
$L(\mathcal{L}_{\mathfrak{D}})$), as in Remark~\ref{ex:automorphism_associated}.
Then $\log F$ is a \emph{weakly well-defined} operator in $L(\mathcal{L})$
(resp. $L(\mathcal{L}_{\mathfrak{D}})$).\end{lem}
\begin{proof}
Put $F=\mathrm{Id}+H$, $H:\mathcal{L}\to\mathcal{L}$. Then 
\begin{equation}
\log F=\log(\mathrm{Id}+H)=\sum_{k=1}^{\infty}(-1)^{k+1}\frac{H^{k}}{k}.\label{eq:logg}
\end{equation}
We prove that operator $H\in L(\mathcal{L})$ is small in the weak
sense (see Definition~\ref{def:sows}) with respect to the sequence
$\big(\frac{(-1)^{k+1}}{k}\big)_{k}$. Applying Proposition~\ref{prop:small_weak_series},
we conclude that the series $\log F$ is a weakly well-defined operator
$\log F:\mathcal{L}\to\mathcal{L}$.

Take $g\in\mathcal{L}$. Then $H\cdot g=g\circ f-g=g\big(\lambda x+\psi(x)\big)-g$.
For $x^{\alpha}\boldsymbol{\ell}^{m}\in\mathcal{S}(g)$, we compute:
\begin{align}
(\lambda x+ & \psi(x))^{\alpha}\boldsymbol{\ell}(\lambda x+\psi(x))^{m}-x^{\alpha}\ell^{m}=\nonumber \\
 & =\lambda^{\alpha}x^{\alpha}\boldsymbol{\ell}^{m}\Big(1+\lambda^{-1}x^{-1}\psi(x)\Big)^{\alpha}\Big(1-\log\lambda\boldsymbol{\ell}-\boldsymbol{\ell}\log\big(1+\lambda^{-1}x^{-1}\psi(x)\big)\Big)^{-m}-\, x^{\alpha}\ell^{m}.\label{eq:com}
\end{align}
Here, $\psi(x)=f(x)-\lambda x$, with $\left(1,0\right)\prec\text{ord}(\psi)$.
We conclude that, for every $g\in\mathcal{L}$, 
\[
\mathcal{S}(H\cdot g)\subseteq\mathcal{S}(g)+R,
\]
where $R$ is a sub-semigroup of $\mathbb{R}_{\geq0}\times\mathbb{Z}$
generated by $(0,1)$ and $(\beta-1,\ell)$ for $(\beta,\ell)\in\mathcal{S}(\psi)$
and containing $(0,0)$. Obviously, since $\left(1,0\right)\prec\text{ord}(\psi)$,
all the elements of $R$ except $(0,0)$ are of order strictly bigger
than $(0,0)$. By Neumann's lemma, $R$ is well-ordered.

Now take any $(\alpha,m)\in R+\mathcal{S}(g)$. If $(\alpha,m)\notin\mathcal{S}(H^{k}.g)$,
for any $k\in\mathbb{N}$, then \eqref{eq:so} holds for every $C_{\alpha,m}>0$.
Suppose that there exists some $\ell\in\mathbb{N}$ such that $(\alpha,m)\in\mathcal{S}(H^{\ell}.g)$.
It can be seen from \eqref{eq:com}, since $\lambda\neq1$, that $(\alpha,m)$
then appears in every $H^{k}\cdot g$, $k\geq\ell$ (unlike the parabolic
case). We prove nevertheless that the coefficient of $x^{\alpha}\boldsymbol{\ell}^{m}$
evolves \emph{controllably} with $k$ in the sense of \eqref{eq:so}
and \eqref{eq:ser_conv}.

To this end, we analyse in what ways we can obtain the monomial $x^{\alpha}\boldsymbol{\ell}^{m}$
in iterates $H^{k}\cdot g$, $k\in\mathbb{N}$. The monomial $x^{\alpha}\boldsymbol{\ell}^{m}$
\emph{evolves} from some \emph{initial monomial} $x^{\beta}\boldsymbol{\ell}^{n}\in\mathcal{S}(g)$
through iterates $H\cdot g$, $H^{2}\cdot g$, etc. The evolution
of coefficients and exponents from one iterate $H^{k}\cdot g$ to
the next one $H^{k+1}\cdot g$ is described by \eqref{eq:com}. Hence,
the rule which governs this evolution is the following: \emph{in each
step, the corresponding monomial $x^{\gamma}\boldsymbol{\ell}^{n}$
either stays the same while its coefficient is multiplied by $\lambda^{\gamma}-1$,
or it is multiplied by a monomial from $R\setminus\{(0,0)\}$}. We
can thus write a finite \emph{chain of changes} corresponding to this
evolution: 
\begin{equation}
x^{\beta}\boldsymbol{\ell}^{n}\rightarrow x^{\beta+\gamma_{1}}\boldsymbol{\ell}^{n+n_{1}}\rightarrow x^{\beta+\gamma_{1}+\gamma_{2}}\boldsymbol{\ell}^{n+n_{1}+n_{2}}\rightarrow\cdots\rightarrow x^{\beta+\gamma_{1}+\cdots+\gamma_{r}}\boldsymbol{\ell}^{n+n_{1}+\cdots+n_{r}}=x^{\alpha}\boldsymbol{\ell}^{m},\label{eq:comb}
\end{equation}
where $(\beta,n)\in\mathcal{S}(g)$, $(\gamma_{i},n_{i})\in R\setminus\{(0,0)\}$,
$i=1,\ldots,r$, $r\in\mathbb{N}_{0}$. By Neumann's lemma, for every
$(\alpha,m)\in R+\mathcal{S}(g)$, there exist only \emph{finitely
many} chains describing the evolution of monomial $x^{\alpha}\ell^{m}$
from elements of $\mathcal{S}\cdot g$. Given a chain as in \eqref{eq:comb},
the integer $r\in\mathbb{N}_{0}$ as well as the pairs $(\beta,n)$
and $(\gamma_{i},n_{i}),\ i=1,\ldots,r,$ do not depend on $k$. We
fix one such chain \eqref{eq:comb} and prove \eqref{eq:so} and \eqref{eq:ser_conv}
only for this chain (in the end we sum up the coefficients of finitely
many chains contributing to $x^{\alpha}\boldsymbol{\ell}^{m}$ and
conclude for the whole $x^{\alpha}\boldsymbol{\ell}^{m}$). For the
coefficient of $x^{\alpha}\boldsymbol{\ell}^{m}$ obtained by this
chain in $H^{k}\cdot g$, we obtain, using \eqref{eq:com}: 
\begin{align*}
 & \big|[H^{k}\cdot g]_{\alpha,m}\big|={k \choose r}\cdot a(\lambda^{\beta}-1)^{k_{1}}\cdot C_{1}\cdot(\lambda^{\beta+\gamma_{1}}-1)^{k_{2}}\cdot C_{2}\cdots(\lambda^{\beta+\gamma_{1}+\cdots+\gamma_{r-1}}-1)^{k_{r}}\cdot C_{r}\cdot(\lambda^{\alpha}-1)^{k_{r+1}},\\
 & \hspace{14cm}\ k\geq k_{0}.
\end{align*}
Here, $k_{0}$ is the index of the first iterate $H^{k_{0}}\cdot g$
in which $x^{\alpha}\boldsymbol{\ell}^{m}$ obtained by chain \eqref{eq:comb}
appears, and $a\in\mathbb{R}$ is the coefficient of the initial monomial,
$ax^{\beta}\boldsymbol{\ell}^{n}$. We choose $r$ iterates (out of
$k$ in total) in which the monomial changes: $H^{k_{1}}\cdot g$,
$H^{k_{1}+k_{2}}\cdot g,\ldots,H^{k_{1}+\cdots+k_{r}}\cdot g$. Note
also that $k_{1}+\cdots+k_{r}+k_{r+1}=k-r$. Note also that the change
of the coefficient in $r$ steps in which the monomial changes depends
only on the fixed chain and not on $k$. That is why, in the above
formula, we multiply by numbers $C_{1},\ldots,C_{r}$ which depend
on $f$ and on the given chain, but not on $k$. The remaining $(k-r)$
steps are characterized by multiplications by respective $\lambda^{\gamma}-1$,
$k-r$ times in total. Therefore, 
\begin{equation}
\big|[H^{k}\cdot g]_{\alpha,m}\big|\leq{k \choose r}\cdot C_{f}^{1}\cdot|\lambda^{\alpha}-1|^{k-r}\leq{k \choose r}C_{f}|\lambda^{\alpha}-1|^{k},\quad k\geq k_{0}.\label{eq:hahh}
\end{equation}
Here, $C_{f}^{1},\ C_{f}$ are \emph{constants} that depend only on
initial $f$, on $g$ and on the fixed chain; they are independent
of $k$.

For the given chain \eqref{eq:comb} contributing to the coefficient
of $x^{\alpha}\boldsymbol{\ell}^{m}$, the inequality \eqref{eq:so}
in Definition~\ref{def:sows} is, by \eqref{eq:hahh}, satisfied
with 
\[
C_{\alpha,m}^{k}={k \choose r}\cdot C_{f}\cdot|\lambda^{\alpha}-1|^{k},\ \ k\geq k_{0}.
\]
Moreover, the series 
\[
\sum_{k=k_{0}}^{\infty}\frac{(-1)^{k+1}}{k}C_{\alpha,m}^{k}=\sum_{k=k_{0}}^{\infty}\frac{(-1)^{k+1}}{k}{k \choose r}\cdot C_{f}\cdot|\lambda^{\alpha}-1|^{k}
\]
converges absolutely (which can easily be checked by, for example,
the ratio test) since $0<\lambda<1$. The operator $H\in L(\mathcal{L})$
is therefore \emph{small in the weak sense} with respect to the sequence
$\big(\frac{(-1)^{k+1}}{k}\big)_{k}$. By Proposition \ref{prop:small_weak_series},
the series $(\log F)\cdot g$, $g\in\mathcal{L}$, converges in $\mathcal{L}$
in the weak topology to an element of $\mathcal{L}$. The operator
$\log F\in L(\mathcal{L})$ is thus weakly well-defined.

Finally, the claim in finitely generated case ($\mathcal{L}_{\mathfrak{D}}$)
follows easily, since $R$ is then finitely generated.
\end{proof}
We prove here the Proposition~\ref{prop: existence-formal-flow-hyperbolic}
stated in Subsection~\ref{sub:fivetwo}. The problem consists in
giving a meaning to the formal time-one map of a vector field $X=\xi\frac{\mathrm{d}}{\mathrm{d}x}$
in the case where $\xi=\lambda x+\mathrm{h.o.t.}$ is \emph{hyperbolic}.
The question is: does $X$ admit a formal one-parameter flow? Hence
we need to study the convergence in $\mathcal{L}$ of the exponential
of $X=\xi\frac{\mathrm{d}}{\mathrm{d}x}$. The problem, compared to
the case $\left(1,0\right)\prec\mathrm{ord}(\xi)$ proven in Subection~\ref{sub:fivetwo},
is that the operator $X=\xi\frac{\mathrm{d}}{\mathrm{d}x}$ is not
\emph{small} any more if $\mathrm{ord}(\xi)=(1,0)$. Hence, its exponential
$H^{t}=\exp\left(tX\right)$, $t\in\mathbb{R}$, is not a well-defined
operator in $L(\mathcal{L})$. Moreover, given $f\in\mathcal{L}$,
the formula: 
\begin{equation}
H^{t}\cdot f=\exp(tX)\cdot f=f+tf'\xi+\frac{t^{2}}{2!}\left(f'\xi\right)'\xi+\frac{t^{3}}{3!}\left(\left(f'\xi\right)'\xi\right)'\xi+\cdots\label{eq:formula-Ht}
\end{equation}
does not converge in $\mathcal{L}$ even with respect to the weaker
product topology with respect to the discrete topology. Indeed, a
fixed monomial is present in infinitely many terms of the sum \eqref{eq:formula-Ht}.
Nevertheless, we show here that operator $X$ is \emph{small in the
weak sense with respect to the sequence $\big(\frac{t^{n}}{n!}\big)_{n}$},
see Definition~\ref{def:sows}. Applying Proposition~\ref{prop:small_weak_series},
we conclude that $H^{t}=\exp\left(tX\right)$ is a \emph{weakly well-defined}
operator in $L(\mathcal{L})$. In other words, the series \eqref{eq:formula-Ht}
converges in $\mathcal{L}$ in the weak topology.
\begin{proof}[Proof of Proposition~\ref{prop: existence-formal-flow-hyperbolic}]
. Let $X=\xi\frac{\mathrm{d}}{\mathrm{d}x}$, with 
\begin{align*}
\xi=\lambda\cdot\mathrm{id}+\psi,\ \lambda\neq0.
\end{align*}
Here, $\psi\in\mathcal{L}$ such that $(1,0)\prec\text{ord}(\psi)$.
Recall that 
\[
H^{t}\cdot f=\sum_{n=0}^{\infty}\frac{t^{n}}{n!}X^{n}\cdot f.
\]
Since $X\cdot f=\xi f',\ f\in\mathcal{L}$, we conclude that $\mathcal{S}(X\cdot f)\subseteq\mathcal{S}(f)+R$,
where $R$ is a sub-semigroup of $\mathbb{R}_{\geq0}\times\mathbb{Z}$
generated by $(0,1)$, $(\beta-1,k)$ for $(\beta,k)\in\mathcal{S}(\psi)$,
and containing $(0,0)$. Note that $\left(0,0\right)\prec\text{ord}(\beta-1,k)$,
since $\left(1,0\right)\prec\text{ord}(\psi)$. Each $X^{n}\cdot f$
is by \eqref{eq:formula-Ht} a sum of $2^{n}$ terms of the type 
\begin{equation}
\big(((f'\cdot*)'\cdot*)'\ldots\big)'\cdot*,\label{str-1}
\end{equation}
with $n$ stars representing either $\lambda x$ or $\psi$.

Fix any $(\alpha,m)\in\cup_{n\in\mathbb{N}_{0}}\mathcal{S}(X^{n}\cdot f)\subseteq\mathcal{S}(f)+R$.
Then $(\alpha,m)\in\mathcal{S}(X^{n_{0}}\cdot f)$, for some $n_{0}\in\mathbb{N}$.
Since $f$ is hyperbolic, we see by \eqref{str-1} that $(\alpha,m)\in\mathcal{S}(X^{n}\cdot f)$,
for infinitely many $n\geq n_{0}$ in general (the coefficient of
$x^{\alpha}\boldsymbol{\ell}^{m}$ may sometimes vanish due to cancellations).
In order to prove \eqref{eq:so} and \eqref{eq:ser_conv}, we analyse
the coefficient of $x^{\alpha}\boldsymbol{\ell}^{m}$ in $X^{n}\cdot f$,
$n\in\mathbb{N}$, $n\geq n_{0}$.

Let $n\geq n_{0}$. The fixed monomial $x^{\alpha}\boldsymbol{\ell}^{m}$
in $X^{n}\cdot f$ is obtained in the course of iterates $X^{\ell}\cdot f$,
$0\leq\ell\leq n$, from some \emph{initial monomials} $ax^{\beta}\boldsymbol{\ell}^{p}\in\mathcal{S}(f)$
, $a\in\mathbb{R}$, which evolve in $n$ steps of iteration to $x^{\alpha}\boldsymbol{\ell}^{m}$.
By \eqref{str-1}, we see that in each step $\ell$ we differentiate
the respective monomial and then multiply by either a monomial from
$\psi$ or by $\lambda x$. After say $k$ multiplications by monomials
from $\psi$ and the remaining $n-k$ multiplications by $\lambda x$,
each following one differentiation, the initial monomial $x^{\beta}\boldsymbol{\ell}^{p}\in\mathcal{S}(f)$
transforms to: 
\[
x^{\beta+(\alpha_{1}-1)+\cdots+(\alpha_{k}-1)}\boldsymbol{\ell}{}^{p+p_{1}+\cdots+p_{k}+r},
\]
where $x^{\alpha_{i}}\boldsymbol{\ell}^{p_{i}}\in\mathcal{S}(\psi)$,
$i=1,\ldots,k,$ and $r\in\mathbb{N}_{0}$, $0\leq r\leq n$ ($r$
corresponding to the number of differentiations of the logarithm part).
In order to obtain all chains of changes of monomials resulting in
$x^{\alpha}\boldsymbol{\ell}^{m}\in\mathcal{S}(X^{n}\cdot f)$, whose
coefficients then add up to the coefficient of $x^{\alpha}\boldsymbol{\ell}^{m}$
in $X^{n}\cdot f$, we search for all $(\beta,p)\in\mathcal{S}(f)$,
$k,\ r\in\mathbb{N}_{0}$ and $(\alpha_{i},p_{i})\in\mathcal{S}(\psi)$,
$i=1,\ldots,k$, such that: 
\[
x^{\beta+(\alpha_{1}-1)+\cdots+(\alpha_{k}-1)}\boldsymbol{\ell}^{p+p_{1}+\cdots+p_{k}+r}=x^{\alpha}\boldsymbol{\ell}^{m}.
\]
By Neumann's lemma, there are only \emph{finitely many} such choices.

For any such choice (of finitely many), put $L:=(k+1)\cdot\max_{i=1\ldots k}\{|p|,|p_{i}|\}+r$.
Its contribution to the coefficient of $x^{\alpha}\boldsymbol{\ell}^{m}$
in $X^{n}\cdot f$ is absolutely bounded by: 
\[
\big|[H^{n}.f]_{\alpha,m}\big|\leq a\cdot\alpha^{n-r}L^{r}{n \choose k}\lambda^{n-k}\leq C_{\xi,f}\cdot{n \choose k}(\alpha\lambda)^{n},\ n\geq N.
\]
Here, $N$ is the smallest iterate $X^{N}\cdot f$ containing $x^{\alpha}\ell^{m}$
obtained in this chain, and $C_{\xi,f}>0$ is a coefficient depending
on the coefficients of $\xi$ and $f$ and on the chosen chain, but
independent of $n$. The term $\alpha^{n-r}L^{r}$ comes from differentiating
$n$ times the initial monomial $x^{\beta}\boldsymbol{\ell}^{p}$.
The term ${n \choose k}\lambda^{n-k}$ comes from $n-k$ multiplications
by $\lambda x$. For the given chain, we put $C_{\alpha,m}^{n}=C_{\xi,f}\cdot{n \choose k}(\alpha\lambda)^{n},\ n\geq N.$
The series 
\[
\sum_{n=N}^{\infty}\frac{t^{n}}{n!}C_{\alpha,m}^{n}=\sum_{n=N}^{\infty}\frac{t^{n}}{n!}C_{\xi,f}\cdot{n \choose k}(\alpha\lambda)^{n}
\]
converges absolutely. Summing contributions to the coefficient of
$x^{\alpha}\ell^{m}$ of all (finitely many) possible chains, we conclude
the same for the absolute convergence of the whole coefficient of
$x^{\alpha}\ell^{m}$. By Proposition~\ref{prop:small_weak_series},
the operator $\exp(tX):\mathcal{L}\to\mathcal{L}$ is weakly well-defined.

The final statements are proven in the same way as in the proof of
Proposition~\ref{prop: existence-formal-flow-parabolic} from Section~\ref{sub:fivetwo}.
Finally, the finitely generated case follows easily, since the semigroup
$R$ is then finitely generated. $\hfill\Box$\end{proof}
\begin{lem}[Lemma~\ref{lem:H-differential-operator} in the hyperbolic case]
\label{lem:diff_hypH} Let $f\in\mathcal{L}^{0}$ (resp. $\mathcal{L}_{\mathfrak{D}}^{0}$)
be a hyperbolic contraction. Let $F=\mathrm{\text{iso}}(f)\in L(\mathcal{L})$
(resp. $L(\mathcal{L}_{\mathfrak{D}})$) and let $H=F-\mathrm{Id}$.
Then all the iterates $H^{k}$ can be written as weakly well-defined
formal differential operators on $\mathcal{L}$ (resp. $\mathcal{L}_{\mathfrak{D}}$):
\begin{equation}
H^{k}=\sum_{\ell=1}^{\infty}h_{\ell}^{k}\frac{\mathrm{d}^{\ell}}{{\mathrm{d}x}^{\ell}},\ h_{\ell}^{k}\in\mathcal{L}\text{ (resp. \ensuremath{\mathcal{L}_{D}})},\ k\in\mathbb{N};\label{eq:hha}
\end{equation}
\end{lem}
\begin{proof}
Let $f=\lambda\cdot\mathrm{id}+\psi,\ \psi\in\mathcal{L},\ \text{ord}(\psi)\succ(1,0)$
and $0<\lambda<1$. Let $h=f-\mathrm{id}=(\lambda-1)\cdot\mathrm{id}+\psi.$
Note that $\mathrm{ord}(h)=(1,0)$. The proof is by induction. It
is a more elaborate version of the proof of Lemma~\ref{lem:X-differential-operator}
in the parabolic case. The induction basis ($k=1$) follows easily
by Taylor expansion, which we have proven to converge to $H.g\in\mathcal{L}$
in the weak topology: 
\begin{equation}
H\cdot g=g(x+h)-g(x)=\sum_{\ell=1}^{\infty}\frac{h^{\ell}}{\ell!}\frac{\mathrm{d}^{\ell}g}{{\mathrm{d}x}^{\ell}},\ g\in\mathcal{L}.\label{eq:tay}
\end{equation}
Thus, $H=\sum_{\ell=1}^{\infty}h_{\ell}^{0}\frac{\mathrm{d}^{\ell}}{{\mathrm{d}x}^{\ell}},$
with the coefficients $h_{\ell}^{0}:=\frac{h^{\ell}}{\ell!}\in\mathcal{L}$,
$\ell\in\mathbb{N}$. Note that, unlike the parabolic case, all summands
of the series are \emph{of the same order $\mathrm{ord}(g)$.} For
every monomial of Taylor expansion \eqref{eq:tay}, the series of
its coefficients converges \emph{absolutely}. Assume now that operators
$H^{m}$, $m\leq k$, can be written in the form \eqref{eq:hha} of
a differential operator, where the series converges in the weak topology
on $\mathcal{L}$. That is, if a monomial appears in infinitely many
summands, the series of its coefficients is convergent. Suppose additionally
that the series of coefficients of every monomial in the expansion
\eqref{eq:hha} converges \emph{absolutely}.

\emph{The induction step}: we prove that the operator $H^{k+1}$ can
be written in the differential form \eqref{eq:tay}. Moreover, we
prove the \emph{absolute} convergence of series of coefficients of
every monomial of the support of $H^{k+1}\cdot g$ in this formula.
By Taylor expansion, we obtain: 
\begin{align}
H^{k+1}\cdot g\,(x)=H(H^{k}\cdot g)\,(x)=H^{k}\cdot g\,(x+h(x))-H^{k}\cdot g\,(x) & =\sum_{i=1}^{\infty}\frac{h(x)^{i}}{i!}\frac{\mathrm{d}^{i}(H^{k}\cdot g)}{\mathrm{d}x^{i}}\nonumber \\
 & =\sum_{i=1}^{\infty}\frac{h^{i}}{i!}\frac{\mathrm{d}^{i}}{\mathrm{d}x^{i}}\Big(\sum_{\ell=1}^{\infty}h_{\ell}^{k}\frac{\mathrm{d}^{\ell}g}{{\mathrm{d}x}^{\ell}}\Big)=\sum_{i=1}^{\infty}\Big(\sum_{\ell=1}^{\infty}h_{i\ell}^{k}\frac{\mathrm{d}^{\ell}g}{{\mathrm{d}x}^{\ell}}\Big),\label{eq:tii-1h}
\end{align}
with $h_{i\ell}^{k}\in\mathcal{L}$, $i,\,\ell\in\mathbb{N}$. The
elements of the double sum \eqref{eq:tii-1} can be represented by
the grid: 
\begin{equation}
\begin{array}{rcccc}
H^{k}\cdot g\ \vline & h_{1}^{k}g' & h_{2}^{k}g'' & h_{3}^{k}g''' & \ldots\\[0.1cm]
\hline h\frac{\mathrm{d}}{\mathrm{d}x}(H^{k}\cdot g)\vline & h\frac{\mathrm{d}}{\mathrm{d}x}(h_{1}^{k}g') & h\frac{\mathrm{d}}{\mathrm{d}x}(h_{2}^{k}g'') & \ldots\\[0.1cm]
\frac{h^{2}}{2!}\frac{\mathrm{d}^{2}}{{\mathrm{d}x}^{2}}(H^{k}\cdot g)\vline & \frac{h^{2}}{2!}\frac{\mathrm{d}^{2}}{{\mathrm{d}x}^{2}}(h_{1}^{k}g') & \frac{h^{2}}{2!}\frac{\mathrm{d}^{2}}{{\mathrm{d}x}^{2}}(h_{2}^{k}g'') & \ldots\\[0.1cm]
\frac{h^{3}}{3!}\frac{\mathrm{d}^{3}}{{\mathrm{d}x}^{2}}(H^{k}\cdot g)\vline\\
\vdots\qquad & \vdots & \vdots & \vdots
\end{array}\label{eq:grih}
\end{equation}
Unlike the parabolic case, the order of the terms \emph{stays the
same} along the rows and along the columns, so that one monomial from
the support may appear in every term of every row and of every column.
The order of the summation in \eqref{eq:tii-1h} is by rows. The double
sum \eqref{eq:tii-1h} converges in this order of the summation in
the weak topology on $\mathcal{L}$ (by assumption and by convergence
of Taylor expansions). In the sequel, we prove that the coefficient
of a fixed monomial of the support of $H^{k+1}.g$ converges (to the
same limit) if we change the order of summation from the summation
by rows to the summation by columns. Since each derivative of $g$
appears only in finitely many first columns, we have thus proven that
the following sum converges in $\mathcal{L}$ (in the weak topology),
to the same limit $H^{k+1}\cdot g$: 
\begin{equation}
H^{k+1}\cdot g=\sum_{\ell=1}^{\infty}\Big(\sum_{i=1}^{\infty}h_{i\ell}^{k}\frac{\mathrm{d}^{\ell}g}{{\mathrm{d}x}^{\ell}}\Big)=\sum_{\ell=1}^{\infty}h_{\ell}^{k+1}\frac{\mathrm{d}^{\ell}g}{{\mathrm{d}x}^{\ell}}.\label{eq:tii-2h}
\end{equation}
We also have that $h_{\ell}^{k+1}:=\sum_{i=1}^{\infty}h_{i\ell}^{k}\in\mathcal{L}$.

We use the following version of the \emph{Moore-Osgood theorem} (see
for example \cite[Theorem 8.3]{rudin}): given a real double sequence
$a_{m,n}$, if the sum 
\[
\sum_{m\in\mathbb{N}}\sum_{n\in\mathbb{N}}|a_{m,n}|
\]
converges, then the following iterated sums exist and commute: 
\[
\sum_{m\in\mathbb{N}}\sum_{n\in\mathbb{N}}a_{m,n}=\sum_{n\in\mathbb{N}}\sum_{m\in\mathbb{N}}a_{m,n}.
\]
Therefore, in order to prove the step of the induction, we need to
prove that the double sum of absolute values of coefficients of every
monomial of \eqref{eq:tii-1h} converges in this order of the summation.

By the induction assumption, the convergence of series of coefficients
of every monomial along the first row in \eqref{eq:grih} is absolute.
Fix a monomial $x^{\alpha}\boldsymbol{\ell}^{m}$ from the support
of \eqref{eq:tii-1h}. We prove here that the convergence of its respective
coefficient along every other row is absolute, and that these limits
converge by columns. By the \emph{Moore-Osgood theorem} stated above,
this will prove the step of the induction.

Note that $\mathcal{S}(h_{j}^{k}g^{(j)}),\ \mathcal{S}\big(\frac{h^{\ell}}{\ell!}\frac{\mathrm{d}^{\ell}}{{\mathrm{d}x}^{\ell}}(h_{j}^{k}g^{(j)})\big)\subseteq\mathcal{S}(g)+R$,
for every $j,\ \ell\in\mathbb{N}$. Here, $R$ is a sub-semigroup
generated by $(0,1)$, $(\beta-1,p)$ for $(\beta,p)\in\mathcal{S}(\psi)$,
and containing $(0,0)$. We denote, for a monomial $x^{\alpha}\boldsymbol{\ell}^{m}\in\mathcal{S}(g)+R$,
\[
c_{j}^{0}(\alpha,m):=[h_{j}^{k}g^{(j)}]_{\alpha,m},\ \ c_{j}^{\ell}(\alpha,m):=\Big[\frac{h^{\ell}}{\ell!}\frac{\mathrm{d}^{\ell}}{{\mathrm{d}x}^{\ell}}(h_{j}^{k}g^{(j)})\Big]_{\alpha,m}\qquad,\ j,\,\ell\in\mathbb{N}.
\]
We prove that 
\[
\sum_{\ell\in\mathbb{N}}\sum_{j\in\mathbb{N}}|c_{j}^{\ell}(\alpha,m)|<\infty,\ (\alpha,m)\in\mathcal{S}(g)+R.
\]
In order to bound $|c_{j}^{\ell}(\alpha,m)|$, note that the monomial
$x^{\alpha}\boldsymbol{\ell}^{m}\in\mathcal{S}\big(\frac{h^{\ell}}{\ell!}\frac{\mathrm{d}^{\ell}}{{\mathrm{d}x}^{\ell}}(h_{j}^{k}g^{(j)})\big),\ \ell\in\mathbb{N},$
is obtained from some \emph{initial }monomial $b_{j}^{0}x^{\beta}\boldsymbol{\ell}^{n}\in\mathcal{S}(h_{j}^{k}g^{(j)}),\ b_{j}^{0}=[h_{j}^{k}g^{(j)}]_{\beta,n}$,
undergoing $\ell$ differentiations and the multiplication by 
\[
\frac{h^{\ell}}{\ell!}=\frac{1}{\ell!}\big((\lambda-1)x+\psi(x)\big)^{\ell}=\frac{(\lambda-1)^{\ell}}{\ell!}x^{\ell}\Big(1+x^{-1}\frac{\psi(x)}{\lambda-1}\Big)^{\ell}.
\]
We obtain: 
\begin{align}
b_{j}^{0}\cdot\beta(\beta-1) & \cdots(\beta-(\ell-r)+1)\cdot n(n+1)\cdots(n+r-1)x^{\beta-\ell}\boldsymbol{\ell}^{n+r}\cdot\nonumber \\
 & \cdot\frac{(\lambda-1)^{\ell}}{\ell!}x^{\ell}{\ell \choose s}(\lambda-1)^{-s}b_{1}x^{\beta_{1}-1}\boldsymbol{\ell}^{p_{1}}\cdots b_{s}x^{\beta_{s}-1}\boldsymbol{\ell}^{p_{s}}\ =\ \star\ x^{\alpha}\boldsymbol{\ell}^{m}.\label{eq:see}
\end{align}
Here, $0\leq r\leq\ell$ is the number of derivatives applied on the
logarithmic components, $s\in\mathbb{N}_{0}$, $0\leq s\leq\ell$,
$b_{i}x^{\beta_{i}}\boldsymbol{\ell}^{p_{i}}\in\mathcal{S}(\psi)$,
$i=1\ldots s$. By Neumann's lemma, there are only finitely many choices
for $r,\ s\in\mathbb{N}_{0}$, $(\beta,n)\in\mathcal{S}(g)+R,\ (\beta_{i},p_{i})\in\mathcal{S}(\psi),\ i=1,\ldots,s$.
The choices are independent of $\ell,\, k,\, j$. We analyse here
only one of the combinations (afterwards, we sum up finitely many
bounds to obtain a bound on the whole coefficient $c_{j}^{\ell}(\alpha,m)$):
\[
|c_{j}^{\ell}(\alpha,m)|\leq C\cdot|c_{j}^{0}(\beta,n)|\cdot\Big|{\beta \choose \ell-r}\Big|\cdot{\ell \choose s}\cdot\frac{(1-\lambda)^{\ell-s}}{\ell!},
\]
where $C\geq0$ depends only on the given combination (independent
of $j$, $\ell$ or $k$). Therefore, 
\[
\sum_{\ell\geq r,s}\sum_{j\in\mathbb{N}}|c_{j}^{\ell}(\alpha,m)|\leq C\cdot\sum_{\ell\geq r,s}\Big|{\beta \choose \ell-r}\Big|{\ell \choose s}\cdot\frac{(1-\lambda)^{\ell-s}}{\ell!}\sum_{j\in\mathbb{N}}|c_{j}^{0}(\beta,n)|.
\]
By the induction assumption, $C(\beta,n):=\sum_{j\in\mathbb{N}}|c_{j}^{0}(\beta,n)|<\infty$.
We have now: 
\begin{align}
\sum_{\ell\geq r,s}\sum_{j\in\mathbb{N}}|c_{j}^{\ell}(\alpha,m)| & \leq C\cdot C(\beta,n)\sum_{\ell\geq r,s}\Big|{\beta \choose \ell-r}\Big|{\ell \choose s}\frac{(1-\lambda)^{\ell-s}}{\ell!}\nonumber \\
 & \leq C\cdot C(\beta,n)\sum_{\ell\geq r,s}\Big|{\beta \choose \ell-r}\Big|(1-\lambda)^{\ell-s}<\infty.\label{eq:sum}
\end{align}
The last sum converges by the ratio test, since $0<1-\lambda<1$.
We have thus proven that the summation in \eqref{eq:grih} may be
done by columns instead of by rows, while the sum $H^{k+1}.g$ remains
the same. The formula \eqref{eq:hha} for $H^{k+1}.g$, $g\in\mathcal{L}$,
thus converges in the weak topology. Moreover, in this formula, the
series of absolute values of coefficients of every fixed monomial
$x^{\alpha}\boldsymbol{\ell}^{m}$ converges to \eqref{eq:sum} (more
accurately, to a \emph{finite sum of sums of the type \eqref{eq:sum}},
each for every possible combination).

Note additionally that from \eqref{eq:hha} we have that $h_{1}^{k}=H^{k}\cdot\mathrm{id}$,
$h_{2}^{k}=\frac{1}{2}H^{k}\cdot x^{2}-xh_{1}^{k}$, etc., for every
$k\in\mathbb{N}$. The finitely generated case follows directly.\end{proof}
\begin{lem}[Lemma~\ref{lem:X-differential-operator} in the hyperbolic case]
\label{lem:diff_hyp} Let $f\in\mathcal{L}^{0}$ (resp. $\mathcal{L}_{\mathfrak{D}}^{0}$)
be a hyperbolic contraction. Let $F=\mathrm{iso}(f)\in L(\mathcal{L})$
(resp. $L(\mathcal{L}_{\mathfrak{D}})$) and $H=F-\mathrm{Id}$. Let
$X=\log F=\log(\mathrm{Id}+H)$. Then $X$ can be written as a weakly
well-defined formal differential operator on $\mathcal{L}$ (resp.
$\mathcal{L}_{\mathfrak{D}}$): 
\begin{equation}
X=\sum_{\ell=1}^{\infty}h_{\ell}\frac{\mathrm{d}^{\ell}}{{\mathrm{d}x}^{\ell}},\ h_{\ell}\in\mathcal{L}\text{ (resp. \ensuremath{\mathcal{L}_{\mathfrak{D}}})}.\label{termww}
\end{equation}
\end{lem}
\begin{proof}
By Lemma \ref{lem:expo_log_hyp}, the operator $X.g$ is given by
the logarithmic series which converges in the weak topology: 
\begin{align}
 & X\cdot g=\log(\text{Id}+H)\cdot g=H\cdot g-\frac{1}{2}H^{2}\cdot g+\frac{1}{3}H^{3}\cdot g+\cdots\nonumber \\
 & \ =\big(h_{1}^{1}g'+h_{2}^{1}g''+h_{3}^{1}g'''+\cdots\big)-\frac{1}{2}\big(h_{1}^{2}g'+h_{2}^{2}g''+h_{3}^{2}g'''+\cdots\big)+\frac{1}{3}\big(h_{1}^{3}g'+h_{2}^{3}g''+h_{3}^{3}g'''+\cdots\big)+\cdots\label{s}
\end{align}
By Lemma~\ref{lem:diff_hypH}, all operators $H^{k}$, $k\in\mathbb{N}$,
can be written as differential operators \eqref{eq:hha}, with convergence
in the weak topology in $\mathcal{L}$. Thus the \emph{double sum}
\eqref{s} converges to $X.g$ in this order of the summation, in
the same topology.

Let us consider a fixed monomial from the support of \eqref{s}, $x^{\alpha}\ell^{m}\in\mathcal{S}(g)+R$,
see Lemma~\ref{lem:diff_hypH}. The monomial may appear in every
term of the double sum \eqref{s}. By the proof of Lemma~\ref{lem:diff_hypH},
in each bracket of \eqref{s}, the series of coefficients of $x^{\alpha}\boldsymbol{\ell}^{m}$
converges \emph{absolutely}. We denote the absolute limit in the $k$-th
bracket by $A_{k}>0$. By the \emph{Moore-Osgood theorem} stated in
the proof of Lemma~\ref{lem:diff_hypH}, to prove that we are allowed
to change the order of the summation in \eqref{s}, that is, to group
the terms in front of every derivative of $g$, we need to prove the
convergence of the sum: 
\begin{equation}
\sum_{k\in\mathbb{N}}\frac{A_{k}}{k}.\label{co}
\end{equation}
The following argument is similar as in the proof of Proposition~\ref{lem:expo-log_and_log-exp}.
In the proof of Lemma~\ref{lem:diff_hypH}, we have described the
iterative step in which the $(k+1)$-st bracket is deduced from the
$k$-th bracket of \eqref{s} (that is, the differential form for
$H^{k+1}.g$ from the differential form for $H^{k}.g$). All monomials
from \eqref{s} belong to $\mathcal{S}(g)+R$. By Neumann's lemma,
there are only finitely many ways in which a fixed monomial $x^{\alpha}\boldsymbol{\ell}^{m}$
belonging to some bracket of \eqref{s} is obtained from previous
brackets and, initially, from monomials of $\mathcal{S}(g)$, \emph{independently
of the bracket}. We adopt the notion of \emph{chains} to describe
the evolution of monomials, similarly as in the proof of Lemma~\ref{lem:expo-log_and_log-exp}.
We fix one (of finitely many chains): an initial monomial $x^{\beta_{0}}\boldsymbol{\ell}^{n_{0}}\in\mathcal{S}(g)$
evolves in $k$ steps to $x^{\alpha}\boldsymbol{\ell}^{m}$, through
(necessarily distinct) monomials $x^{\beta_{i}}\boldsymbol{\ell}^{p_{i}}$,
$i=1,\ldots,r$, $r\in\mathbb{N}_{0}$, as described in \eqref{eq:see}:
\begin{equation}
x^{\beta_{0}}\boldsymbol{\ell}^{n_{0}}\rightarrow x^{\beta_{1}}\boldsymbol{\ell}^{p_{1}}\rightarrow x^{\beta_{2}}\boldsymbol{\ell}^{p_{2}}\rightarrow\cdots\rightarrow x^{\beta_{r}}\boldsymbol{\ell}^{p_{r}}=x^{\alpha}\boldsymbol{\ell}^{m}.\label{eq:combi}
\end{equation}
To estimate $A_{k}$, $k\geq r$, for this fixed combination, we use
the estimate \eqref{eq:sum} from the proof of Lemma~\ref{lem:diff_hypH},
where the sum of absolute values of coefficients of a monomial in
the $(\ell+1)$-st bracket is estimated by the sum of absolute values
of coefficients of its corresponding (for the given chain) monomial
in the $\ell$-th bracket, $\ell\in\mathbb{N}$. Note that the estimate
\eqref{eq:sum} is \emph{independent of $\ell$}. For the above combination
\eqref{eq:combi}, in $k-r$ of total $k$ steps the monomial remains
the same, and changes in the remaining $r$ steps. Let $C_{\ell}(\beta_{i},p_{i})$
denote the sum of absolute values of coefficients of the monomial
$x^{\beta_{i}}\ell^{p_{i}}$ in the $\ell$-th bracket, $\ell\in\mathbb{N}$.
By \eqref{eq:sum}, in the steps where the respective monomial stays
the same (then, $r=s=0$, $C=1$) , we have the estimate: 
\[
C_{\ell+1}(\beta_{i},p_{i})\leq C_{\ell}(\beta_{i},p_{i})\cdot\sum_{j\geq0}\Big|{\beta_{i} \choose j}\Big|(1-\lambda)^{j},\ \ell\in\mathbb{N}.
\]
Note that it is sufficient to prove that the terms of \eqref{s} starting
from some fixed derivative can be regrouped as in the statement of
the lemma (the terms with first finitely many derivatives form a finite
sum of series, so the order of the summation can be changed trivially).
Therefore, in the above sum, without loss of generality we can take
$j\geq j_{0}$ instead of $j\geq0$, for any $j_{0}\in\mathbb{N}$.
To each chain, we associate a number $0<A<1$ and an integer $j_{0}\in\mathbb{N}_{0}$,
such that the sum $\sum_{j\geq j_{0}}\Big|{\beta_{i} \choose j}\Big|(1-\lambda)^{j}$
above is bounded by $A$, for every $(\beta_{i},p_{i})$ of the given
chain. This follows from the convergence of the series and the fact
that there exist only finitely many $\beta_{i}$-s in the given chain.
Notice that the constant $0<A<1$ depends only on the chain. Therefore,
for the steps in which the monomial remains the same, we have: 
\[
C_{\ell+1}(\beta_{i},p_{i})\leq A\cdot C_{\ell}(\beta_{i},p_{i}),\ \ell\in\mathbb{N},\ 0<A<1.
\]
On the other hand, there are only finitely many ($r$) steps (for
the given chain) in which the corresponding monomial $x^{\beta_{i}}\boldsymbol{\ell}^{p_{i}}$
changes to $x^{\beta_{i+1}}\boldsymbol{\ell}^{p_{i+1}}$, $i=0\ldots r-1$.
By \eqref{eq:sum}, we have a simple estimate: 
\[
C_{\ell+1}(\beta_{i+1},p_{i+1})\leq DC_{\ell}(\beta_{i},p_{i}),\ i=0,\ldots,r-1,\ \ell\in\mathbb{N},
\]
where $D>0$ depends only on the chain. We obtain the estimate: 
\[
A_{k}\leq a\cdot D^{r}A^{k-r}\leq CA^{k}.
\]
Here, $a\in\mathbb{R}$ is the coefficient of $x^{\beta_{0}}\boldsymbol{\ell}^{p_{0}}$
in $g$, and $C>0$ and $0<A<1$ depend only on the chain. The series
\eqref{co} thus converges. Since there are only finitely many chains
contributing to the coefficient of $x^{\alpha}\boldsymbol{\ell}^{m}$
in \eqref{s}, the result follows.
\end{proof}

\begin{proof}[Proof of Theorem B in the hyperbolic case]
\emph{} Let $f(x)=\lambda x+\mathrm{h.o.t.}\in\mathcal{L}$, $0<\lambda<1$.
Let $F=\text{iso}(f)\in L(\mathcal{L})$, see Remark~\ref{ex:automorphism_associated}.
By Lemma~\ref{lem:expo_log_hyp}, the operator $\log F\in L(\mathcal{L})$
is weakly well-defined. Using Proposition~\ref{prop:derivative}
and Lemma~\ref{lem:diff_hyp}, we prove (as in the parabolic case
in Section~\ref{sub:fivethree}) that the operator $X=\log F$ is
a vector field: $X=\xi\frac{\mathrm{d}}{\mathrm{d}x}$, $\xi\in\mathcal{L}$
with $\text{ord}(\xi)=(1,0)$. It follows from Lemma~\ref{lem:expo-log_and_log-exp},
that: 
\[
\exp(X)\cdot\mathrm{id}=\exp(\log F)\cdot\mathrm{id}=F\cdot\mathrm{id}=f,
\]
so $f$ is the time-one map of $X$.

We now prove the \emph{uniqueness} of $X$. Let $X=\xi\frac{\mathrm{d}}{\mathrm{d}x}$
be \emph{any} vector field such that $f$ (a hyperbolic contraction)
is its time-one map. Since $f$ is hyperbolic, by Proposition~\ref{prop:unic}
it follows that $\text{ord}(\xi)=(1,0)$. By Proposition~\ref{prop: existence-formal-flow-hyperbolic},
the family $\exp(tX)$ defines a $\mathcal{C}^{1}$-flow of $X$.
By Proposition~\ref{prop:unic}, the $\mathcal{C}^{1}$-flow of $X$
is unique, so it follows that: 
\[
f=\exp\big(\xi\frac{\mathrm{d}}{\mathrm{d}x}\big)\cdot\mathrm{id}=\xi+\frac{1}{2!}\xi'\xi+\frac{1}{3!}(\xi'\xi)'\xi+\cdots
\]
Using the above expansion, we additionally conclude that $\xi(x)=\lambda x+\mathrm{h.o.t.}$
if and only if $f(x)=e^{\lambda}\, x+\mathrm{h.o.t.}$ Since $f$
is a hyperbolic contraction, it follows that $\xi(x)=\lambda x+\mathrm{h.o.t.}$
with $\lambda<0$. By Corollary~\ref{cor}, we have: 
\[
\exp(X)\cdot h=h\circ f=F\cdot h,\ h\in\mathcal{L}.
\]
That is, $\exp X=F$. By Lemma~\ref{lem:uniq}, $X=\log F$ and uniqueness
follows. 
\end{proof}
Finally, we prove Lemma~\ref{lem:expo-log_and_log-exp} and Lemma~\ref{lem:uniq}
from Sections~\ref{sub:fiveone} and \ref{sub:fivetwo} for the hyperbolic
case. 
\begin{proof}[Proof of Lemma~\ref{lem:expo-log_and_log-exp} in the hyperbolic
case]
\emph{}By Lemma~\ref{lem:expo_log_hyp}, the operator $\log F\in L(\mathcal{L})$
is weakly well-defined. As in the proof of Theorem B above, using
Proposition~\ref{prop:derivative} and Lemma~\ref{lem:diff_hyp},
the operator $\log F$ is a vector field. Thus, $\log F=\xi\frac{\mathrm{d}}{\mathrm{d}x}$,
$\xi\in\mathcal{L},$ with $\text{ord}(\xi)=(1,0)$. By Proposition~\ref{prop: existence-formal-flow-hyperbolic},
$\exp(\log F)\in L(\mathcal{L})$ is weakly well-defined. Having proven
that all operators are weakly well-defined, the equality follows by
symbolic computation with formal exp-log series. 
\end{proof}

\begin{proof}[Proof of Lemma~\ref{lem:uniq} in the case $\text{ord}(\xi)=(1,0)$]
\emph{}By Proposition~\ref{prop: existence-formal-flow-hyperbolic},
$\exp(X)\in L(\mathcal{L})$ is a weakly well-defined operator and
an isomorphism associated with $f=\exp(X)\cdot\mathrm{id}$. $f\in\mathcal{L}$
is hyperbolic since $\text{ord}(\xi)=(1,0)$. More precisely, we compute:
\begin{align}
f(x)=\exp(X)\cdot\mathrm{id} & =x+\xi+\frac{1}{2!}\xi'\xi+\frac{1}{3!}(\xi'\xi)'\xi+\cdots\nonumber \\
 & =x+(\lambda x+\mathrm{h.o.t.})+\frac{1}{2!}(\lambda^{2}x+\mathrm{h.o.t.})+\frac{1}{3!}(\lambda^{3}x+\mathrm{h.o.t.})+\cdots=e^{\lambda}x+\mathrm{h.o.t.}
\end{align}
Since $\lambda<0$, we have $0<e^{\lambda}<1$, so $f$ is a hyperbolic
contraction. By Lemma~\ref{lem:expo_log_hyp}, the operator $\log F$
is a weakly well-defined operator. The equality follows by symbolic
computation with formal exp-log series. 
\end{proof}
We illustrate the convergence of coefficients (that is, the convergence
in the weak topology in $\mathcal{L}$ of respective series) in the
hyperbolic case on the simplest hyperbolic elements of $\mathcal{L}$:\\

\begin{example}
~
\begin{enumerate}[1., font=\textup, nolistsep, leftmargin=0.6cm]
\item Let $f(x)=\lambda x$, $0<\lambda<1$. By Theorem B (hyperbolic case),
the embedding vector field for $f$ is given by $X=\xi\frac{\mathrm{d}}{\mathrm{d}x}$,
where 
\[
\xi(x)=\log(F)\cdot\mathrm{id}=(\lambda-1)x-\frac{1}{2}(\lambda-1)^{2}x+\frac{1}{3}(\lambda-1)^{3}x+\cdots=\log\lambda\cdot x\in\mathcal{L}.
\]

\item Let $X=ax\frac{\mathrm{d}}{\mathrm{d}x}$, $a\in\mathbb{R}$. By Proposition~\ref{prop: existence-formal-flow-hyperbolic},
the field $X$ admits a flow $\{f_{t}:\, t\in\mathbb{R}\}\subset\mathcal{L}$:
\[
f_{t}(x)=\exp(tX)\cdot\mathrm{id}=x+t\xi+\frac{t^{2}}{2!}\xi'\xi+\cdots=x+tax+\frac{t^{2}a^{2}}{2!}x+\frac{t^{3}a^{3}}{3!}x+\cdots=e^{ta}\cdot x,\quad t\in\mathbb{R}.
\]

\end{enumerate}
The above series converge in $\mathcal{L}$ in the weak topology.
Note that they \emph{do not} converge in $\mathcal{L}$ neither in
the formal topology nor in the product topology with respect to the
discrete topology, since the monomial $x$ appears in every term.
\end{example}

\subsection{Theorem B in the strongly hyperbolic case\label{sub:fivefive}}

We first observe that a strongly hyperbolic element of $\mathcal{L}^{H}$
does not embed in the $\mathcal{C}^{1}$-flow of a vector field. Indeed,
let $X=\xi\frac{\mathrm{d}}{\mathrm{d}x}$. If $\left(1,0\right)\preceq\mathrm{ord}\left(\xi\right)$,
it follows from Propositions \ref{prop: existence-formal-flow-parabolic}
and \ref{prop: existence-formal-flow-hyperbolic} that all the elements
of the $\mathcal{C}^{1}$-flow of $X$ are either parabolic or hyperbolic.
If $\mathrm{ord}\left(\xi\right)\prec\left(1,0\right)$, it follows
from Proposition \ref{prop:unic} that $X$ does not admit any $\mathcal{C}^{1}$-flow.
We have moreover the following \emph{negative version} of Propositions
\ref{prop: existence-formal-flow-parabolic} and \ref{prop: existence-formal-flow-hyperbolic}:
\begin{prop}
\label{prop:not-weakly-well-defined}Let $X=\xi\frac{\mathrm{d}}{\mathrm{d}x}$,
$\xi\in\mathcal{L}$, such that $\mathrm{ord}\left(\xi\right)\prec\left(1,0\right)$.
Then the exponential operator $\exp\left(tX\right)$, $t\in\mathbb{R}$,
is not weakly well-defined.\end{prop}
\begin{proof}
Consider the expansion
\begin{equation}
\exp\left(tX\right)\cdot\mathrm{id}=\mathrm{id}+t\xi+\frac{t^{2}}{2!}\xi'\xi+\frac{t^{3}}{3!}\left(\xi'\xi\right)'\xi+\cdots\label{eq:expansion-exp-strongly}
\end{equation}
We observe that the orders of the terms in this expansion are unboundedly
increasing instead of decreasing. Hence this exponential series does
not converge in $\mathcal{L}$ in any of the topologies considered
in this work (see Subsection \ref{sub:transfinite-sequences-elements-L}).
\end{proof}
However, the results of Section \ref{sec:proof-theorem-A} lead to
the following embedding statement, which is a \emph{weak} version
of Theorem B for strongly hyperbolic elements:
\begin{thm*}[Weaker version of Theorem B, the strongly hyperbolic case]
\label{prop:weak} Let $f\in\mathcal{L}^{H}$ (resp. $f\in\mathcal{L}_{\mathfrak{D}}^{H}$)
be strongly hyperbolic. Then $f$ embeds in a flow $\left(f^{t}\right)_{t\in\mathbb{R}}$
of elements of $\mathcal{L}^{H}$ (resp. $\mathcal{L}_{\mathfrak{D}}^{H}$).\end{thm*}
\begin{proof}
Write $f\left(x\right)=\lambda x^{\alpha}+\mathrm{h.o.t}$, $\lambda\neq0$,
$\alpha\neq1$. According to Theorem A $(c)$, there exists a change
of variables $\varphi\in\mathcal{L}^{0}$ such that $f_{0}\left(x\right)=\varphi^{-1}\circ f\circ\varphi\left(x\right)=x^{\alpha}$.
Obviously, $f_{0}(x)=x^{\alpha}$ embeds in the $f_{0}^{t}\left(x\right)=x^{\left(\alpha^{t}\right)}$,
$t\in\mathbb{R}$. Hence, $f$ embeds in the flow: 
\begin{equation}
f^{t}\left(x\right)=\left(\varphi\circ f_{0}^{t}\circ\varphi^{-1}\right)\left(x\right),\ t\in\mathbb{R}.\label{eq:opg}
\end{equation}
The claim in the finitely generated case follows easily.
\end{proof}
We notice here an important difference between the parabolic or hyperbolic
case and the strongly hyperbolic case. If $f\in\mathcal{L}^{H}$ is
parabolic or hyperbolic, there exists a well-ordered subset $S\subseteq\mathbb{R}_{>0}\times\mathbb{Z}$
which contains the supports of all the elements of the $\mathcal{C}^{1}$-flow
in which $f$ embeds. It is not the case anymore if $f$ is strongly
hyperbolic. Moreover, in this case, the monomials of the $\left(f^{t}\right)_{t}$,
and not only their coefficients, depend on $t\in\mathbb{R}$. The
following example illustrates these facts, as well as other specific
features of the strongly hyperbolic situation.
\begin{example}[A counterexample to the exponential formula for the flow in the strongly
hyperbolic case]
Consider the flow $f_{0}^{t}\left(x\right)=x^{\left(\alpha^{t}\right)}$,
$t\in\mathbb{R}$. The strongly hyperbolic element $f_{0}^{1}\left(t\right)=x^{\alpha}$
embeds in this flow, and all the elements $f_{0}^{t}$ are strongly
hyperbolic. Since $\mathcal{S}\left(f_{0}^{t}\right)=\left\{ \left(\alpha^{t},0\right)\right\} $,
these supports are not contained in a common well-ordered subset of
$\mathbb{R}_{>0}\times\mathbb{Z}$. Hence the family $\left(f_{0}^{t}\right)$,
$t\in\mathbb{R}$, is not a \emph{$\mathcal{C}^{1}$-flow} in the
sense of Definition \ref{def:time-one-map}.\\

Let us now consider 
\[
\xi\left(x\right):=\frac{\mathrm{d}f_{0}^{t}\left(x\right)}{\mathrm{d}t}\Big|_{t=0}=-\log\alpha\cdot x\boldsymbol{\ell}^{-1}\in\mathcal{L}.
\]
It would seem that $\left(f_{0}^{t}\right)$ is a flow of the vector
field $X=\xi\frac{\mathrm{d}}{\mathrm{d}x}$. But we have just noticed
that $\left(f_{0}^{t}\right)$ is not a $\mathcal{C}^{1}$-flow. Moreover,
since $\mathrm{ord}\left(\xi\right)=\left(1,-1\right)\prec\left(1,0\right)$,
we have seen in Proposition \ref{prop:unic} that $X$ does not admit
any $\mathcal{C}^{1}$-flow.

It is nevertheless interesting to observe the result of the exponential
formula \eqref{eq:exponential-operator} for the the field $X=\xi\frac{\mathrm{d}}{\mathrm{d}x}$.
We obtain: 
\begin{align*}
\exp\big(-\log\alpha\cdot x\boldsymbol{\ell}^{-1}\cdot\frac{\mathrm{d}}{\mathrm{d}x}\big)\cdot\mathrm{id}=x- & \log\alpha\cdot x\boldsymbol{\ell}^{-1}+\\
+ & \frac{\log^{2}\alpha}{2!}(x\boldsymbol{\ell}^{-2}-x\boldsymbol{\ell}^{-1})+\\
+ & \frac{\log^{3}\alpha}{3!}(-x\boldsymbol{\ell}^{-3}x+3x\boldsymbol{\ell}^{-2}-x\boldsymbol{\ell}^{-1})+\\
+ & \frac{\log^{4}\alpha}{4!}(x\boldsymbol{\ell}^{-4}-6x\boldsymbol{\ell}^{-3}+7x\boldsymbol{\ell}^{-2}-x\boldsymbol{\ell}^{-1})+\cdots
\end{align*}
The order of terms obviously increases, and the above sum \emph{does
not converge in $\mathcal{L}$ in any of the mentioned topologies}.
Formally, it is not even an element of $\mathcal{L}$. However, if
we regroup and sum the terms along the diagonals going from bottom
to top, using the convergence of the exponential series, we obtain:
\begin{align*}
\exp\big(-\log & \alpha\cdot x\boldsymbol{\ell}^{-1}\frac{\mathrm{d}}{\mathrm{d}x}\big)\cdot\mathrm{id}=\\
= & \, x-x\boldsymbol{\ell}^{-1}\cdot(e^{\log\alpha}-1)+x\boldsymbol{\ell}^{-2}\cdot\frac{1}{2!}\cdot(e^{\log\alpha}-1)^{2}-x\boldsymbol{\ell}^{-3}\cdot\frac{1}{3!}\cdot(e^{\log\alpha}-1)^{3}+\cdots=\\
= & \, x-x\boldsymbol{\ell}^{-1}\cdot(\alpha-1)+\frac{1}{2!}x\boldsymbol{\ell}^{-2}\cdot(\alpha-1)^{2}-\frac{1}{3!}x\boldsymbol{\ell}^{-3}\cdot(\alpha-1)^{3}+\cdots=x\cdot e^{(\alpha-1)\log x}=x^{\alpha}\in\mathcal{L}.
\end{align*}
Hence in some sense $f_{0}$ embeds in the flow of $X$, but not as
it is defined in the present work. We intend to give a precise meaning
to the above computations in a subsequent work.
\end{example}

\section{Examples\label{examp}}
\begin{example}
\label{five} 
\[
f(x)=x+x\boldsymbol{\ell}+\mathrm{h.o.t.}
\]
By Theorem~A, we obtain the formal normal form $f_{0}$ and its embedding
vector field $X_{0}$: 
\begin{align*}
 & f_{0}(x)=x+x\boldsymbol{\ell}+bx\boldsymbol{\ell}^{3},\\
 & \widehat{f}_{0}=\exp(X_{0}).\text{id},\ \ X_{0}=\frac{x}{\boldsymbol{\ell}^{-1}+1/2+(1/2+b)\boldsymbol{\ell}}\frac{\mathrm{d}}{\mathrm{d}x}.
\end{align*}
Here, $b\in\mathbb{R}$ depends on the terms of $f$ up to $x\boldsymbol{\ell}^{3}$. 
\end{example}
In the next example we explain on a very simple example of a Dulac
germ why we need a \emph{transfinite sequence} of power-logarithmic
changes of variables to derive the finite formal normal form from
Theorem~A. That is, we illustrate why a standard sequence of changes
of variables is not sufficient for elimination.

\begin{example}[Dulac germ]
\label{insuffic} Take $f(x)=x+x^{2}\boldsymbol{\ell}^{-1}+x^{2}$.
This germ is of \emph{Dulac type} - it has the expansion $f(x)=x+x^{2}P_{1}(-\log x)$,
where $P_{1}(x)=x+1$. By Theorem~A, the finite formal normal form
of $f$ in $\mathcal{L}$ is: 
\[
f_{0}(x)=x+x^{2}\boldsymbol{\ell}^{-1}+bx^{3}\boldsymbol{\ell}^{-1},\ b\in\mathbb{R}.
\]
Let us illustrate on this example the process used in the proof of
Theorem A. We first eliminate the term $x^{2}$ from $f$. Computing
the first finitely many (important) terms of $f\circ\varphi-\varphi\circ f$,
for a change of variables $\varphi(x)=x+cx^{\beta}\boldsymbol{\ell}^{\ell}$,
$(\beta,\ell)\succ(1,0)$, $c\in\mathbb{R}$, we obtain: 
\[
f\circ\varphi-\varphi\circ f=c(\beta-2)x^{\beta+1}\boldsymbol{\ell}^{\ell-1}+c(\ell+1)x^{\beta+1}\boldsymbol{\ell}^{\ell}-c^{2}x^{\beta+1}\boldsymbol{\ell}^{2\ell-1}+c(2-\beta)(-1)^{m}x^{\beta+1}\boldsymbol{\ell}^{\ell-m}+\mathrm{h.o.t}.
\]
We conclude: by a change of variables $\varphi(x)=x+cx\boldsymbol{\ell}^{-m+1}$,
$m\leq0$, for an appropriate $c\in\mathbb{R}$, we eliminate the
term $x^{2}\boldsymbol{\ell}^{-m}$, but at the same time we generate
the \emph{next one}: $x^{2}\boldsymbol{\ell}^{-m+1}$. Thus we need
a transfinite sequence of changes of variables: 
\begin{align*}
f(x)=x+x^{2}\boldsymbol{\ell}^{-1}+x^{2} & \stackrel{\varphi_{1}(x)=x+c_{1}x\boldsymbol{\ell}}{\longrightarrow}f_{1}(x)=x+x^{2}\boldsymbol{\ell}^{-1}+a_{1}x^{2}\boldsymbol{\ell}+\mathrm{h.o.t}.\\
 & \stackrel{\varphi_{2}(x)=x+c_{2}x\boldsymbol{\ell}^{2}}{\longrightarrow}f_{2}(x)=x+x^{2}\boldsymbol{\ell}^{-1}+a_{2}x^{2}\boldsymbol{\ell}^{2}+\mathrm{h.o.t}.\longrightarrow\cdots
\end{align*}

\end{example}

\begin{example}[Formal normal forms in $\mathcal{L}$ of formal power series]
 \label{fps} Let $f\in\mathbb{R}[[x]]$ be a parabolic formal diffeomorphism,
\[
f(x)=x+x^{k+1}+o(x^{k+1}),\ k\in\mathbb{N}
\]
The standard formal normal form in $\mathbb{R}[[x]]$ is equal to:
\[
f_{s}(x)=x+x^{k+1}+bx^{2k+1},\ b\in\mathbb{R}.
\]

On the other hand, Theorem A gives a normal form $f_{0}$ of $f$
in the wider class $\mathcal{L}^{0}$ of changes of variables. Note
that $\mathbb{R}[[x]]\subset\mathcal{L}$. We prove here that $f_{0}$
is equal to 
\begin{equation}
f_{0}(x)=x+x^{k+1}.\label{fonf}
\end{equation}
Note that, allowing wider class of logarithmic changes of variables,
we remove also the residual term $x^{2k+1}$ from $f_{s}$.

Let us now prove \eqref{fonf}. By Theorem~A, we have: 
\[
f_{0}(x)=x+x^{k+1}+bx^{2k+1}\boldsymbol{\ell},\ b\in\mathbb{R}.
\]
By the algorithm in the proof of Theorem~A applied to parabolic power
series $f$, we show that its residual coefficient $b$ is actually
equal to $0$. Indeed, in order to eliminate all the terms before
the residual one, we use in the algorithm non-logarithmic changes
of variables $\varphi_{m,0}(x)=x+cx^{m}$, $c\in\mathbb{R}$, $m\in\{2,3,\ldots,k\}$.
By these changes of variables, no logarithmic terms are generated
in $f$. Therefore, $f$ is transformed into: 
\begin{equation}
f(x)=x+x^{k+1}+bx^{2k+1}+dx^{2k+2}+\text{h.o.t.},\ b,\, d\in\mathbb{R}.\label{r1}
\end{equation}
In the next step, we remove the residual term $x^{2k+1}$ from $f$.
By \eqref{taa}, we use a logarithmic change of variables $\varphi_{k+1,-1}(x)=x+cx^{k+1}\boldsymbol{\ell}^{-1},\ c\in\mathbb{R}$.
We compute: 
\begin{equation}
\varphi_{k+1}^{-1}\circ f\circ\varphi_{k+1}=f+cx^{2k+1}+r(x).\label{r2}
\end{equation}
Here, $r\in\mathcal{L}$ may contain logarithmic terms, but its leading
monomial is of order at least $(2k+2)$ in $x$. Choosing $c=-b$,
the term $x^{2k+1}$ is eliminated from $F$. Therefore, by \eqref{r1}
and \eqref{r2}, 
\[
\varphi_{k+1}^{-1}\circ f\circ\varphi_{k+1}=x+x^{k+1}+(dx^{2k+2}+\text{h.o.t.})+r(x).
\]
The terms after $x^{k+1}$ are of strictly higher order than the residual
order $(2k+1,1)$, so they are eliminated by changes of variables
from $\mathcal{L}^{0}$.
\end{example}

\section{\label{sec:Appendix}Appendix}

In the proof of Theorem A in Subsection~\ref{sub:proof-precise-form-A},
\emph{Part 2}, in order to prove the existence of a formal normalizing
change of variables $\varphi\in\mathcal{L}^{0}$ for $f\in\mathcal{L}^{H}$
as the composition of a \emph{transfinite sequence} of elementary
changes of variables, we index the set of all elementary changes of
variables used in the normalization by their orders. We describe here
explicitely the set of orders of all elementary changes needed for
reduction of $f$ to a formal normal form $f_{0}$. In addition, this
description allows us to prove easily that, if $f\in\mathcal{L}_{\mathfrak{D}}^{H}$
(that is, if $f$ is of finite type), then $\varphi$ is also of finite
type.

In Subsection~\ref{sub:one}, we give the main lemma of the appendix.
It explains how the support of a transseries behaves under the action
of an elementary change of variables. In Subsection~\ref{sub:two},
we use this lemma to control the orders of the normalizing elementary
changes of variables. Finally, we discuss the finite type cases in
Subsection~\ref{sub:three}.

We analyze only the case when $f$ is \emph{parabolic}. The analysis
for other two cases can be done similarly and we omit it.

\subsection{The action of an elementary change of variables on the support}

\label{sub:one}

Consider $f\in\mathcal{L}^{H}$ parabolic, $f\left(x\right)=x+ax^{\alpha}\boldsymbol{\ell}^{p}+\mathrm{h.o.t},\ (\alpha,p)\succ(1,0).$
Let $S=\mathcal{S}\left(f-\mathrm{id}\right)$, $\left(\alpha,p\right)=\min\left(S\right)$,
and $\overline{S}=S\setminus\left\{ \left(\alpha,p\right)\right\} $.
Recall that we denote by $\left\langle A\right\rangle $ the additive
semigroup generated by a subset $A$ of $\mathbb{R}_{\geq0}\times\mathbb{Z}$.
We introduce the set 
\[
\mathcal{R}=\left\langle \overline{S}-\left(\alpha,p+1\right)\right\rangle +\mathbb{N}_{*}\left(\alpha-1,p\right)+\left\{ 1\right\} \times\mathbb{N}_{*},
\]
where $\mathbb{N}_{*}$ means $\mathbb{N}\setminus\left\{ 0\right\} $.
It follows from Neumann's Lemma that $\mathcal{R}$ is well-ordered.
Moreover, it is easily seen that all elements of $\overline{S}-\left(\alpha,p+1\right)$
are bigger than or equal to $\left(0,0\right)$.

Notice that $S\subseteq\mathcal{R}$ (this remark will allow us to
initiate a transfinite induction in the next subsection). Indeed,
if $\left(\alpha_{1},p_{1}\right)\in S$, we write: 
\begin{align*}
\left(\alpha_{1},p_{1}\right) & =\left(\alpha_{1}-\alpha,p_{1}-p-1\right)+\left(\alpha,p+1\right)\\
 & =\left(\alpha_{1}-\alpha,p_{1}-p-1\right)+\left(\alpha-1,p\right)+\left(1,1\right)\in\mathcal{R}.
\end{align*}
We now prove the main lemma of the appendix. 

\begin{lem} \label{lem:appendix-main-lemma}Consider a parabolic
series $f\left(x\right)=x+ax^{\alpha}\boldsymbol{\ell}^{p}+a_{1}x^{\gamma_{1}}\boldsymbol{\ell}^{r_{1}}+\cdots\in\mathcal{L}^{H}$
such that all exponents $\left(\gamma_{i},r_{i}\right)$ belong to
$\mathcal{R}$. Let $\varphi\left(x\right)=x+cx^{\beta}\boldsymbol{\ell}^{m}$,
$(\beta,m)\succ(1,0)$, be such that 
\begin{equation}
\left(\beta,m\right)=\left(\gamma_{1}-\alpha+1,r_{1}-p\right)\text{ or }\left(\gamma_{1}-\alpha+1,r_{1}-p-1\right).\label{eq:rel}
\end{equation}
Then $\mathcal{S}\left(\varphi^{-1}\circ f\circ\varphi-\mathrm{id}-ax^{\alpha}\boldsymbol{\ell}^{p}\right)$
is contained in $\mathcal{R}$.\end{lem} 

\begin{rem}\label{rem:hom} The change of variables $\varphi(x)=x+cx^{\beta}\boldsymbol{\ell}^{m}$
with $(\beta,m)$ as in \eqref{eq:rel} eliminates the term $a_{1}x^{\gamma_{1}}\boldsymbol{\ell}^{r_{1}}$
from $f$. The exponent $\left(\beta,m\right)$ is given by the homological
equation \eqref{solv}, as the proof of Theorem A. Notice that we
denote the elementary change of variables here by $\varphi$ instead
of $\varphi_{\beta,m}$ for easier reading of the forthcoming computations.\end{rem}
\begin{proof}[Proof of Lemma \ref{lem:appendix-main-lemma}]
 We proceed in several steps: we have to control the supports $\mathcal{S}\left(\varphi\right)$,
$\mathcal{S}\left(\varphi^{-1}\right)$, then $\mathcal{S}\left(\varphi^{-1}\circ f\right)$,
and finally of $\mathcal{S}\left(\varphi^{-1}\circ f\circ\varphi\right)$,
for $f$ and $\varphi$ as in Lemma~\ref{lem:appendix-main-lemma}.

\noindent 1. \textbf{Control of the support $\mathcal{S}\left(\varphi\right)$.}
Assume $\left(\beta,m\right)=\left(\gamma_{1}-\alpha+1,r_{1}-p\right)\text{ or }\left(\gamma_{1}-\alpha+1,r_{1}-p-1\right).$
We claim that 
\[
\left(\beta,m\right)\in\mathcal{R}_{1}:=\left\langle \overline{S}-\left(\alpha,p+1\right)\right\rangle +\mathbb{N}_{0}\left(\alpha-1,p\right)+\left\{ 0\right\} \times\mathbb{N}_{0}.
\]
Notice that $\mathcal{R}_{1}$ is well-ordered.\\

We observe that, if $\left(\beta,m\right)=\left(\gamma_{1},r_{1}\right)-\left(\alpha,p+1\right)+\left(1,0\right)$,
then, for all $k\in\mathbb{N}_{0}$, $\left(\beta,m\right)+\left(-1,k\right)\in\mathcal{R}_{1}$.
In the same way, if $\left(\beta,m\right)=\left(\gamma_{1},r_{1}\right)-\left(\alpha,p\right)+\left(1,0\right)$,
then 
\[
\left(\beta,m\right)+\left(-1,k\right)=\left(\gamma_{1},r_{1}\right)-\left(\alpha,p+1\right)+\left(0,k+1\right)\in\mathcal{R}_{1}.
\]
In particular, we have that $(\beta,m)\in\mathcal{R}_{1}$ and $(\beta-1,m)\in\mathcal{R}_{1}$.

From now on, we suppose that $\left(\beta,m\right)=\left(\gamma_{1},r_{1}\right)-\left(\alpha,p+1\right)+\left(1,0\right)$.
\\

2. \textbf{Control of the support $\mathcal{S}\left(\varphi^{-1}\right)$.}
For this purpose, we consider the isomorphism $\Phi$ associated to
the change of variables $\varphi$. It holds that $\Phi.\mathrm{id}=\varphi$.
We analyse $\varphi^{-1}$ using the inverse operator: $\varphi^{-1}=\Phi^{-1}.\text{id}$.
Since $\varphi\left(x\right)=x+cx^{\beta}\boldsymbol{\ell}^{m}$,
we have:

\noindent 
\[
\Phi\left(h\right)\left(x\right)=h\big(\varphi(x)\big)=h\left(x+cx^{\beta}\boldsymbol{\ell}^{m}\right)=h\left(x\right)+\sum_{i=1}^{\infty}c_{i}h^{\left(i\right)}(x)\, x^{i\beta}\boldsymbol{\ell}^{im},\ h\in\mathcal{L}.
\]
Hence, $\Phi=\text{Id}+\sum_{i=1}^{\infty}c_{i}x^{i\beta}\boldsymbol{\ell}^{im}\frac{\mathrm{d}^{i}}{\mathrm{d}x^{i}}=\text{Id}+P$.
It is easily seen that $P$ is a small operator, so that $\Phi^{-1}$
is well-defined by the convergent series $\Phi^{-1}=\sum_{k=0}^{\infty}\left(-1\right)^{k}P^{k}$
and $\varphi^{-1}\left(x\right)=x+\sum_{k=1}^{\infty}\left(-1\right)^{k}P^{k}\left(x\right)$.

We claim that: 
\[
\mathcal{S}\left(\varphi^{-1}\left(x\right)-x\right)\subseteq\mathbb{N}_{*}\left(\beta-1,m\right)+\left\{ 1\right\} \times\mathbb{N}.
\]
Indeed, $P\left(x\right)=c_{1}x^{\beta}\boldsymbol{\ell}^{m}$, and
we can write $\left(\beta,m\right)=\left(\beta-1,m\right)+\left\{ 1\right\} \times\mathbb{N}$.
Inductively, a consecutive action of $P$ leads to a series in terms:
\[
x^{i\beta}\boldsymbol{\ell}^{im}\frac{\mathrm{d}^{i}}{\mathrm{d}x^{i}}\left(x^{k\left(\beta-1\right)+1}\boldsymbol{\ell}^{km+j}\right),\hspace{1em}i\ge0,k\ge1,j\in\mathbb{N}.
\]
Hence, the elements of their support can be written 
\[
\begin{alignedat}{1}(i\beta & +k\left(\beta-1\right)+1-i,im+km+j+s)\\
 & =\left(\left(i+k\right)\left(\beta-1\right)+1,\left(i+k\right)m+j+s\right)\\
 & =\left(i+k\right)\left(\beta-1,m\right)+\left(1,j+s\right)\in\mathbb{N}_{*}\left(\beta-1,m\right)+\left\{ 1\right\} \times\mathbb{N},
\end{alignedat}
\]
which proves the claim. Since we have proven above that $(\beta-1,m)\in\mathcal{R}_{1}$,
it follows that 
\[
\mathcal{S}(\varphi^{-1}-\mathrm{id})\subseteq\mathcal{R}_{1}.
\]

\noindent 3. \textbf{Control of the support $\mathcal{S}\left(\varphi^{-1}\circ f\right)$.}
Let us write $f\left(x\right)=x+\varepsilon\left(x\right)$, $\text{ord}(\varepsilon)\succ(1,0)$,
and $\varphi^{-1}\left(x\right)=x+\sum b_{\mu s}x^{\mu}\boldsymbol{\ell}^{s}$.
Then 
\begin{align*}
\varphi^{-1}\left(f\left(x\right)\right) & =\varphi^{-1}\left(x+\varepsilon\left(x\right)\right)=\varphi^{-1}\left(x\right)+\sum_{k=1}^{\infty}\frac{1}{k!}\left(\varphi^{-1}\right)^{\left(k\right)}\left(x\right)\varepsilon^{k}\\
 & =\varphi^{-1}\left(x\right)+\varepsilon\left(x\right)+\left(\sum b_{\mu s}x^{\mu}\boldsymbol{\ell}^{s}\right)'\varepsilon\left(x\right)+\sum_{k=2}^{\infty}\frac{1}{k!}\left(\sum b_{\mu s}x^{\mu}\boldsymbol{\ell}^{s}\right)^{\left(k\right)}\varepsilon\left(x\right)^{k}.
\end{align*}
We already know that $\mathcal{S}\left(\varepsilon\right)\subseteq\mathcal{R}$.
Let us study the exponents of the series $\left(x^{\mu}\boldsymbol{\ell}^{s}\right)^{\left(k\right)}\varepsilon\left(x\right)^{k}$.
They are of the form: 
\begin{align*}
\left(\mu-k,s+\nu\right) & +\big(\left(\gamma_{i_{1}},r_{i_{1}}\big)+\cdots+\left(\gamma_{i_{k}},r_{i_{k}}\right)\right)=\left(\mu,s\right)+\left(\gamma_{i_{1}}-1,r_{i_{1}}\right)+\cdots+\left(\gamma_{i_{k}}-1,r_{i_{k}}\right)+\left(0,\nu\right)\\
 & =\left(\mu,s\right)+\left(\gamma_{i_{1}}-\alpha,r_{i_{1}}-p-1\right)+\cdots+\left(\gamma_{i_{k}}-\alpha,r_{i_{k}}-p-1\right)+k\left(\alpha-1,p\right)+\left(0,\nu+k\right),
\end{align*}
where $\nu\in\mathbb{N}$, $k\ge1$ and $\left(\gamma_{i_{1}},r_{i_{1}}\right),\ldots,\left(\gamma_{i_{k}},r_{i_{k}}\right)\in\mathcal{S}(\varepsilon)\subseteq\mathcal{R}$.
Each pair $\left(\gamma_{i_{j}}-\alpha,r_{i_{j}}-p-1\right)$ belongs
to $\mathcal{R}_{1}$. Recall from the previous step that $\mathcal{S}(\varphi^{-1}-\mathrm{id})\subseteq\mathcal{R}_{1}$.
Hence, these exponents can be written as 
\[
\left(\bar{\mu},\bar{s}\right)+k\left(\alpha-1,p\right)+\left(1,\nu+k\right),
\]
where $\left(\bar{\mu},\bar{s}\right)\in\mathcal{R}_{1}$. So they
belong to $\mathcal{R}$. Hence, we can write $\varphi^{-1}\left(f\left(x\right)\right)=x+g\left(x\right)$,
with $\mathcal{S}\left(g\right)\subseteq\mathcal{R}$.\\

\noindent 4. \textbf{Control of the support $\mathcal{S}\left(\varphi^{-1}\circ f\circ\varphi\right)$.}
We have 
\[
\varphi^{-1}\left(f\left(\varphi\left(x\right)\right)\right)=x+g\left(\varphi\left(x\right)\right)=x+g\left(x+cx^{\beta}\boldsymbol{\ell}^{m}\right),
\]
where $\mathcal{S}(g)\subseteq\mathcal{R}$ from the previous step.
The elements of the support of $\varphi^{-1}\circ f\circ\varphi$
can be written 
\begin{align*}
\left(\tau,l\right) & =\big(\mu+k(\beta+1),s+km+j\big)\\
 & =\left(\mu,s\right)+k\left(\beta-1,m\right)+\left(0,j\right),\quad\left(\mu,s\right)\in\mathcal{S}(g)\subseteq\mathcal{R},\, k\ge0,\, j\in\mathbb{N}.
\end{align*}
Since we have shown above that $\left(\beta-1,m\right)\in\mathcal{R}_{1}$,
so is $\left(\tau,l\right)\in\mathcal{R}$. 
\end{proof}

\subsection{The control of the orders of the normalizing elementary changes of
variables}

\label{sub:two}

Let $(\varphi_{\beta,m})$ be a transfinite sequence of elementary
changes of variables used to normalize $f$. Let $(\psi_{\beta,m})$
denote their partial compositions and $\varphi$ the limit of these
partial compositions (hence $\varphi$ is the change of variables
which normalizes $f$). We prove first that the supports of $(f_{\beta,m}-\mathrm{id})$,
$f_{\beta,m}:=\psi_{\beta,m}^{-1}\circ f\circ\psi_{\beta,m}$, are
all contained in the set $\mathcal{R}$ of the previous subsection.
This is based on a straightforward transfinite induction:
\begin{enumerate}[i), font=\textup, nolistsep, leftmargin=0.6cm]
\item We have already noticed that $\mathcal{S}\left(f-\mathrm{id}\right)$
is contained in $\mathcal{R}$.
\item The non-limit case follows directly by Lemma \ref{lem:appendix-main-lemma}:
if the support of $f$ is contained in $\mathcal{R}$, so is the support
of $\varphi_{\beta,m}^{-1}\circ f\circ\varphi_{\beta,m}$.
\item The limit case comes from the obvious classical fact in Hahn fields:
\emph{Consider a transfinite sequence $\left(g_{\mu}\right)_{\mu<\theta}$
of elements of a Hahn field, which admits a limit $g$ and whose supports
are contained in a common well-ordered set $W$. Then $\mathcal{S}\left(g\right)\subseteq W$.}
\end{enumerate}
By Lemma~\ref{lem:appendix-main-lemma} and Remark~\ref{rem:hom},
the supports of all elementary changes, $\mathcal{S}(\varphi_{\beta,m}-\mathrm{id})$,
are contained in $\mathcal{R}_{1}$. By an easy computation in the
non-limit case and using the classical result mentioned under $(iii)$
above in the limit case, we conclude that $\mathcal{S}(\psi_{\beta,m}-\mathrm{id}),\ \mathcal{S}(\varphi-\mathrm{id})\subseteq\mathcal{R}_{1}$.

\subsection{Finite type cases.}

\label{sub:three} The goal of this subsection is to show that if
a parabolic series $f$ is in addition of finite type, then so is
the normalizing change of variables $\varphi$ built in the proof
of Theorem A in Section \ref{sub:proof-precise-form-A}. We keep all
the notations as above.

Assume that $f$ is of finite type. We prove that in this case the
sets $\mathcal{R}$ and $\mathcal{R}_{1}$, as well as the support
of the composition $\mathcal{S}(\varphi)$, are of finite type. These
claims follow from the following easy result:
\begin{lem}
\label{lem:finitely_generated_set}Consider a subset $A$ of finite
type of $\mathbb{R}_{>0}\times\mathbb{Z}$. Let $\left(\alpha,p\right)\in\mathbb{R}_{>0}\times\mathbb{Z}$
and $\overline{A}:=\{\left(\beta,m\right)\in A\colon\left(\beta,m\right)\succeq\left(\alpha,p\right)\}$.
Then the set $\overline{A}-\left(\alpha,p\right)$ is also of finite
type.\end{lem}
\begin{proof}
Suppose $A$ is contained in the sub-semigroup $G$ of $\mathbb{R}_{\ge0}\times\mathbb{Z}$
generated by the elements $\left(\alpha_{1},p_{1}\right)$,\ldots{},$\left(\alpha_{k},p_{k}\right)$
of $\mathbb{R}_{>0}\times\mathbb{Z}$. For each $i=1,\ldots,k,$ denote
by $N_{i}$ the smallest positive integer such that $N_{i}\left(\alpha_{i},p_{i}\right)\succeq\left(\alpha,p\right)$.
Consider an element $\left(\gamma,r\right)\in A\subseteq G,\ (\gamma,r)\succeq\left(\alpha,p\right)$.
Then $\left(\gamma,r\right)=\sum_{i=1}^{k}n_{i}\left(\alpha_{i},p_{i}\right)$,
$n_{i}\in\mathbb{N}_{0}$. There exist only \emph{finitely many} elements
$(\gamma,r)\in A$ such that the respective $n_{i}$-s satisfy $n_{i}<N_{i}$,
for all $i=1,\ldots,k.$ On the other hand, if one of the $n_{i}$'s,
say $n_{1}$, is greater than or equal to $N_{1}$, we write 
\[
\left(\gamma,r\right)-\left(\alpha,p\right)=\left(n_{1}-N_{1}\right)\left(\alpha_{1},p_{1}\right)+n_{2}\left(\alpha_{2},p_{2}\right)+\cdots+n_{k}\left(\alpha_{k},p_{k}\right)+N_{1}\left(\alpha_{1},p_{1}\right)-(\alpha,p).
\]
This shows that $\overline{A}-\left(\alpha,p\right)$ is contained
in the sub-semigroup of $\mathbb{R}_{\ge0}\times\mathbb{Z}$ finitely
generated by the elements $\left(\alpha_{1},p_{1}\right)$,\ldots{},$\left(\alpha_{k},p_{k}\right)$,
$N_{i}(\alpha_{i},p_{i})-(\alpha,p)$, $i=1,\ldots,k$ (note that
$N_{i}(\alpha_{i},p_{i})-(\alpha,p)\succeq(0,0)$), and the elements
$(\gamma,r)-\left(\alpha,p\right)$, for the \emph{finitely many}
$(\gamma,r)\succeq(\alpha,p)$ for which the respective $n_{i}$-s
satisfy $n_{i}<N_{i}$, for all $i=1,\ldots,k.$ 
\end{proof}
If $f$ is of the finite type, we apply Lemma~\ref{lem:finitely_generated_set}
to the set $\overline{S}-\left(\alpha,p+1\right)$ from the previous
section. It follows that the sets $\mathcal{R}$ and $\mathcal{R}_{1}$
are also of the finite type. Since, by Sections~\ref{sub:one} and
\ref{sub:two}, we have that $\mathcal{S}(\varphi-\mathrm{id})\in\mathcal{R}_{1}$
and $\mathcal{S}(f_{0}-\mathrm{id})\in\mathcal{R}$, we deduce that
$\varphi$ and $f_{0}$ are of the finite type.

\address{$^{1}$ and $^{3}$ : Universit\'e de Bourgogne, D\'epartment de
Math\'ematiques, Institut de Math\'ematiques de Bourgogne, B.P.
47 870-21078-Dijon Cedex, France }

$^{2}$ and $^{4}$ : University of Zagreb, Department of Applied
Mathematics, Faculty of Electrical Engineering and Computing, Unska
3, 10000 Zagreb, Croatia
\end{document}